%% file: StochSQP.tex
\documentclass{siamart171218}

\RequirePackage{amsmath}
\RequirePackage{amsfonts}
\RequirePackage{dsfont}
\RequirePackage{amssymb}
\RequirePackage{hyperref}
\RequirePackage{multirow}
\RequirePackage{algorithm}
\RequirePackage[noend]{algpseudocode}
\RequirePackage{graphicx}
\RequirePackage{xcolor}
\allowdisplaybreaks

\usepackage{dpr_fec,ifthen}
\usepackage{epstopdf}
\usepackage{xcolor}
\usepackage[numbers, sort&compress ]{natbib}
\usepackage{todonotes}
\usepackage{enumitem}
\usepackage{booktabs}
\usepackage{arydshln}

\newtheorem{assumption}[theorem]{Assumption}
\newtheorem{remark}[theorem]{Remark}

\renewtheorem{definition}[theorem]{Definition}

\newtheorem{innercustomoracle}{Oracle}
\newenvironment{customoracle}[1]
  {\renewcommand\theinnercustomoracle{#1}\innercustomoracle}
  {\endinnercustomoracle} 

\newcommand{\SSSPQ}{\texttt{SS-SQP}}
\newcommand{\ASSPQ}{\texttt{AS-SQP}}

\newcommand{\change}[1]{{\color{black}#1}}

\newcommand{\dTrue}{d} 
\newcommand{\yTrue}{y} 
\newcommand\norm[1]{\lVert#1\rVert}

\newcommand{\taubartrial}{\bar\tau^{\rm trial}}
\newcommand{\tauTrue}{\tau}
\newcommand{\tautrialTrue}{\tau^{\rm trial}}
\newcommand{\Tau}{\mathcal{T}}

\newcommand\abs[1]{\left\lvert{#1}\right\rvert}

\def\sF{\mathcal{F}}

\title {A Sequential Quadratic Programming Method with High Probability Complexity Bounds for Nonlinear Equality Constrained Stochastic Optimization\footnotemark[1]}
\author{Albert~S.~Berahas\footnotemark[2]
   \and Miaolan~Xie\footnotemark[3]
   \and Baoyu~Zhou\footnotemark[4]}
\date  {\today}

\begin{document}

\maketitle

\renewcommand{\thefootnote}{\fnsymbol{footnote}}
\footnotetext[1]{This material is based upon work supported by the Office of Naval Research
under award number N00014-21-1-2532.}
\footnotetext[2]{Department of Industrial and Operations Engineering, University of Michigan, Ann Arbor, MI, USA; E-mails: \email{albertberahas@gmail.com}}
\footnotetext[3]{School of Industrial Engineering, Purdue University, 
West Lafayette, IN, USA; E-mail: \email{xie537@purdue.edu}}
\footnotetext[4]{School of Computing and Augmented Intelligence, Arizona State University, Tempe, AZ, USA; E-mail: \email{baoyu.zhou@asu.edu}}
\renewcommand{\thefootnote}{\arabic{footnote}}

\begin{abstract}
  \input{abstract}
\end{abstract}



\input{body}

\input{acknowledgements}


\small
\bibliographystyle{siamplain}
\bibliography{references}


\appendix

\input{appendix}

\end{document}

%% file: abstract.tex
A step-search sequential quadratic programming method is proposed for solving nonlinear equality constrained stochastic optimization problems. It is assumed that constraint function values and derivatives are available, but only stochastic approximations of the objective function and its associated derivatives can be computed via inexact probabilistic zeroth- and first-order oracles. Under reasonable assumptions, a high probability bound on the number of iterations that the algorithm requires to reach a first-order $\varepsilon$-stationary iterate is derived, where $\varepsilon$ is lower bounded by a positive quantity dictated by the noise level of the inexact probabilistic zeroth- and first-order oracles. Numerical results on standard nonlinear optimization test problems illustrate the advantages and limitations of our proposed method.

%% file: body.tex
\numberwithin{equation}{section}
\numberwithin{theorem}{section}

\section{Introduction}
We propose a step-search\footnote{We use the term \emph{step-search} methods, coined in~\cite{jin2022high} to differentiate with line-search methods. Step-search methods are similar to line-search methods, but the search (step) direction can change during the back-tracking procedure.} sequential quadratic programming (SQP) algorithm for solving nonlinear equality constrained stochastic optimization problems of the form
\begin{equation}\label{eq.prob}
    \min_{x\in\R{n}}\  f(x) \quad \text{s.t.}\ c(x) = 0,
\end{equation}
where $f:\R{n}\to\R{}$ and $c:\R{n}\to\R{m}$ are both continuously differentiable. We consider the setting in which exact function and derivative information of the objective function is unavailable, instead, only $\bar{f}(x;\Xi^0(x))$ and $\bar{g}(x;\Xi^1(x))$, the random estimates of the objective function $f(x)$ and its first-order derivative $\nabla f(x)$, 
are available via inexact probabilistic oracles, where $\Xi^0(x)$ and $\Xi^1(x)$ (with probability space $(\Omega,\mathcal{F}_{\Omega},P)$) denote the underlying randomness in the objective function and gradient estimates, respectively. 
On the other hand, the constraint function value $c(x)$ and its Jacobian $\nabla c(x)^T$  are assumed to be available. 
Such deterministically constrained stochastic optimization problems arise in multiple science and engineering applications, including but not limited to computer vision \cite{RoyMhamHara18}, multi-stage optimization \cite{ShapDentRusz21}, natural language processing \cite{NandPathSing19}, network optimization \cite{Bert98}, and PDE-constrained optimization~\cite{ReesDollWath10}. 

The majority of the methods proposed for solving deterministically equality constrained 
stochastic optimization problems follow either projection or penalty approaches. The former type of methods (e.g., stochastic projection methods \cite{HazaLuo16,KushClar12,Lan20,LuFreu21}) require that the feasible region satisfies strict conditions, to ensure well-definedness, that are not satisfied by general nonlinear functions and thus are not readily applicable. In contrast, stochastic penalty methods \cite{ChenTungVeduMori18,RaviDinhLokhSing19}, do not impose such conditions 
on the feasible region. These methods transform constrained problems into unconstrained problems via constraint penalization in the objective function, and apply stochastic algorithms to solve the transformed unconstrained problems. Stochastic penalty methods are easy to implement and well-studied, however, the empirical performance of such methods is sensitive to parameter choices and ill-conditioning, and is usually inferior to paradigms that treat constraints as constraints.

Recently, a class of stochastic SQP methods has been developed for solving \eqref{eq.prob}. These methods outperform stochastic penalty methods empirically and have convergence guarantees in expectation~\cite{BeraCurtRobiZhou21,NaAnitKola22}. In~\cite{BeraCurtRobiZhou21}, the authors propose an objective-function-free stochastic SQP method with adaptive step sizes for the fully stochastic regime. In contrast, in~\cite{NaAnitKola22}, the authors propose a stochastic step-search (referred to as line-search in the paper~\cite{NaAnitKola22}) SQP method for the setting in which the errors in the function and derivative approximations can be diminished. We note that several algorithmic choices in the two papers~\cite{BeraCurtRobiZhou21,NaAnitKola22}, e.g., merit functions and merit parameters, are different. 
Several extensions have been proposed \cite{BeraCurtONeiRobi21,CurtRobiZhou21,BeraBollZhou22,BeraShiYiZhou22,na2021inequality,OztoByrdNoce21}, however, very few of these works (or others in the literature) derive worst-case iteration complexity (or sample complexity) because of the difficulties that arise in the constrained setting and due to the stochasticity. Notable exceptions are, \cite{CurtONeiRobi21} where the authors  provide convergence rates (and complexity guarantees) for the algorithm proposed in \cite{BeraCurtRobiZhou21}, and \cite{BeraBollZhou22,NaMaho22} that provide complexity bounds for variants of the stochastic SQP methods under additional assumptions and in the setting in which the errors can be diminished. We note that, with the exception of~\cite{OztoByrdNoce21}, all methods mentioned above assume access to {\sl unbiased} estimates of the gradients (and function values where necessary).

For all aforementioned methods, the most vital ingredient is the quality and reliability of the random estimates of the objective function and its derivatives. In our setting, neither the objective function nor its derivatives are assumed to be directly accessible, only stochastic approximations of them are accessible to the algorithm in the form of inexact probabilistic zeroth-order and first-order oracles (precise definitions will be introduced in Section~\ref{sec.oracle}). Such oracles have been proposed and utilized in several works; e.g., \cite{jin2022high,cao2022first,GratRoyeViceZhan15,BandScheVice14, scheinberg2022stochastic, menickelly2023stochastic}. Moreover, these probabilistic oracles and their variants  have been proposed for direct-search methods \cite{GratRoyeViceZhan15,RobeRoye22}, trust-region methods \cite{BandScheVice14,BlanCartMeniSche19,ChenMeniSche18,GratRoyeViceZhan18}, and step-search methods \cite{CartSche18,PaquSche20,NaAnitKola22,bellavia2022linesearch}. We note that only \cite{NaAnitKola22} considers the setting with (equality) constraints, but iteration complexity (or sample complexity) results are not provided.

\subsection{Contributions}
In this paper, we design, analyze, and implement a step-search SQP (\SSSPQ) method for solving nonlinear equality constrained stochastic optimization problems where exact constraint function values and derivatives are available, but only stochastic approximations of the objective function and its associated derivatives can be computed. These stochastic approximations are computed via inexact probabilistic zeroth- and first-order oracles, which are 
similar to those in~\cite{jin2022high}, with parameters controlling the accuracy and reliability of the approximations, and allowing for biased approximations. 
Our proposed algorithm is inspired by state-of-the-art line-search SQP methods~\cite{byrd2008inexact} in conjunction with the recent stochastic adaptive step-search framework developed in \cite{jin2022high} for the unconstrained stochastic setting.
At every iteration, the algorithm constructs a model of the reduction in the merit function that serves the dual purpose of a measure of sufficient progress (part of the step size computation) and a proxy for convergence. 
To mitigate the challenges that arise due to the noise in the objective function evaluations, our step-search method employs a relaxed sufficient decrease condition similar to that proposed in~\cite{BeraByrdNoce19}. 
Under reasonable assumptions, we provide a high probability iteration $\varepsilon$-complexity bound for the proposed algorithm. 
Specifically, we prove that with overwhelmingly high probability, our algorithm generates a first-order $\varepsilon$-stationary iterate in $\mathcal{O}(\varepsilon^{-2})$ iterations, where $\varepsilon$ is bounded away from zero and its lower bound is dictated by the noise and bias in the zeroth- and first-order oracles. 
When exact objective function and gradient values are computable, the $\varepsilon$-complexity bound matches that of the deterministic algorithm  in~\cite{CurtONeiRobi21}. 
In~\cite{CurtONeiRobi21}, the authors also provide non-asymptotic convergence guarantees for the stochastic algorithm proposed in~\cite{BeraCurtRobiZhou21} (for which only asymptotic convergence was proven).
There are two key differences between our contributions and \cite{CurtONeiRobi21}: $(i)$ our algorithm requires access to estimates of the objective function whereas the method in \cite{CurtONeiRobi21} is objective-function-free; and $(ii)$ our first-order oracle provides estimates with sufficient accuracy only with some probability and can provide arbitrarily bad estimates otherwise. 
Finally, experiments on standard nonlinear equality constrained test problems~\cite{gould2015cutest} illustrate the efficiency and efficacy of our proposed algorithm.

\subsection{Notation} Let $\R{}$ denote the set of real numbers, $\R{n}$ denote the set of $n$-dimensional real vectors, $\R{m\times n}$ denote the set of $m$-by-$n$-dimensional real matrices, $\N{}$ denote the set of natural numbers, and $\mathbb{S}^n$ denote the set of $n$-by-$n$-dimensional real symmetric matrices. For any $a\in\R{}$, let $\R{}_{>a}$ ($\R{}_{\geq a}$) denote the set of real numbers strictly larger than (larger than or equal to) $a$. We use $\|\cdot\|$ to denote the $\ell_2$-norm. We use $k\in\N{}$ as the iteration counter of the algorithm, and for brevity, we use a subscript $k$ for denoting information at the $k$th iterate, e.g., $f_k := f(x_k)$. 
We use $\bar{f}(x;\xi^0(x))$ and $\bar{g}(x;\xi^1(x))$ to denote realizations of $\bar{f}(x;\Xi^0(x))$ and $\bar{g}(x;\Xi^1(x))$, respectively; see Section~\ref{sec.oracle}.

\subsection{Organization} The rest of this paper is organized as follows. The algorithmic framework is introduced in Section~\ref{sec.alg}. The analysis of the algorithm is established in Section~\ref{sec.analysis}. 
We report numerical results in  Section~\ref{sec.numerical}. 
Concluding remarks and future research directions are given in  Section~\ref{sec.conclusion}.

\section{Algorithm}\label{sec.alg}

To solve \eqref{eq.prob}, we design an iterative stochastic algorithm based on the SQP paradigm. Every realization of the algorithm generates the following sequences: 
$(i)$ a primal iterate sequence $\{x_k\}$, $(ii)$ a primal trial iterate sequence $\{x_k^+\}$, $(iii)$ a primal search direction sequence $\{\bar{d}_k\}$, $(iv)$ a dual iterate sequence $\{\bar{y}_k\}$, $(v)$ a step size sequence $\{\alpha_k\}$, $(vi)$ a merit parameter sequence $\{\bar{\tau}_k\}$, and, $(vii)$ a trial merit parameter sequence $\{\taubartrial_k\}$.
These aforementioned sequences are realizations of some stochastic process. For example, the primal iterate sequence $\{x_k\}$ is a realization of a stochastic process $\{X_k\}\subset\R{n}$, while the primal trial iterate sequence $\{x_k^+\}$ is a realization of another stochastic process $\{X_k^+\}\subset\R{n}$.
We discuss each of these sequences 
below. We make the following assumption throughout the remainder of this paper. 

\begin{assumption}\label{ass.prob}
Let $\Xcal\subseteq\R{n}$ be an open convex set containing 
the primal iterate sequence $\{X_k\}$ and the primal trial iterate sequence $\{X_k^+\}$.
The objective function $f:\R{n}\to\R{}$ is continuously differentiable and bounded below over $\Xcal$.  The objective gradient function $\nabla f:\R{n}\to\R{n}$ is $L$-Lipschitz continuous and bounded over $\Xcal$.  The constraint function $c:\R{n}\to\R{m}$ (where $m\leq n$) is continuously differentiable and bounded over $\Xcal$, and each gradient $\nabla c_i:\R{n}\to\R{n}$ is $\gamma_i$-Lipschitz continuous and bounded over $\Xcal$ for all $i\in\{1,\ldots,m\}$.  The singular values of $J := \nabla c^T$ are bounded away from zero over $\Xcal$.
\end{assumption}
\sloppy Assumption~\ref{ass.prob} is a standard assumption in the deterministic constrained optimization literature~\cite{byrd2008inexact,NoceWrig06,wachter2005line}. Assumption~\ref{ass.prob} is implicitly about the behavior of the algorithm. In essence, the algorithm converges to something meaningful if it does not happen to generate an unbounded sequence of iterates. This logic and limitation is true for the deterministic case \cite[Assumption G]{wachter2005line} and \cite[Assumption 4.1]{byrd2008inexact}, and is also required in our analysis. 
Under Assumption~\ref{ass.prob}, there exist constants $\{\kappa_g,\kappa_c,\kappa_J,\kappa_{\sigma}\}\subset\R{}_{>0}$ and $f_{\inf}\in\R{}$ such that for all $k\in\N{}$, 
\begin{equation*}
    f_{\inf} \leq f_k,\ \|\nabla f_k\|\leq \kappa_g,\ \|c_k\|_1\leq \kappa_c,\ \|J_k\|\leq \kappa_J, \text{ and }\|(J_kJ_k^T)^{-1}\|\leq \kappa_{\sigma}.
\end{equation*}
We should note that by Assumption~\ref{ass.prob}, linear independence constraint qualifications (LICQ) hold. 
Moreover, under Assumption~\ref{ass.prob}, for all $x\in\R{n}$, $d\in\R{n}$ and $\alpha\in\R{}_{\geq 0}$ it follows that
\begin{equation}\label{eq.lipschitz_continuity}
\begin{aligned}
    f(x+\alpha d) &\leq f(x) + \alpha \nabla f(x)^Td + \tfrac{L}{2}\alpha^2\|d\|^2 \\
    \text{and} \ \ \|c(x+\alpha d)\|_1 &\leq \|c(x) + \alpha \nabla c(x)^Td\|_1 + \tfrac{\Gamma}{2}\alpha^2\|d\|^2, \ \ \text{where $\Gamma = \sum_{i=1}^m\gamma_i$.}
\end{aligned}
\end{equation}

In this paper, we are particularly interested in finding some primal-dual iterate $(x,y)\in\R{n}\times\R{m}$ that satisfies the first-order stationarity conditions of \eqref{eq.prob}. To this end, let $\mathcal{L}: \mathbb{R}^n \times \mathbb{R}^m \rightarrow \mathbb{R}$ be the Lagrangian of~\eqref{eq.prob}, defined as
\begin{equation}\label{eg.lagrangian}
    \mathcal{L}(x,y) = f(x) + y^Tc(x),
\end{equation}
where $y \in \mathbb{R}^m$ are the dual variables. The first-order
stationarity conditions for \eqref{eq.prob}, which are necessary by Assumption~\ref{ass.prob} (due to the inclusion of the LICQ), are
\begin{equation}\label{eq.firstorder}
    0 = \begin{bmatrix}
        \nabla_x \mathcal{L}(x,y) \\
        \nabla_y \mathcal{L}(x,y)
    \end{bmatrix} = \begin{bmatrix}
        \nabla f(x) + \nabla c(x)y \\
        c(x)
    \end{bmatrix}.
\end{equation}

In the remainder of this section we introduce the key algorithmic components: the merit function and its associated models, the search direction computation and merit parameter updating mechanism, and the inexact probabilistic zeroth- and first-order oracles. The main algorithm is Algorithm~\ref{alg.sqp_line_search_practical}.

\subsection{Merit function}
The merit function $\phi:\R{n}\times\R{}_{>0} \to \R{}$ is defined as
\begin{equation}\label{eq.merit_function}
    \phi(x,\tau) := \tau f(x) + \|c(x)\|_1,
\end{equation}
where $\tau\in\R{}_{>0}$, the merit parameter, balances the objective function and the constraint violation. Given the gradient (approximation) $g \in \R{n}$ and a search direction $d\in\R{n}$, the model of merit function  $l:\R{n}\times\R{}_{>0}\times\R{n}\times\R{n}\to\R{}$ is defined as
\begin{equation*}
    l(x,\tau,g,d) := \tau(f(x) + g^Td) + \|c(x) + \nabla c(x)^Td\|_1.
\end{equation*}
Given a search direction $d \in \R{n}$ that satisfies linearized feasibility, i.e., $c(x) + \nabla c(x)^Td = 0$, the reduction in the model of the merit function $\Delta l:\R{n}\times\R{}_{>0}\times\R{n}\times\R{n}\to\R{}$ is defined as
\begin{equation}\label{eq.model_reduction}
\begin{aligned}
    \Delta l(x,\tau,g,d) :=&\  l(x,\tau,g,0) - l(x,\tau,g,d) 
    = -\tau g^Td + \|c(x)\|_1.
\end{aligned}
\end{equation}
We use the reduction in the model of the merit function \eqref{eq.model_reduction} to monitor the progress made by our proposed algorithm. 
We discuss this in more detail in Section~\ref{sec.algorithm_pre}.

\subsection{Algorithmic components}\label{sec.algorithm_pre}
We now establish how to: $(i)$ compute the primal search direction sequence $\{\bar{d}_k\}$, $(ii)$ update the merit parameter sequence $\{\bar\tau_k\}$, and $(iii)$ update the primal iterate sequence $\{x_k\}$, in any realization 
of Algorithm~\ref{alg.sqp_line_search_practical}. These sequences depend on the approximation of the gradient of the objective function 
sequence $\{\bar{g}(x_k;\Xi^1(x_k))\}$. 
For any $x_k\in\R{n}$, let
$\bar{g}(x_k;\xi^1(x_k))$ denote a
realization of $\bar{g}(x_k;\Xi^1(x_k))$. 
Since $\{x_k\}$ is a realization of the stochastic process $\{X_k\}\subset\R{n}$, we define $G_k = \bar{g}(X_k;\Xi^1(X_k))$ with realizations $\bar{g}(x_k;\xi^1(x_k))$.
To simplify the notation, in this subsection we drop the dependence on the randomness, e.g., $\bar{g}_k = \bar{g}(x_k;\xi^1(x_k))$.

For all $k\in\N{}$ in any realization 
of Algorithm~\ref{alg.sqp_line_search_practical}, 
the primal search direction $\bar{d}_k\in\R{n}$ and the dual variable $\bar{y}_k\in\R{m}$ are computed by solving the linear system of equations
\begin{equation}\label{eq.SQP}
    \begin{bmatrix} H_k & J_k^T \\ J_k & 0 \end{bmatrix}\begin{bmatrix} \bar{d}_k \\ \bar{y}_k \end{bmatrix} = - \begin{bmatrix} \bar{g}_k \\ c_k \end{bmatrix},
\end{equation}
where $\{H_k\}$ satisfies the following assumption. 
\begin{assumption}\label{ass.H}
For all $k\in\N{}$, $H_k\in\Smbb^n$ is chosen independently from $G_k$, where $\{G_k\}$ is a stochastic process with realizations $\{\bar{g}_k\}$ (a sequence of gradient estimates). Moreover, there exist constants $\{\kappa_H,\zeta\}\subset\R{}_{>0}$ such that for all $k\in\N{}$ in any realization 
of Algorithm~\ref{alg.sqp_line_search_practical}, $\|H_k\| \leq \kappa_H$ and $u^TH_ku \geq \zeta\|u\|^2$ for any $u\in\Null(J_k)$.
\end{assumption}
It is well known that under Assumptions~\ref{ass.prob} and~\ref{ass.H}, there is a unique solution $(\bar{d}_k,\bar{y}_k)$ to \eqref{eq.SQP}, and, thus, the vectors $\bar{d}_k\in\R{n}$ and $\bar{y}_k\in\R{m}$ are well-defined~\cite{NoceWrig06}.

Next, we present the merit parameter updating mechanism. Given constants $\{\epsilon_{\tau},\sigma\}\subset (0,1)$, for all $k\in\N{}$ in any realization 
of Algorithm~\ref{alg.sqp_line_search_practical}, we compute $\bar\tau_{k}$ via
\begin{equation}\label{eq.tau_update}
    \bar\tau_{k} \leftarrow \bcases \bar\tau_{k-1} \ &\text{if }\bar\tau_{k-1} \leq \taubartrial_k; \\
    \min\left\{(1-\epsilon_{\tau})\bar\tau_{k-1},\taubartrial_k\right\} &\text{otherwise,}\ecases
\end{equation}
where
\begin{equation}\label{eq.tau_trial}
    \taubartrial_k \leftarrow \bcases \infty \ &\text{if }\bar{g}_k^T\bar{d}_k + \max\left\{\bar{d}_k^TH_k\bar{d}_k,0\right\} \leq 0; \\
    \tfrac{(1-\sigma)\|c_k\|_1}{\bar{g}_k^T\bar{d}_k + \max\left\{\bar{d}_k^TH_k\bar{d}_k,0\right\}} &\text{otherwise.} \ecases
\end{equation}
The merit parameter updating mechanism ensures that the sequence of merit parameter values is non-increasing. Moreover, the updating mechanism is designed to ensure that the reduction in the model of the merit function is sufficiently positive. We make these claims concrete in the following lemma.
\begin{lemma}\label{lem.model_reduction}
Suppose Assumptions~\ref{ass.prob} and~\ref{ass.H} hold. Following the merit parameter updating mechanism described in~\eqref{eq.tau_update}--\eqref{eq.tau_trial}, it follows that for all $k\in\mathbb{N}$ in any realization 
of Algorithm~\ref{alg.sqp_line_search_practical}
\begin{equation}\label{eq.model_reduction_cond}
    \Delta l(x_k,\bar\tau_k,\bar{g}_k,\bar{d}_k) \geq \bar\tau_{k} \max\left\{\bar{d}_k^TH_k\bar{d}_k,0\right\} + \sigma\|c_k\|_1.
\end{equation}
Furthermore, if $\bar\tau_k \neq \bar\tau_{k-1}$, then $0 < \bar\tau_{k} \leq (1-\epsilon_{\tau})\bar\tau_{k-1}$.
\end{lemma}
\begin{proof}
By \eqref{eq.tau_update}, we have $\bar\tau_k \leq \taubartrial_k$. Moreover, by  \eqref{eq.model_reduction}, \eqref{eq.tau_update} and \eqref{eq.tau_trial}, it follows that  \eqref{eq.model_reduction_cond} is satisfied for all $k \in \N{}$. By \eqref{eq.tau_update}, if $\bar\tau_k\neq \bar\tau_{k-1}$, then $\bar\tau_k = \min\left\{(1-\epsilon_{\tau})\bar\tau_{k-1},\taubartrial_k\right\} \leq (1-\epsilon_{\tau})\bar\tau_{k-1}$. Moreover, when $c_k = 0$, it follows from Assumption~\ref{ass.H}, \eqref{eq.SQP} and \eqref{eq.tau_trial} that $\bar{d}_k\in\Null(J_k)$ and $\bar{g}_k^T\bar{d}_k + \max\{\bar{d}_k^TH_k\bar{d}_k,0\} = \bar{g}_k^T\bar{d}_k + \bar{d}_k^TH_k\bar{d}_k = c_k^T\bar{y}_k = 0$, which implies $\taubartrial_k = \infty$. Therefore, we have $\taubartrial_k > 0$ for all $k\in\N{}$. Finally, by $\bar\tau_{-1}\in\R{}_{>0}$ and \eqref{eq.tau_update}, we have $\bar\tau_k > 0$ for all $k\in\N{}$.
\end{proof}
\noindent
We note that in the deterministic setting, the reduction in the model of the merit function is zero only at iterates that satisfy~\eqref{eq.firstorder}.

At each iteration $k\in\N{}$ in any realization 
of Algorithm~\ref{alg.sqp_line_search_practical}, after
updating the merit parameter $\bar\tau_k$, we evaluate $\Delta l(x_k,\bar\tau_k,\bar{g}_k,\bar{d}_k)$,  the stochastic model reduction of
the merit function, and use it to check for sufficient progress. Specifically, given a step size $\alpha_k$, we compute a candidate iterate $x_k^+ := x_k + \alpha_k \bar{d}_k$ and  check whether sufficient progress can be made via the modified sufficient decrease condition 
\begin{equation}\label{eq.line_serach_cond_old_old}
    \bar\phi(x_k^+,\bar\tau_{k}; \xi^0(x_k^+)) \leq \bar\phi(x_k,\bar\tau_{k};\xi^0(x_k)) - \alpha_k\theta\Delta l(x_k,\bar\tau_{k},\bar{g}_k,\bar{d}_k)  + 2\bar \tau_k\epsilon_f,
\end{equation}
where $\bar\phi(x_k^+,\bar\tau_{k}; \xi^0(x_k^+))$ and $\bar\phi(x_k,\bar\tau_{k};\xi^0(x_k))$ are merit function estimates, $\theta \in (0,1)$ is a user-defined parameter and $\epsilon_f$ is an upper bound on the expected noise in the objective function approximations. 
We note that $\bar\phi(x_k^+,\bar\tau_{k}; \xi^0(x_k^+))$ and  $\bar\phi(x_k,\bar\tau_{k};\xi^0(x_k))$ are realizations of the zeroth-order oracle described in detail in Section~\ref{sec.oracle}. 
The positive term on the right-hand-side allows for a relaxation in the sufficient decrease condition, i.e., the merit function may increase after a step, and serves to correct for the noise in the merit function approximations.
If \eqref{eq.line_serach_cond_old_old} is satisfied, we accept the candidate point $x_k^+$ by setting $x_{k+1}\leftarrow x_k^+$, and potentially increase the step size for the next iteration, i.e.,  $\alpha_{k+1}\geq \alpha_k$. If \eqref{eq.line_serach_cond_old_old} is not satisfied, the algorithm does not accept the candidate iterate, instead, it sets $x_{k+1}\leftarrow x_k$ and shrinks the step size for the next iteration, i.e., $\alpha_{k+1} < \alpha_k$. This step update rule is the centerpiece of our step-search method, and is 
fundamentally different from traditional line-search strategies; see \cite{CartSche18,BeraCaoSche21,jin2022high}. Contrary to line-search methods, which compute a search direction and then look for a step size along that direction, in our approach the search direction changes in every iteration (even when a step is not taken).

We conclude this section by drawing a few parallels to the unconstrained setting. First, in the unconstrained setting (with $H_k = I$), the quantity $\Delta l(x_k,\bar\tau_k,\bar{g}_k,\bar{d}_k)$ reduces to $\|\bar{g}_k\|^2$, which provides a sufficient descent measure and is an approximate first-order stationarity measure. In the constrained setting, the reduction in the model of the merit function will play a similar role. Second, in the unconstrained setting, \eqref{eq.line_serach_cond_old_old} recovers the sufficient decrease condition used by a class of noisy line-/step-search unconstrained optimization algorithms; see e.g., \cite{BeraByrdNoce19,BeraCaoSche21,jin2022high}.

\subsection{Probabilistic oracles}\label{sec.oracle}

In many real-world applications exact objective function and derivative information cannot be readily computed. Instead, in lieu of these quantities, approximations are available via inexact probabilistic zeroth- and first-order oracles. These oracles produce approximations of different accuracy and reliability, and are formally introduced below.

\begin{customoracle}{0}[\textbf{Probabilistic zeroth-order oracle}]\label{oracle.zero} Given $x\in\R{n}$, the oracle computes $\bar{f}(x; \xi^0(x))$, a realization of $\bar{f}(x; \Xi^0(x))$, which is a (random) estimate of the objective function value $f(x)$, where $\Xi^0(x)$ denotes the underlying randomness (may depend on $x$) with associated probability space $\left(\Omega, \mathcal{F}_{\Omega}, P\right)$. Let  $\change{E(x;\Xi^0(x))}:=|\bar{f}(x; \Xi^0(x))-f(x)|$. For any $x\in\R{n}$, $\change{E(x;\Xi^0(x))}$ is a ``one-sided" sub-exponential random variable with parameters $\{\nu,b\}\subset\R{}_{\geq0}$, 
whose mean is bounded by some constant $\epsilon_{f}\in\R{}_{\geq 0}$. 
Specifically,
for all $x\in\R{n}$ and $\lambda\in [0,1/b]$, 
\begin{equation}\label{eq:zero_order}
\begin{aligned}
	&\Embb_{\Xi^0(x)}\left[\change{E(x;\Xi^0(x))}\right] \leq \epsilon_f  \\
	\text{ and } \; &\Embb_{\Xi^0(x)}\left[\exp( \lambda(\change{E(x;\Xi^0(x))} - \change{\Embb_{\Xi^0(x)}\left[E(x;\Xi^0(x))\right]}) )\right] \leq \exp\left(\tfrac{\lambda^2\nu^2}{2}\right).
\end{aligned}
\end{equation} 
The stochastic approximation of the merit function value is defined as   $\bar{\phi}(x, \tau; \xi^0(x))= \tau \bar{f}(x; \xi^0(x)) + \norm{c(x)}_1$.
\end{customoracle}

\begin{customoracle}{1}[\textbf{Probabilistic first-order oracle}]\label{oracle.first}
Given $x\in\R{n}$ and $\alpha\in\R{}_{>0}$, the oracle computes $\bar{g}(x;\xi^1(x))$, a realization of $\bar{g}(x; \Xi^1(x))$, which is a (random) estimate of the gradient of the objective function 
$\nabla f(x)$, such that 
\begin{equation}\label{eq:first_order}
\begin{aligned}
	\Pmbb_{\Xi^1(x)} \Bigg[&\|\bar{g}(x;\Xi^1(x)) - \nabla f(x)\| \leq  \\
	& \qquad \max\left\{ \epsilon_{g}, \kappa_{\mathrm{FO}} \alpha \sqrt{\Delta l(x,\bar\tau(x;\Xi^1(x)),\bar{g}(x;\Xi^1(x)),\bar{d}(x;\Xi^1(x)))} \right\}\Bigg]\geq 1-\delta,
\end{aligned}
\end{equation}
where $\Xi^1(x)$ denotes the underlying randomness (may depend on $x$) with associated probability space $(\Omega, \mathcal{F}_{\Omega}, P)$, $(1-\delta) \in (\tfrac 12,1]$ is the probability that the oracle produces a gradient estimate that is ``sufficiently accurate'' (the reliability of the oracle) and $\{\epsilon_{g},\kappa_{\mathrm{FO}}\} \subset \R{}_{\geq 0}$ 
are constants intrinsic to the oracle (the precision of the oracle).
\end{customoracle}

In the rest of the paper, to simplify notation we drop the dependence on $x$ in \change{$\Xi^0(x)$ (\change{resp., }$\xi^0(x)$) and $\Xi^1(x)$ (\change{resp., }$\xi^1(x)$)}. Moreover, we use \change{$(\Xi_k^0,\Xi_k^+,\Xi_k^1)$ and $(\xi_k^0,\xi_k^+,\xi_k^1)$ to represent $(\Xi^0(x_k),\Xi^0(x_k^+),\Xi^1(x_k))$ and $(\xi^0(x_k),\xi^0(x_k^+),\xi^1(x_k))$, respectively.} 

\begin{remark} We make a few remarks about Oracles~\ref{oracle.zero} and \ref{oracle.first}:
\begin{itemize}[wide]
    \item Oracles~\ref{oracle.zero} and \ref{oracle.first} are similar to those defined in \cite{cao2022first,jin2022high}. For a full discussion and examples of the oracles, we refer interested readers to \cite[Section~5]{jin2022high}.
    \item Oracle~\ref{oracle.first} generalizes the ones defined in~\cite{cao2022first,jin2022high} to the equality constrained setting. 
    Indeed, the right-hand-side of Oracle~\ref{oracle.first} reduces to  $\max\left\{ \epsilon_g, \kappa_{\mathrm{FO}}\alpha\norm{\bar{g}(x;\Xi^1)} \right\}$ in the unconstrained setting, and is precisely what is used in~\cite{cao2022first,jin2022high}.
    \item The presence of $\epsilon_g\in\mathbb{R}_{\geq 0}$ in the max term in Oracle~\ref{oracle.first} allows the gradient approximations to be biased; the magnitude of the bias is proportional to $\epsilon_g$.
\end{itemize}
\end{remark}


\subsection{Algorithmic framework}

We are ready to introduce our stochastic step-search SQP  method (\SSSPQ) in Algorithm~\ref{alg.sqp_line_search_practical}. 

\balgorithm[ht]
  \caption{Adaptive Step-Search SQP (\SSSPQ{})}
  \label{alg.sqp_line_search_practical}
  \balgorithmic[1]
    \Require initial iterate $x_0\in\R{n}$; initial merit parameter $\bar\tau_{-1}\in\R{}_{>0}$; maximum step size $\alpha_{\max}\in (0,1]$; initial step size $\alpha_0 \in (0,\alpha_{\max}]$; parameter $\epsilon_f \in\R{}_{\geq 0}$ 
    of the zeroth-order oracle (Oracle~\ref{oracle.zero}); and other parameters $\{\gamma,\theta,\sigma,\epsilon_{\tau}\}\subset (0,1)$
	\For{\textbf{all} $k \in \N{}$}
	\State Generate $\bar{g}_k = \bar{g}(x_k;\xi_k^1)$ via Oracle~\ref{oracle.first} with $\alpha = \alpha_k$, $\bar{d}_k = \bar{d}(x_k;\xi_k^1)$ as in \eqref{eq.SQP}, and $\bar\tau_{k} = \bar\tau(x_k;\xi_k^1)$ as in \eqref{eq.tau_update}--\eqref{eq.tau_trial} \label{line.gradient_generation}
	\State Let $x_k^+ = x_k + \alpha_k\bar{d}_k$, and generate $\bar\phi(x_k,\bar\tau_{k};\xi_k^0)$ and $\bar\phi(x_k^+,\bar\tau_{k};\xi_k^+)$ via Oracle~\ref{oracle.zero}\label{line.trial}
	\If{\eqref{eq.line_serach_cond_old_old} holds} 
	\State Set $x_{k+1}\leftarrow x_k^+$ and $\alpha_{k+1}\leftarrow \min\{ \alpha_{\max}, \gamma^{-1} \alpha_k \}$ 
	\Else
	\State Set $x_{k+1}\leftarrow x_k$ and $\alpha_{k+1} \leftarrow \gamma\alpha_k$ 
	\EndIf
	\EndFor
  \ealgorithmic
\ealgorithm

\begin{remark}
We make the following remarks about \SSSPQ{}:
\begin{itemize}[wide]
	\item (Step-search) Algorithm~\ref{alg.sqp_line_search_practical} is a step-search algorithm, whose main difference from traditional line-search methods is that only a single trial iterate is tested at every iteration. That is, if \eqref{eq.line_serach_cond_old_old} is not satisfied, the step size is reduced and a new search direction and candidate iterate are  computed in the next iteration. This strategy has been employed in other papers; e.g., see~\cite{CartSche18,BeraCaoSche21,jin2022high,NaAnitKola22}. We should note that at every iteration, even if the iterate does not change, our algorithm requires new objective function and gradient estimates in the next iteration.
	\item (Modified sufficient decrease condition \eqref{eq.line_serach_cond_old_old}) The $2\bar \tau_k\epsilon_f$ term on the right-hand-side of \eqref{eq.line_serach_cond_old_old} is a correction term added to compensate for the inexactness of the probabilistic zeroth-order oracle (Oracle~\ref{oracle.zero}). 
	This correction provides a relaxation to the sufficient decrease requirement. In contrast to traditional sufficient decrease conditions, the modified condition~\eqref{eq.line_serach_cond_old_old} allows for a relaxation that is proportional to the noise level of Oracle~\ref{oracle.zero}.
	\item (Objective function evaluations; Line~\ref{line.trial}) The randomness associated with the evaluation of the objective function value at the candidate iterate $x_k^+$ (Line~\ref{line.trial}) is not the same as that of the evaluation at the current point $x_k$. That is, the evaluation of the objective function at the candidate iterate $x_k^+$ is independent of the evaluation of the objective function at the current iterate $x_k$. We further note that, in fact, one only needs the noise in the difference of the function estimates, i.e., $\abs{\bar f(x_k,\Xi^0_k) - \bar f(x_k^+,\Xi_k^+) -(f(x_k) - f(x_k^+)) }$, to be sub-exponential, since this is the fundamental quantity that needs to be controlled in the analysis.
    Moreover, we note that even for \textbf{unsuccessful} iterations (where the iterates do not change) the objective function values are re-evaluated.
	\item (Objective gradient evaluations; Line~\ref{line.gradient_generation}) To generate an estimate of the  
	gradient of the objective function that satisfies the conditions of Oracle~\ref{oracle.first}, one can employ a procedure (a loop) similar to~\cite[Algorithm 2]{scheinberg2022stochastic}. The idea is to refine the estimate progressively in order to generate one that satisfies the condition. In many real-world problems, including empirical risk minimization in machine learning, one can improve the gradient approximation by progressively using a larger number of samples.
    \item (Maximum step size $\alpha_{\max}$) We pick $\alpha_{\max}\in(0,1]$ mainly to simplify our analysis. That being said, the unit upper bound on $\alpha_{\max}$ is motivated by the deterministic constraint setting. In the deterministic setting (without any noise), the merit function decrease is upper bounded by a nonsmooth function, whose only point of nonsmothness is at $\alpha = 1$, which complicates the analysis; see \cite[Lemma~2.13]{BeraCurtRobiZhou21}.
    \item (Parameters $\epsilon_f$ and $\epsilon_g$ in Oracles~\ref{oracle.zero} and~\ref{oracle.first}) To implement and run Algorithm~\ref{alg.sqp_line_search_practical}, (an estimate of) $\epsilon_f$ (bias in Oracle \ref{oracle.zero}) is required and $\epsilon_g$ (bias in Oracle~\ref{oracle.first}) is not. The noise in the objective gradient estimates is an intrinsic quantity of Oracle~\ref{oracle.first}, and while it is not required for implementing Algorithm~\ref{alg.sqp_line_search_practical}, it plays a central role in the analysis as it defines the neighborhood of convergence. With regards to $\epsilon_f$ (the noise in the objective function estimates), 
    the algorithm requires an estimate (or upper bound) of this quantity. In practice, this can be estimated via sampling or other techniques, e.g., \cite{hamming2012introduction}. Finally, we should note that similar constants (oracle conditions) are required for many of the stochastic adaptive line-/step-search (and trust region) methods in the unconstrained setting; see e.g., \cite{BeraCaoSche21,BeraByrdNoce19,OztoByrdNoce21,cao2022first,sun2023trust}. It is reasonable to expect that these requirements extend to the constrained stochastic setting.
\end{itemize}
\end{remark}

Before we proceed, we define the stochastic process related to the algorithm. Let $M_k$ denote $\{\Xi_k^0, \Xi_k^+, \Xi_k^{1}\}$ with realizations $\{\xi_k^0, \xi_k^+, \xi_k^1\}$. The algorithm generates a stochastic process: 
$\{({G}_k, D_k, \mathcal{T}_k, \bar\phi(X_k,\mathcal{T}_k;\Xi_k^0), \bar\phi(X_k^+,\mathcal{T}_k;\Xi_k^+), X_k, A_k)\}$  with realizations $\{(\bar g_k, \bar d_k, \bar\tau_k, \bar\phi(x_k,\bar\tau_k;\xi_k^0), \bar\phi(x_k^+,\bar\tau_k;\xi_k^+), x_k, \alpha_k)\}$, adapted to the filtration $\{{\cal F}_k:\, k\geq 0\}$, where ${\cal F}_k=\boldsymbol{\sigma} (M_0, M_1, \ldots, M_k)$ and $\boldsymbol{\sigma}$ denotes the $\boldsymbol{\sigma}$-algebra. 
At iteration $k$, $G_k$ is the random gradient, $D_k$ is the random primal search direction, $\mathcal{T}_k$ is the random merit parameter, $\bar\phi(X_k,\mathcal{T}_k;\Xi_k^0)$ and $\bar\phi(X_k^+,\mathcal{T}_k;\Xi_k^+)$ are the random noisy merit function evaluations at the current point and the candidate point, respectively, $X_k$ is the random iterate at iteration $k$ and $A_k$ is the random step size. Note that ${G}_k, D_k, \mathcal{T}_k$ are dictated by $\Xi_k^1$ (Oracle~\ref{oracle.first}) and the noisy merit function evaluations are dictated by $\Xi_k^0$ and $\Xi_k^+$ (Oracle~\ref{oracle.zero}).

\section{Theoretical analysis}\label{sec.analysis}
In this section, we analyze the behavior of Algorithm~\ref{alg.sqp_line_search_practical}. For brevity, throughout this section, we assume  Assumptions~\ref{ass.prob} and~\ref{ass.H} hold and do not restate this fact in every lemma and theorem. We begin by presenting some preliminary results, definitions, and assumptions and then proceed to present a worst-case iteration complexity bound for Algorithm~\ref{alg.sqp_line_search_practical}.

\subsection{Preliminaries, definitions \& assumptions}

We first define some \emph{deterministic} quantities that are used in the analysis of Algorithm~\ref{alg.sqp_line_search_practical}, and which are never explicitly computed in the implementation of the algorithm. 
For all $k\in\N{}$ in any realization 
of Algorithm~\ref{alg.sqp_line_search_practical}, let
$(\dTrue_k,\yTrue_k)\in\R{n}\times\R{m}$ be the solution of the deterministic counterpart of \eqref{eq.SQP}, i.e., 
\begin{equation}\label{eq.SQP_det}
    \begin{bmatrix} H_k & J_k^T \\ J_k & 0 \end{bmatrix}\begin{bmatrix} \dTrue_k \\ \yTrue_k \end{bmatrix} = - \begin{bmatrix} \nabla f_k \\ c_k \end{bmatrix}.
\end{equation}
The norm of the gradient of the Lagrangian (defined in~\eqref{eg.lagrangian}) of~\eqref{eq.prob}, used as a first-order stationarity measure, is bounded at every primal-dual iterate $(x_k,y_k)$ as
\begin{equation}\label{eq.stationarity_measurement}
    \left\|\bbmatrix \nabla f_k + J_k^T\yTrue_k \\ c_k \ebmatrix \right\| = \left\|\bbmatrix -H_k\dTrue_k \\ -J_k\dTrue_k \ebmatrix\right\| \leq (\kappa_H + \kappa_J)\|\dTrue_k\|,
\end{equation}
where the equality is by \eqref{eq.SQP_det} and the inequality follows by Assumptions~\ref{ass.prob} and~\ref{ass.H}. Thus,  \eqref{eq.stationarity_measurement} implies that $\dTrue_k$, the primal search direction, can be used as a proxy of the first-order stationary measure. The following lemma shows that the tuple $(\dTrue_k,\yTrue_k)$ is bounded for all $k\in\N{}$ in any realization 
of Algorithm~\ref{alg.sqp_line_search_practical}.
\begin{lemma}\label{lem.dy_bound}
There exist constants $\{\kappa_d,\kappa_y\}\subset\R{}_{>0}$ such that $\|\dTrue_k\| \leq \kappa_d$ and $\|\yTrue_k\| \leq \kappa_y$ for all $k\in\N{}$ in any realization 
of Algorithm~\ref{alg.sqp_line_search_practical}.
\end{lemma}
\begin{proof}
By the Cauchy–Schwarz inequality and \eqref{eq.SQP_det}, we have
\begin{equation*}
\left\|\bbmatrix \dTrue_k \\ \yTrue_k \ebmatrix\right\| = \left\|\bbmatrix H_k & J_k^T \\ J_k & 0 \ebmatrix^{-1}\bbmatrix \nabla f_k \\ c_k \ebmatrix\right\| \leq \left\|\bbmatrix H_k & J_k^T \\ J_k & 0 \ebmatrix^{-1}\right\|\left\|\bbmatrix \nabla f_k \\ c_k \ebmatrix\right\|,
\end{equation*}
where both terms on the right-hand side of the inequality are bounded by Assumptions~\ref{ass.prob} and~\ref{ass.H}, which concludes the proof.
\end{proof}

Moreover, for all $k\in\N{}$ in any realization 
of Algorithm~\ref{alg.sqp_line_search_practical}, we define $\tauTrue_{k} \in\R{}_{>0}$ and $\tautrialTrue_k \in\R{}_{>0}$, the deterministic counterparts of \eqref{eq.tau_update} and \eqref{eq.tau_trial},
\begin{equation}\label{eq.tautrue_update}
    \tauTrue_{k} \leftarrow \bcases \bar\tau_k \ &\text{if }\bar\tau_k \leq \tautrialTrue_k; \\
    \min\left\{(1-\epsilon_{\tau})\bar\tau_k,\tautrialTrue_k\right\} &\text{otherwise,}\ecases
\end{equation}
where
\begin{equation}\label{eq.tautrue_trial}
    \tautrialTrue_k \leftarrow \bcases \infty \ &\text{if }\nabla f_k^T\dTrue_k + \max\left\{\dTrue_k^TH_k\dTrue_k,0\right\} \leq 0; \\
    \tfrac{(1-\sigma)\|c_k\|_1}{\nabla f_k^T\dTrue_k + \max\left\{\dTrue_k^TH_k\dTrue_k,0\right\}} &\text{otherwise.} \ecases
\end{equation}
We emphasize again that $\{(\tauTrue_k,\tautrialTrue_k)\}_{k\in\N{}}$ are introduced only for the purposes of the analysis, and in Algorithm~\ref{alg.sqp_line_search_practical} they are never computed (not even in the setting in which the true gradient is used, i.e., $\bar{g}_k = \nabla f(x_k)$). We also note that this definition is not the same as that in~\cite{BeraCurtRobiZhou21,CurtONeiRobi21}. The difference is in the fact that in the computation of $\tauTrue_k$, the comparison is made to $\bar\tau_k$ instead of $\bar\tau_{k-1}$. This is important for the analysis, since this guarantees $\tau_k \leq \bar{\tau}_k$, for all $k\in\N{}$ in any realization 
of Algorithm~\ref{alg.sqp_line_search_practical}.

\change{We analyze Algorithm~\ref{alg.sqp_line_search_practical} within the context of a specific event occurring that pertains to the merit parameters. 
Specifically, we consider the event in which the stochastic process associated with the merit parameter sequence $\{\mathcal{\Tau}_k\}$ (with realizations $\{\bar{\tau}_k\}$) generated by Algorithm \ref{alg.sqp_line_search_practical} is bounded away from zero (Assumption~\ref{ass.good_merit_paramter}). Such an event has been adopted in previous literature \cite{BeraCurtRobiZhou21,BeraCurtONeiRobi21,BeraShiYiZhou22,CurtONeiRobi21,CurtRobiZhou21}; we refer readers to \cite[Section~3.2]{BeraCurtRobiZhou21} and \cite[Section~4.2]{CurtONeiRobi21} for detailed discussions. The assumption on the occurrence of such an event is also supported by our numerical experiments; see the discussion in Section~\ref{sec.noisy_noisy}. We emphasize that the implementation of the algorithm does not require knowledge of $\bar\tau_{\min}$ (see Assumption \ref{ass.good_merit_paramter} for the definition). Moreover, if the stochastic gradients happen to be uniformly bounded, then one can show that the merit parameter sequence is bounded away from zero under Assumptions~\ref{ass.prob} and~\ref{ass.H}; see e.g.,  \cite[Proposition~3.18]{BeraCurtRobiZhou21} and \cite[Lemma~4.6]{BeraShiYiZhou22}.}

\change{
\begin{assumption}\label{ass.good_merit_paramter}
Event $\mathcal{E}:=\mathcal{E}(\bar{k}_{\max},\bar\tau_{\min})$ occurs, with given constants $\bar{k}_{\max}\in\N{}$ and $\bar\tau_{\min}\in\R{}_{>0}$, in the sense that there exists $\mathcal{\Tau}'\in\R{}_{>0}$ such that
\begin{equation*}
    \mathcal{\Tau}_k = \mathcal{\Tau}' \geq \bar\tau_{\min}>0, \quad\text{ for all }\quad k\geq \bar{k}_{\max},
\end{equation*}
where $\{\mathcal{\Tau}_k\}$ is a stochastic process with realizations $\{\bar\tau_k\}$, a sequence of merit parameter values, generated by Algorithm \ref{alg.sqp_line_search_practical}. Moreover, conditioned on the occurrence of the event $\mathcal{E}$, \eqref{eq:zero_order}--\eqref{eq:first_order} hold with the same constants for all primal iterates $\{X_k\}$ (with realization $\{x_k\}$) and all primal trial iterates $\{X_k^+\}$ (with realization $\{x_k^+\}$) generated by Algorithm~\ref{alg.sqp_line_search_practical}. 
\end{assumption}
}



\change{Throughout the remainder of the paper, we assume that the event $\mathcal{E}$ occurs, meaning that Assumption \ref{ass.good_merit_paramter} holds.} Next, we state and \change{prove} a useful property with regards to the deterministic merit parameter sequence $\{\tauTrue_k\}$ defined in~\eqref{eq.tautrue_update}.

\begin{lemma}\label{lem.deter_tau_lb}
Suppose Assumption~\ref{ass.good_merit_paramter} holds, then there exists a positive constant $\tauTrue_{\min}\in\R{}_{>0}$ such that, for any realization 
of Algorithm~\ref{alg.sqp_line_search_practical} \change{for which event $\mathcal{E}$ occurs}, $\tauTrue_k \geq \tauTrue_{\min}$ for all $k\in\mathbb{N}$.
\end{lemma}
\begin{proof}
By \cite[Lemma~2.16]{BeraCurtRobiZhou21}, $\{\tautrialTrue_k\}\subset\mathbb{R}_{>0} \cup \{+\infty\}$ 
is always bounded away from zero. We define $\tautrialTrue_{\min}\in\R{}_{>0}$ such that $\tautrialTrue_{\min} \leq \tautrialTrue_k$ for all $k\in\N{}$ in every realization 
of Algorithm~\ref{alg.sqp_line_search_practical}. \change{When event $\mathcal{E}$ occurs (see Assumption~\ref{ass.good_merit_paramter}), by \eqref{eq.tautrue_update}--\eqref{eq.tautrue_trial},} 
one may pick $\tau_{\min} = \min\{(1-\epsilon_{\tau})\tautrialTrue_{\min},\bar\tau_{\min}\}$ to conclude the proof.
\end{proof}
 
Let $E_k$ and $E_k^+$ be the errors in the objective function evaluations from Oracle~\ref{oracle.zero}, i.e., 
$E_k := \abs{\bar{f}(X_k; \Xi_k^0)-f(X_k)}$ and $E_k^+ := \abs{\bar{f}(X_k^+; \Xi_k^+)-f(X_k^+)}$, with realizations $e_k$ and $e_k^+$, respectively. Next, we introduce several definitions necessary for the analysis of Algorithm~\ref{alg.sqp_line_search_practical}. 
Specifically, we define \emph{\textbf{true/false}} iterations (Definition~\ref{def.true}), \emph{\textbf{successful/unsuccessful}} iterations (Definition~\ref{def.successful}) and \emph{\textbf{large/small steps}} (Definition~\ref{def.largestep}), and introduce three indicator variables respectively. 

\begin{definition}\label{def.true}
For any realization 
of Algorithm~\ref{alg.sqp_line_search_practical}, iteration $k\in\N{}$ is \textbf{true} if 
\begin{equation}\label{eq.true}
    \|\bar{g}_k - \nabla f_k\| \leq \max\left\{ \epsilon_g, \kappa_{\mathrm{FO}}\alpha_k\sqrt{\Delta l(x_k,\bar\tau_{k},\bar{g}_k,\bar{d}_k)} \right\} \text{ and } e_k + e_k^+ \leq 2\epsilon_f,
\end{equation}
where $\Delta l(x_k,\bar\tau_{k},\bar{g}_k,\bar{d}_k)$ is defined in  \eqref{eq.model_reduction} and the constants 
$\epsilon_f$, $\epsilon_g$ and $\kappa_{\mathrm{FO}}$ are the same as in Oracles~\ref{oracle.zero} and \ref{oracle.first}. If \eqref{eq.true} does not hold, we call the iteration a \textbf{false} iteration. We use the random indicator variable $I_k$ to denote if an iteration is true.
\end{definition}
	
\begin{definition}\label{def.successful} Given $\theta\in (0,1)$, for all $k\in\N{}$ in any realization 
of Algorithm~\ref{alg.sqp_line_search_practical}, let $\bar\phi(x_k,\bar\tau_{k};\xi_k)$ and $\bar\phi(x_k^+,\bar\tau_{k};\xi_k^+)$ be obtained by Oracle~\ref{oracle.zero}. 
If \eqref{eq.line_serach_cond_old_old} 
holds, then iteration $k$ is \textbf{successful}, otherwise, it is an \textbf{unsuccessful} iteration. We use the random indicator variable $\Theta_k$ to denote whether an iteration is successful.
\end{definition}
	
\begin{definition}\label{def.largestep}
For all $k\in\N{}$ in any realization 
of Algorithm~\ref{alg.sqp_line_search_practical}, if $\min\{ \alpha_k,\alpha_{k+1} \} \geq \tilde\alpha$ where $\tilde\alpha$ is some problem-dependent positive real number (defined in Lemma~\ref{algo_behave}),  then we call the step a \textbf{large step} and set the indicator variable $U_k = 1$.  Otherwise, we call the step $k$ a \textbf{small step} and set $U_k = 0$.
\end{definition}

We show that under appropriate conditions, if the step is a \emph{\textbf{small step}} and the iteration is \emph{\textbf{true}}, then, the iteration is guaranteed to be \emph{\textbf{successful}}  (see Lemma \ref{algo_behave}). 
The last definition is for the stopping time $(T_{\varepsilon_{\Delta l}})$ and a measure of progress $(\{Z_k\})$.

\begin{definition}\label{def.iter_term}
\sloppy For any realization 
of Algorithm~\ref{alg.sqp_line_search_practical}, define $T_{\varepsilon_{\Delta l}} = \min\{ k:\sqrt{\Delta {l(x_k,\tauTrue_k,\nabla f_k,
\dTrue_k)}} \leq \varepsilon_{\Delta l} \}$, the number of iterations required 
to reach a first-order $\varepsilon$-stationary iterate, where $\varepsilon = \Omega(\varepsilon_{\Delta l})$. 
We discuss the explicit relationship between $\varepsilon$ and $\varepsilon_{\Delta l}$ in Remark~\ref{rem.stationary_measure}. Moreover, for all $k\in\mathbb{N}$ in any realization 
of Algorithm~\ref{alg.sqp_line_search_practical} \change{for which event $\mathcal{E}$ occurs (see Assumption~\ref{ass.good_merit_paramter})}, let $Z_k := \phi(x_k,\bar\tau_{k}) - \phi_{\min} -(\bar\tau_{k}f_{\inf} - \bar\tau_{\min}f_{\inf})$, where $\phi_{\min}$ is a lower bound of $\phi(\cdot,\bar\tau_{\min})$ over $\Xcal$ and $\bar\tau_{\min}$ is the constant 
defined in Assumption~\ref{ass.good_merit_paramter}.
\end{definition}

\begin{remark}\label{rem.stationary_measure}

A key ingredient of our algorithm is the stopping time $T_{\varepsilon_{\Delta l}}$ that is related to $\Delta l(x_k,\tau_k,\nabla f_k,d_k)$. In fact, by \eqref{eq.stationarity_measurement}, \change{Assumption~\ref{ass.prob}, Assumption~\ref{ass.H}, Assumption~\ref{ass.good_merit_paramter}, Lemma~\ref{lem.deter_tau_lb}} and Lemma~\ref{lem.Deltal_lb_det} (see below), the stopping time $T_{\varepsilon_{\Delta l}}$ defined in Definition~\ref{def.iter_term} is the number of iterations needed to achieve a first-order $\varepsilon$-stationary iterate, i.e., 
\begin{equation}\label{eq.epsilon_complexity}
	\max\{\|\nabla f_k + J_k^Ty_k\|,\sqrt{\|c_k\|}\} \leq \varepsilon, \quad\text{where}\quad \varepsilon = \tfrac{\max\{\kappa_H,1\}}{\sqrt{\kappa_l\tau_{\min}}}\cdot \varepsilon_{\Delta l}
\end{equation}
\change{and $\{\kappa_H,\tau_{\min},\kappa_l\}\subset\R{}_{>0}$ are defined in Assumption~\ref{ass.H} and Lemmas~\ref{lem.deter_tau_lb} and~\ref{lem.Deltal_lb_det}.} 
Due to the existence of noise and bias in the zeroth- and first-order oracles (Oracles~\ref{oracle.zero} and~\ref{oracle.first}), $\varepsilon$ is to be bounded away from zero.
We note that \eqref{eq.epsilon_complexity} is the same stationarity measure as that used in \cite[Eq. (5)]{CurtONeiRobi21}, and is a non-standard first-order stationary measure compared to $\left\|\begin{bmatrix} \nabla f_k + J_k^Ty_k \\ c_k \end{bmatrix}\right\|$. That said, one can show that $\left\|\begin{bmatrix} \nabla f_k + J_k^Ty_k \\ c_k \end{bmatrix}\right\| \leq 2\max\{\|\nabla f_k + J_k^Ty_k\|,\|c_k\|\} \leq \tfrac{2\max\{\kappa_H,\kappa_J\}}{\sqrt{\kappa_l\tau_{\min}}}\varepsilon_{\Delta l} = \Omega(\varepsilon)$. Throughout this paper we focus on (and provide $\varepsilon$-complexity bounds for) \eqref{eq.epsilon_complexity} as it provides a stronger result for feasibility ($\|c_k\|$) when $\varepsilon < 1$. 
\end{remark}

\subsection{Main Technical Results}

We build toward the main result of the paper (Theorem~\ref{thm:subexp_noise}) through a sequence of technical lemmas. Our first lemma shows that $Z_k$ (defined in Definition~\ref{def.iter_term}) is always non-negative.
\begin{lemma}\label{lem_progress} For all $k\in\N{}$ in any realization 
of Algorithm~\ref{alg.sqp_line_search_practical} \change{for which event $\mathcal{E}$ occurs (see Assumption~\ref{ass.good_merit_paramter}),} 
$Z_k \geq 0$.
\end{lemma}
\begin{proof}
It follows from \eqref{eq.merit_function} and Definition~\ref{def.iter_term} that
\begin{equation*}
\begin{aligned}
    Z_k &= \phi(x_k,\bar\tau_{k}) - \phi_{\min} - (\bar\tau_{k}f_{\inf} - \bar\tau_{\min}f_{\inf})\\
    &= (\bar\tau_k(f_k-f_{\inf}) + \|c_k\|_1) - \phi_{\min} + \bar\tau_{\min}f_{\inf}\\
    &\geq (\bar\tau_{\min}(f_k-f_{\inf}) + \|c_k\|_1) - \phi_{\min} + \bar\tau_{\min}f_{\inf}
    = \phi(x_k,\bar\tau_{\min}) - \phi_{\min} \geq 0,
\end{aligned}
\end{equation*}
which concludes the proof.
\end{proof}
By Lemma~\ref{lem_progress}, \change{if event $\mathcal{E}$ occurs, 
then} $\{Z_k\}$ is always non-negative and $Z_k = 0$ if and only if $\bar\tau_k(f_k-f_{\inf}) = \bar\tau_{\min}(f_k-f_{\inf})$ and $\phi(x_k,\bar\tau_{\min}) = \phi_{\min}$. In fact, $\{Z_k\}$ uniformly bounded below by a constant suffices to prove our main theoretical results (we do not necessarily need $\{Z_k\}$ to converge to zero). A similar property of $\{Z_k\}$ also appears in the unconstrained setting; see e.g., \cite{jin2022high}.

The next lemma provides a useful lower bound for the reduction in the model of the merit function, $\Delta l(x_k,\bar\tau_k,\bar{g}_k,\bar{d}_k)$, that is related to the primal search direction ($\|\bar{d}_k\|^2$) and a measure of infeasibility ($\|c_k\|$).

\begin{lemma}\label{lem.Deltal_lb}
There exists some constant $\kappa_l\in\R{}_{>0}$ such that for all $k\in\N{}$ in any realization 
of Algorithm~\ref{alg.sqp_line_search_practical}, $\Delta l(x_k,\bar\tau_k,\bar{g}_k,\bar{d}_k) \geq \kappa_l\bar\tau_k(\|\bar{d}_k\|^2 + \|c_k\|_1)$.
\end{lemma}
\begin{proof}
For any iteration $k\in\N{}$, by \cite[Lemma~3.4]{BeraCurtRobiZhou21}, there exists some constant $\kappa_l\in\R{}_{>0}$ such that $-\bar\tau_k(\bar{g}_k^T\bar{d}_k + \tfrac{1}{2}\max\{\bar{d}_k^TH_k\bar{d}_k,0\}) + \|c_k\|_1 \geq \kappa_l\bar\tau_k(\|\bar{d}_k\|^2 + \|c_k\|_1)$. 
By $\bar\tau_k\in\R{}_{>0}$ (from Lemma~\ref{lem.model_reduction}), this implies $\Delta l(x_k,\bar\tau_k,\bar{g}_k,\bar{d}_k) = -\bar\tau_k\bar{g}_k^T\bar{d}_k + \|c_k\|_1 \geq -\bar\tau_k(\bar{g}_k^T\bar{d}_k + \tfrac{1}{2}\max\{\bar{d}_k^TH_k\bar{d}_k,0\}) + \|c_k\|_1$, which concludes the proof.
\end{proof}

\begin{lemma}\label{lem.Deltal_lb_det}
There exists some constant $\kappa_l\in\R{}_{>0}$ such that for all $k\in\N{}$ in any realization 
of Algorithm~\ref{alg.sqp_line_search_practical}, $\Delta l(x_k,\tau_k,g_k,d_k) \geq \kappa_l\tau_k(\|d_k\|^2 + \|c_k\|_1)$.
\end{lemma}
\begin{proof} The proof follows the same logic as that of Lemma~\ref{lem.Deltal_lb} with the stochastic quantities replaced by their deterministic counterparts. By \cite[Lemma~3.4]{BeraCurtRobiZhou21}, the desired inequality is satisfied for the same constant $\kappa_l$ defined in Lemma~\ref{lem.Deltal_lb}.
\end{proof}

The next lemma bounds the errors in the  stochastic search directions and dual variables, respectively, with respect to the errors in the gradient approximations.

\begin{lemma}\label{lem.step_diff_bound}
For all $k\in\N{}$ in any realization 
of Algorithm~\ref{alg.sqp_line_search_practical}, there exist constants $\{\zeta, \zeta_y\}\subset\R{}_{>0}$ such that $\|\bar{d}_k - \dTrue_k\| \leq \zeta^{-1}\|\bar{g}_k - \nabla f_k\|$ and $\|\bar{y}_k - \yTrue_k\| \leq \zeta_y\|\bar{g}_k - \nabla f_k\|$, where $\zeta$ is defined in Assumption~\ref{ass.H}. 
\end{lemma}
\begin{proof}
By the Cauchy–Schwarz inequality, Assumption~\ref{ass.H}, \eqref{eq.SQP_det}, and the fact that $(\bar{d}_k - \dTrue_k) \in\Null(J_k)$, it follows that
\begin{equation}\label{eq.new}
\begin{aligned}
    \|\bar{d}_k - \dTrue_k\|\|\bar{g}_k - \nabla f_k\| &\geq (\bar{d}_k - \dTrue_k)^T(\nabla f_k - \bar{g}_k) \\
    &= (\bar{d}_k - \dTrue_k)^T(H_k(\bar{d}_k - \dTrue_k) + J_k^T(\bar{y}_k - \yTrue_k)) \\
    &= (\bar{d}_k - \dTrue_k)^TH_k(\bar{d}_k - \dTrue_k) \geq \zeta\|\bar{d}_k - \dTrue_k\|^2,
\end{aligned}
\end{equation}
and 
$\|\bar{d}_k - \dTrue_k\| \leq \zeta^{-1}\|\bar{g}_k - \nabla f_k\|$. 
By the triangle and  Cauchy–Schwarz inequalities, Assumptions~\ref{ass.prob} and \ref{ass.H}, and \eqref{eq.SQP_det} and \eqref{eq.new}, 
it follows that
\begin{equation*}
\begin{aligned}
\|\bar{y}_k - \yTrue_k\| &= \|(J_kJ_k^T)^{-1}J_k\left((\bar{g}_k - \nabla f_k) + H_k(\bar{d}_k - \dTrue_k)\right)\| \\
&\leq \|(J_kJ_k^T)^{-1}\|\|J_k\|(\|\bar{g}_k - \nabla f_k\| + \|H_k\|\|\bar{d}_k - \dTrue_k\|) \\
&\leq \kappa_{\sigma}\kappa_J(1+\kappa_H\zeta^{-1})\|\bar{g}_k - \nabla f_k\|.
\end{aligned}
\end{equation*}
Setting $\zeta_y = \kappa_{\sigma}\kappa_J(1+\kappa_H\zeta^{-1})$ completes the proof. 
\end{proof}

The next lemma relates the inner product of the stochastic gradient and stochastic search direction to the stochastic reduction in the model of the merit function. We consider two cases that are related to the two cases in the max term of Oracle~\ref{oracle.first}.

\begin{lemma}\label{lem.|gd|_2}
For all $k\in\N{}$ in any realization 
of Algorithm~\ref{alg.sqp_line_search_practical}:
\begin{itemize}
	\item If $\|\bar{g}_k - \nabla f_k\| \leq \kappa_{\mathrm{FO}}\alpha_k\sqrt{\Delta l(x_k,\bar\tau_{k},\bar{g}_k,\bar{d}_k)}$, then
	\begin{align*}
	    \bar\tau_k|\bar{g}_k^T\bar{d}_k| \leq \left(\tfrac{\max\{\kappa_H,\kappa_y\}}{\kappa_l} + \tfrac{\sqrt{\bar\tau_k}\left(1 + \kappa_H\zeta^{-1}\right)\kappa_{\mathrm{FO}}\alpha_k}{\sqrt{\kappa_l}}\right)\Delta l(x_k,\bar\tau_k,\bar{g}_k,\bar{d}_k).
	\end{align*}
	\item If $\|\bar{g}_k - \nabla f_k\| \leq \epsilon_g$, 
	\begin{align*}
	    \bar\tau_k|\bar{g}_k^T\bar{d}_k| \leq \tfrac{\max\{\kappa_H,\kappa_y\} + 1}{\kappa_l} \Delta l(x_k,\bar\tau_k,\bar{g}_k,\bar{d}_k) + \tfrac{\bar\tau_k\left(1 + \kappa_H\zeta^{-1}\right)^2}{4}\epsilon_g^2.
	\end{align*}
\end{itemize}
\end{lemma}
\begin{proof} 
If $\|\bar{g}_k - \nabla f_k\| \leq \kappa_{\mathrm{FO}}\alpha_k\sqrt{\Delta l(x_k,\bar\tau_{k},\bar{g}_k,\bar{d}_k)}$, by the triangle inequality, \eqref{eq.SQP}, Assumption~\ref{ass.H}, and Lemmas~\ref{lem.dy_bound},~\ref{lem.Deltal_lb} and~\ref{lem.step_diff_bound}, it follows that
\begin{align*}
    \bar\tau_k|\bar{g}_k^T\bar{d}_k|=\ & \bar\tau_k|(H_k\bar{d}_k + J_k^T\yTrue_k + J_k^T(\bar{y}_k - \yTrue_k))^T\bar{d}_k| \\
    \leq \ &\bar\tau_k(|\bar{d}_k^TH_k\bar{d}_k| + |\yTrue_k^TJ_k\bar{d}_k| + |(\bar{y}_k - \yTrue_k)^TJ_k\bar{d}_k|) \\
    \leq \ &\bar\tau_k(\kappa_H\|\bar{d}_k\|^2 + \|\yTrue_k\|\|c_k\| + \|(\bar{g}_k - \nabla f_k) + H_k(\bar{d}_k - \dTrue_k)\|\|\bar{d}_k\|) \\
    \leq \ &\max\{\kappa_H,\kappa_y\} \bar\tau_k(\|\bar{d}_k\|^2 + \|c_k\|) + \bar\tau_k(\|\bar{g}_k - \nabla f_k\| + \kappa_H\|\bar{d}_k - \dTrue_k\|)\|\bar{d}_k\| \\
    \leq \ &\tfrac{\max\{\kappa_H,\kappa_y\}}{\kappa_l}\Delta l(x_k,\bar\tau_k,\bar{g}_k,\bar{d}_k) + \bar\tau_k\left(1 + \kappa_H\zeta^{-1}\right)\|\bar{g}_k - \nabla f_k\|\|\bar{d}_k\| \\
    \leq \ &\tfrac{\max\{\kappa_H,\kappa_y\}}{\kappa_l}\Delta l(x_k,\bar\tau_k,\bar{g}_k,\bar{d}_k) + \tfrac{\sqrt{\bar\tau_k}\left(1 + \kappa_H\zeta^{-1}\right)\kappa_{\mathrm{FO}}\alpha_k}{\sqrt{\kappa_l}}\Delta l(x_k,\bar\tau_k,\bar{g}_k,\bar{d}_k), 
\end{align*}
which completes the first part of the proof.

Using similar logic, if $\|\bar{g}_k - \nabla f_k\| \leq \epsilon_g$, by the triangle inequality, \eqref{eq.SQP}, 
Assumption~\ref{ass.H}, \change{Lemmas~\ref{lem.dy_bound},~\ref{lem.Deltal_lb},~\ref{lem.step_diff_bound}},  and the fact that $ab\leq a^2+\tfrac{b^2}{4}$ holds for any $\{a,b\}\subset\mathbb{R}$,  it follows that
\begin{align*}
   \bar\tau_k|\bar{g}_k^T\bar{d}_k| \leq \ &\tfrac{\max\{\kappa_H,\kappa_y\}}{\kappa_l}\Delta l(x_k,\bar\tau_k,\bar{g}_k,\bar{d}_k) + \bar\tau_k\left(1 + \kappa_H\zeta^{-1}\right)\|\bar{g}_k - \nabla f_k\|\|\bar{d}_k\| \\
    \leq \ &\tfrac{\max\{\kappa_H,\kappa_y\}}{\kappa_l}\Delta l(x_k,\bar\tau_k,\bar{g}_k,\bar{d}_k) + \tfrac{\sqrt{\bar\tau_k}\left(1 + \kappa_H\zeta^{-1}\right)}{\sqrt{\kappa_l}}\epsilon_g\sqrt{\Delta l(x_k,\bar\tau_k,\bar{g}_k,\bar{d}_k)} \\
    \leq \ &\tfrac{\max\{\kappa_H,\kappa_y\} + 1}{\kappa_l} \Delta l(x_k,\bar\tau_k,\bar{g}_k,\bar{d}_k) + \tfrac{\bar\tau_k\left(1 + \kappa_H\zeta^{-1}\right)^2}{4}\epsilon_g^2,
\end{align*}
which completes the proof.
\end{proof}

The next lemma provides a useful upper bounds for the errors related to the stochastic search directions (and gradients) for the same two cases as in Lemma~\ref{lem.|gd|_2}.

\begin{lemma}\label{lem.gd_dHd_diff_case_a_2}
For all $k\in\N{}$ in any realization 
of Algorithm~\ref{alg.sqp_line_search_practical}:
\begin{itemize}
\item If $\|\bar{g}_k - \nabla f_k\| \leq \kappa_{\mathrm{FO}}\alpha_k\sqrt{\Delta l(x_k,\bar\tau_{k},\bar{g}_k,\bar{d}_k)}$, then
\begin{equation*}
\begin{aligned}
    |\nabla f_k^T\dTrue_k - \bar{g}_k^T\bar{d}_k| \leq \ &\left(\tfrac{(1+\kappa_H\zeta^{-1})\kappa_{\mathrm{FO}}\alpha_k}{\sqrt{\kappa_l\bar\tau_k}} + \tfrac{\kappa_{\mathrm{FO}}^2\alpha_k^2}{\zeta}\right)\Delta l(x_k,\bar\tau_k,\bar{g}_k,\bar{d}_k) \\ 
    \text{and}\ \ |\dTrue_k^TH_k\dTrue_k - \bar{d}_k^TH_k\bar{d}_k| \leq \ &\left(\tfrac{2\kappa_H\zeta^{-1}\kappa_{\mathrm{FO}}\alpha_k}{\sqrt{\kappa_l\bar\tau_k}} + \tfrac{\kappa_H\kappa_{\mathrm{FO}}^2\alpha_k^2}{\zeta^2}\right)\Delta l(x_k,\bar\tau_k,\bar{g}_k,\bar{d}_k).
\end{aligned}
\end{equation*}
\item If $\|\bar{g}_k - \nabla f_k\| \leq \epsilon_g$, then
\begin{equation*}
\begin{aligned}
|\nabla f_k^T\dTrue_k - \bar{g}_k^T\bar{d}_k| \leq \ & \tfrac{(1+\kappa_H\zeta^{-1})\epsilon_g}{\sqrt{\kappa_l\tauTrue_k}}\sqrt{\Delta l(x_k,\tauTrue_k,\nabla f_k,\dTrue_k)} + \zeta^{-1}\epsilon_g^2 \\ 
\text{and}\ \ |\dTrue_k^TH_k\dTrue_k - \bar{d}_k^TH_k\bar{d}_k| \leq \ & \tfrac{2\kappa_H\zeta^{-1}\epsilon_g}{\sqrt{\kappa_l\tauTrue_k}}\sqrt{\Delta l(x_k,\tauTrue_k,\nabla f_k,\dTrue_k)} + \kappa_H\zeta^{-2}\epsilon_g^2 .
\end{aligned}
\end{equation*}
\end{itemize}
\end{lemma}
\begin{proof}
We begin with $\|\bar{g}_k - \nabla f_k\| \leq \kappa_{\mathrm{FO}}\alpha_k\sqrt{\Delta l(x_k,\bar\tau_{k},\bar{g}_k,\bar{d}_k)}$. By the triangle and Cauchy–Schwarz inequalities, Assumption~\ref{ass.prob}, and Lemmas~\ref{lem.dy_bound},~\ref{lem.Deltal_lb} and~\ref{lem.step_diff_bound}, 
\begin{align*}
    &|\nabla f_k^T\dTrue_k - \bar{g}_k^T\bar{d}_k| \\
    = \ &|(\bar{g}_k - \nabla f_k)^T\bar{d}_k + (\nabla f_k - \bar{g}_k)^T(\bar{d}_k - \dTrue_k) + \bar{g}_k^T(\bar{d}_k - \dTrue_k)| \\
    = \ &|(\bar{g}_k - \nabla f_k)^T\bar{d}_k + (\nabla f_k - \bar{g}_k)^T(\bar{d}_k - \dTrue_k) - (H_k\bar{d}_k + J_k^T\bar{y}_k)^T(\bar{d}_k - \dTrue_k)| \\
    \leq \ &\|\bar{g}_k - \nabla f_k\|\|\bar{d}_k\| + \|\nabla f_k - \bar{g}_k\|\|\bar{d}_k - \dTrue_k\| + \kappa_H\|\bar{d}_k\|\|\bar{d}_k - \dTrue_k\| \\
    \leq \ &\left(\tfrac{(1+\kappa_H\zeta^{-1})\kappa_{\mathrm{FO}}\alpha_k}{\sqrt{\kappa_l\bar\tau_k}} + \zeta^{-1}\kappa_{\mathrm{FO}}^2\alpha_k^2\right)\Delta l(x_k,\bar\tau_k,\bar{g}_k,\bar{d}_k).
\end{align*}
Additionally, under Assumption~\ref{ass.H} it follows that
\begin{align*}
    |\dTrue_k^TH_k\dTrue_k - \bar{d}_k^TH_k\bar{d}_k| = \ &|2\bar{d}_k^TH_k(\bar{d}_k - \dTrue_k) - (\bar{d}_k - \dTrue_k)^TH_k(\bar{d}_k - \dTrue_k)| \\
    \leq \ &2|\bar{d}_k^TH_k(\bar{d}_k - \dTrue_k)| + |(\bar{d}_k - \dTrue_k)^TH_k(\bar{d}_k - \dTrue_k)| \\
    \leq \ &2\kappa_H\|\bar{d}_k\|\|\bar{d}_k - \dTrue_k\| + \kappa_H\|\bar{d}_k - \dTrue_k\|^2 \\
    \leq \ &\left(\tfrac{2\kappa_H\zeta^{-1}\kappa_{\mathrm{FO}}\alpha_k}{\sqrt{\kappa_l\bar\tau_k}} + \kappa_H\zeta^{-2}\kappa_{\mathrm{FO}}^2\alpha_k^2\right)\Delta l(x_k,\bar\tau_k,\bar{g}_k,\bar{d}_k),
\end{align*}
which completes the first part of the proof.

If $\|\bar{g}_k - \nabla f_k\| \leq \epsilon_g$, following similar logic as the first part of the proof, by the triangle and Cauchy–Schwarz inequalities,  \eqref{eq.SQP_det}, and Lemmas~\ref{lem.dy_bound},~\ref{lem.Deltal_lb_det} and~\ref{lem.step_diff_bound}, 
\begin{align*}
    &|\nabla f_k^T\dTrue_k - \bar{g}_k^T\bar{d}_k| \\
    = \ &|(\bar{g}_k - \nabla f_k)^T(\bar{d}_k - \dTrue_k) + (\bar{g}_k - \nabla f_k)^T\dTrue_k + \nabla f_k^T(\bar{d}_k - \dTrue_k)| \\
    = \ &|(\bar{g}_k - \nabla f_k)^T(\bar{d}_k - \dTrue_k) + (\bar{g}_k - \nabla f_k)^T\dTrue_k - (H_k\dTrue_k + J_k^T\yTrue_k)^T(\bar{d}_k - \dTrue_k)| \\
    \leq \ &|(\bar{g}_k - \nabla f_k)^T(\bar{d}_k - \dTrue_k)| + |(\bar{g}_k - \nabla f_k)^T\dTrue_k| + |\dTrue_k^TH_k(\bar{d}_k - \dTrue_k)| + |\yTrue_k^TJ_k(\bar{d}_k - \dTrue_k)| \\
    \leq \ &\zeta^{-1}\|\bar{g}_k - \nabla f_k\|^2 + (1 + \kappa_H\zeta^{-1})\|\dTrue_k\|\|\bar{g}_k - \nabla f_k\| \\
    \leq \ &\zeta^{-1}\epsilon_g^2 + \tfrac{(1+\kappa_H\zeta^{-1})\epsilon_g}{\sqrt{\kappa_l\tauTrue_k}}\sqrt{\Delta l(x_k,\tauTrue_k,\nabla f_k,\dTrue_k)}.
\end{align*}
Additionally, under Assumption~\ref{ass.H} it follows that
\begin{align*}
    |\dTrue_k^TH_k\dTrue_k - \bar{d}_k^TH_k\bar{d}_k|= \ &|(\dTrue_k - \bar{d}_k)^TH_k(\dTrue_k - \bar{d}_k) + 2\dTrue_k^TH_k(\bar{d}_k - \dTrue_k)| \\
    \leq \ &\kappa_H\|\dTrue_k - \bar{d}_k\|^2 + 2\kappa_H\|\dTrue_k\|\|\dTrue_k - \bar{d}_k\| \\
    \leq \ &\kappa_H\zeta^{-2}\epsilon_g^2 + \tfrac{2\kappa_H\zeta^{-1}\epsilon_g}{\sqrt{\kappa_l\tauTrue_k}}\sqrt{\Delta l(x_k,\tauTrue_k,\nabla f_k,\dTrue_k)},
\end{align*}
which completes the proof.
\end{proof}

The next lemma provides a bound on the merit function across an iteration.

\begin{lemma}\label{lem.suff_decrease_exact}
For all $k\in\N{}$ in any realization 
of Algorithm~\ref{alg.sqp_line_search_practical},
\begin{equation*}
\begin{aligned}
    &\phi(x_k + \alpha_k\bar{d}_k,\bar\tau_k) - \phi(x_k,\bar\tau_k) \\
    \leq \ &-\alpha_k\Delta l(x_k,\bar\tau_k,\bar{g}_k,\bar{d}_k) + \alpha_k\bar\tau_k(\nabla f_k - \bar{g}_k)^T\bar{d}_k + \tfrac{\bar\tau_kL+\Gamma}{2}\alpha_k^2\|\bar{d}_k\|^2.
\end{aligned}
\end{equation*}
\end{lemma}
\begin{proof}
By Algorithm~\ref{alg.sqp_line_search_practical}, for any $k\in\N{}$, $0 < \alpha_k\leq \alpha_{\max} \leq 1$. Moreover, by the triangle inequality, \eqref{eq.lipschitz_continuity}, \eqref{eq.merit_function} and \eqref{eq.SQP}, it follows that
\begin{align*}
&\phi(x_k + \alpha_k \bar{d}_k, \bar\tau_k) - \phi(x_k,\bar\tau_k) \\
= \ &\bar\tau_k(f(x_k + \alpha_k\bar{d}_k) - f_k) + (\|c(x_k + \alpha_k\bar{d}_k)\|_1 - \|c_k\|_1) \\
\leq \ &\bar\tau_k(\alpha_k\nabla f_k^T\bar{d}_k + \tfrac{L}{2}\alpha_k^2\|\bar{d}_k\|^2) + (\|c_k + \alpha_kJ_k\bar{d}_k\|_1 - \|c_k\|_1 + \tfrac{\Gamma}{2}\alpha_k^2\|\bar{d}_k\|^2) \\
\leq \ &\alpha_k\bar\tau_k\nabla f_k^T\bar{d}_k + |1-\alpha_k|\|c_k\|_1 + \alpha_k\|c_k + J_k\bar{d}_k\|_1 - \|c_k\|_1 + \tfrac{\bar\tau_kL+\Gamma}{2}\alpha_k^2\|\bar{d}_k\|^2 \\
    = \ &-\alpha_k\Delta l(x_k,\bar\tau_k,\bar{g}_k,\bar{d}_k) + \alpha_k\bar\tau_k(\nabla f_k - \bar{g}_k)^T\bar{d}_k + \tfrac{\bar\tau_kL+\Gamma}{2}\alpha_k^2\|\bar{d}_k\|^2,
\end{align*}
which completes the proof.
\end{proof}

Due to the quality and reliability of the zeroth- and first-order oracles (Oracles~\ref{oracle.zero} and \ref{oracle.first}), one can only guarantee convergence to a neighborhood of the solution. Assumption~\ref{ass.main_accuracy} provides a lower bound on the size of the convergence neighbourhood in terms of $\varepsilon$ (and $\varepsilon_{\Delta l}$). We note again that similar restrictions are required in the unconstrained noisy setting with inexact probabilistic oracles; see e.g., \cite{BeraCaoSche21,jin2022high}.

\begin{assumption}\label{ass.main_accuracy}
\change{When event $\mathcal{E}$ occurs (see Assumption~\ref{ass.good_merit_paramter}), let} 
\begin{equation*}
    \varepsilon >  \max\left\{\tfrac{\epsilon_g}{\eta},\omega_{10}\sqrt{\epsilon_f}\right\} \tfrac{\max\{\kappa_H,1\}}{\sqrt{\kappa_l\tau_{\min}}},
\end{equation*}
which is equivalent to $\varepsilon_{\Delta l} > \max\left\{\tfrac{\epsilon_g}{\eta},\omega_{10}\sqrt{\epsilon_f}\right\}$ 
by Remark~\ref{rem.stationary_measure}, 
where \change{$\{\kappa_H,\tau_{\min},\kappa_l\}\subset\R{}_{>0}$ are defined in Assumption~\ref{ass.H} and Lemmas~\ref{lem.deter_tau_lb} and~\ref{lem.Deltal_lb_det}, $(\epsilon_f,\epsilon_g)$ are intrinsic to Oracles~\ref{oracle.zero} and~\ref{oracle.first} conditioned on the event $\Ecal$,} $0 < \eta < 2(1-\theta) \min\left\{ \tfrac{1}{\eta_1 + \eta_2}, \tfrac{1}{\eta_3 + \eta_4}\right\}$
, and $\{\eta_1,\eta_2,\eta_3, \eta_4,\omega_{10}\}\subset\R{}_{>0}$ 
are defined in Table~\ref{tbl:constants} (Appendix~\ref{app.constants}) and $p \in \left(\tfrac 1 2,1\right]$.
\end{assumption}

Assumption~\ref{ass.main_accuracy} involves many constants and is indeed hard to parse. We make all constants explicit in order to show the exact dependence on the convergence neighborhood. That being said, what is important is that the lower bound of $\varepsilon$ is proportional to the bias in the gradient approximations and proportional to the square root of the noise level in the function approximations.

We are now ready to present the key lemma of this section. 
In Lemma~\ref{algo_behave}, we first define $(p,\tilde\alpha,h(\cdot))$, where $p \in \left(\tfrac 1 2,1\right]$ is a lower bound on the probability of a \textbf{\emph{true}} iteration conditioned on the past (before the stopping time), $\tilde\alpha \in\mathbb{R}_{>0}$ is the \textbf{\emph{large step}} threshold, and $h:\mathbb{R}_{>0}\to\mathbb{R}_{>0}$ is a monotonically increasing function (in $\alpha$) that bounds the potential progress made at any given iteration. Moreover, we prove five results that can be summarized as follows: $(i)$ lower bound (proportional to $\epsilon_f$) on the potential progress with step size $\tilde\alpha$; $(ii)$ conditioned on the past \change{and on event $\mathcal{E}$ 
(see Assumption~\ref{ass.good_merit_paramter})}, the next iteration is \textbf{\emph{true}} with probability at least $p$; $(iii)$ bound the potential progress made in any \textbf{\emph{true}} and \textbf{\emph{successful}} iterations \change{conditioned on 
event $\mathcal{E}$}; 
$(iv)$ \textbf{\emph{true}} iterations with \textbf{\emph{small step}} sizes are \textbf{\emph{successful}} \change{conditioned on 
event $\mathcal{E}$}; 
and, $(v)$ bound (proportional to $\epsilon_f$) the potential increase in $Z_k$ at any iteration $k$.

\begin{lemma}\label{algo_behave}
Suppose Assumptions~\ref{ass.good_merit_paramter} and~\ref{ass.main_accuracy} hold. For all $k< T_{\varepsilon_{\Delta l}}$, let
\begin{itemize}[leftmargin=0.2in]
    \item $p = 1 - \delta$ when the noise is bounded by $\epsilon_f$, and $p = 1-\delta -\exp\left(-\min\{\tfrac{u^2}{2\nu^2},\tfrac{u}{2b}\}\right)$ otherwise $($with \change{$u = \inf_{x\in\mathcal{X}} \{\epsilon_f - \E_{\Xi^0}[E(x,\Xi^0)|\Ecal]\}$}, 
    \item $h(\alpha) = \alpha\theta\varepsilon_{\Delta l}^2\min\left\{\tfrac{1-\eta\omega_4}{1+\epsilon_{\tau}\omega_1}, 1-\eta\omega_5, \tfrac{1}{1+\omega_2+\omega_3}\right\} $,
    \item and, $\tilde\alpha = \min\left\{\tfrac{1-\theta}{\omega_7}, \tfrac{\omega_8}{\bar\tau_{\min}L+\Gamma} \right\}$.

    The constants $\{\omega_1,\omega_2,\omega_3,\omega_4,\omega_5,\omega_7,\omega_8\} \subset\R{}_{>0}$ are defined in Table~\ref{tbl:constants} (Appendix~\ref{app.constants}).
\end{itemize}
Then,  the following results hold: 
\begin{itemize}
	\item[(i)] $h(\tilde\alpha) > \tfrac{4\bar\tau_{-1}}{p-1/2}\epsilon_f$. 
	\item[(ii)] $\Pmbb\left[I_k = 1|\change{\Fcal_{k-1}\cap \mathcal{E}}\right] \geq p$ with some $p\in \left(\tfrac{1}{2} + \tfrac{4\bar\tau_{-1}\epsilon_f}{h(\tilde\alpha)},1\right]$. 
    \item[(iii)] For any realization 
    of Algorithm~\ref{alg.sqp_line_search_practical} \change{for which event $\mathcal{E}$ occurs (see Assumption~\ref{ass.good_merit_paramter})}, if an iteration $k$ is \textbf{true} and \textbf{successful}, then $Z_{k+1} \leq Z_k - h(\alpha_k) + 4\bar\tau_{-1}\epsilon_f$. 
    \item[(iv)] For any realization 
    of Algorithm~\ref{alg.sqp_line_search_practical} \change{for which event $\mathcal{E}$ occurs (see Assumption~\ref{ass.good_merit_paramter})}, if an iteration $k$ is \textbf{true} and $\alpha_k \leq \tilde\alpha$, then the iteration $k$ is also \textbf{successful}. 
    \item[(v)] For all $k\in\mathbb{N}$ in any realization 
    of Algorithm~\ref{alg.sqp_line_search_practical}, $Z_{k+1} \leq Z_k + 2\bar\tau_{-1}\epsilon_f + \bar\tau_{-1}(e_k + e_k^+)$. 
\end{itemize}
\end{lemma}
\begin{proof} For brevity, the proof is deferred to Appendix~\ref{app.proof_lemma}.
\end{proof}

The next two lemmas will be used in the $\varepsilon$-complexity analysis that follows.

\begin{lemma}\label{lem:azumaheoffding}
\change{Suppose Assumption~\ref{ass.good_merit_paramter} holds.}
For all $ t\geq 1$ and any $\hat{p} \in [0, p)$, we have 
\begin{equation*}
    \P\left[\sum_{k=0}^{t-1} I_k<\hat pt \change{\Bigg| \mathcal{E}} \right] \leq e^{-\tfrac{(p-\hat{p})^2}{2p^2}t}\change{,}
\end{equation*}
\change{where event $\mathcal{E}$ is defined in Assumption~\ref{ass.good_merit_paramter}.}
\end{lemma}
\begin{proof}
The proof is the same as \cite[Lemma $3.1$]{jin2022high}.
\end{proof}

\begin{lemma}\label{lemma4}
\change{Suppose Assumption~\ref{ass.good_merit_paramter} holds.}
For any positive integer $t$ and any $\hat{p} \in \left(\tfrac12, 1\right]$, we have
\begin{equation*}
\P\left[T_{\varepsilon_{\Delta l}} > t,\; 	\sum_{k=0}^{t-1} I_k\geq \hat pt, \sum_{k=0}^{t-1} \Theta_kI_kU_k < \left(\hat{p}-\tfrac12\right)t - \tfrac{l}{2} \change{\Bigg| \mathcal{E}} \right]=0,
\end{equation*}
where $l = \max\left\{ - \tfrac{\ln \alpha_0 - \ln \tilde\alpha}{\ln\gamma}, 0\right\}$ \change{and event $\mathcal{E}$ is defined in Assumption~\ref{ass.good_merit_paramter}}. 
\end{lemma}
\begin{proof}
The proof is the same as \cite[Lemma $3.5$]{jin2022high}.
\end{proof}

We now present the main theorem; the iteration $\varepsilon$-complexity of Algorithm~\ref{alg.sqp_line_search_practical}.

\begin{theorem}\label{thm:subexp_noise}
\sloppy Suppose Assumptions~\ref{ass.prob}, \ref{ass.H}, \ref{ass.good_merit_paramter} and \ref{ass.main_accuracy} hold and that the conditions of Oracles~\ref{oracle.zero} and~\ref{oracle.first} are satisfied \change{conditioned on the 
event $\mathcal{E}$ (see Assumption~\ref{ass.good_merit_paramter})}. Then, for any $s \geq 0$, $\hat{p} \in \left( \tfrac12 + \tfrac{4\bar\tau_{-1}\epsilon_{f}+s}{h(\tilde{\alpha})}, p\right)$, and $t \geq \tfrac{R}{\hat{p} - \tfrac12 - \tfrac{{4\bar\tau_{-1}\epsilon_{f}}+s}{h(\tilde{\alpha})}}$, 
\begin{equation*}
    \P\left[T_{\varepsilon_{\Delta l}} \leq t \change{| \mathcal{E}}\right] \geq 1 - e^{-\tfrac{(p-\hat{p})^2}{2p^2}t} - e^{-\min\left\{\tfrac{s^2t}{2(2\bar\tau_{-1}\nu)^2},\tfrac{st}{2(2\bar\tau_{-1}b)}\right\}},
\end{equation*}
where $R = \tfrac{Z_0}{h(\tilde\alpha)}+\max\left\{ \tfrac{\ln \tilde\alpha -\ln \alpha_0 }{2\ln\gamma}, 0\right\}$, and $(p,\tilde\alpha,h(\cdot))$ 
are as defined in Lemma~\ref{algo_behave}.
\end{theorem}
\begin{proof}
By the law of total probability,
\begin{align*}
    \P\left[T_{\varepsilon_{\Delta l}} > t \change{|\mathcal{E}}\right] = 
&\P\underbrace{\left[T_{\varepsilon_{\Delta l}} > t,\; \tfrac{1}{t}\sum_{k=0}^{t-1}(2\bar\tau_{-1}\epsilon_f+\bar\tau_{-1}(E_k+E^+_k)) > 4\bar\tau_{-1}\epsilon_f + s \change{\Bigg| \mathcal{E}}\right]}_{\textbf{\emph{A}}}\\
    &\quad + 
\P\underbrace{\left[T_{\varepsilon_{\Delta l}} > t,\; \tfrac{1}{t}\sum_{k=0}^{t-1}(2\bar\tau_{-1}\epsilon_f+\bar\tau_{-1}(E_k+E^+_k)) \leq 4\bar\tau_{-1}\epsilon_f + s\change{\Bigg| \mathcal{E}}\right]}_{\textbf{\emph{B}}}.
\end{align*}
First we bound $\P[\textbf{\emph{A}}]$. \change{Conditioned on the event $\mathcal{E}$ (see Assumption~\ref{ass.good_merit_paramter}), for each iteration $k$, since $E_k$ and $E_k^+$ satisfy the one-sided sub-exponential bound \eqref{eq:zero_order} with parameters $(\nu,b)$, one can show that $\bar\tau_{-1}(E_k+E^+_k)$ satisfies \eqref{eq:zero_order} with parameters $(2\bar\tau_{-1}\nu, 2\bar\tau_{-1}b)$.} 
Moreover, since \change{$\mathbb{E}[\bar\tau_{-1}(E_k+E^+_k)|\Ecal]$ is bounded} by $2\bar\tau_{-1}\epsilon_f$,
applying the one-sided Bernstein's inequality, for any $s \geq 0$
\begin{align*}
    \P[\textbf{\emph{A}}] 
         \leq\ \P\left[\tfrac{1}{t}\sum_{k=0}^{t-1}\bar\tau_{-1}(E_k+E^+_k)>  2\bar\tau_{-1}\epsilon_f + s \change{\Bigg| \mathcal{E}}\right]\leq e^{-\min\left\{\tfrac{s^2t}{2(2\bar\tau_{-1}\nu)^2},\tfrac{st}{2(2\bar\tau_{-1}b)}\right\}}.
\end{align*}
Let $l = \max\left\{ - \tfrac{\ln \alpha_0 - \ln \tilde{\alpha}}{\ln\gamma}, 0\right\}$.
To bound $\P[B]$ we apply the law of total probability,
\change{
\begin{align*}
    \P[\textbf{\emph{B}}] &= \P\underbrace{\left[\sum_{k=0}^{t-1} \Theta_kI_kU_k < \left(\hat{p}-\tfrac12\right)t - \tfrac{l}{2},\ \textbf{\emph{B}}\right]}_{\textbf{\emph{B}}_1} +\P\underbrace{\left[\sum_{k=0}^{t-1} \Theta_kI_kU_k \geq \left(\hat{p}-\tfrac12\right)t - \tfrac{l}{2}, \ \textbf{\emph{B}}\right]}_{\textbf{\emph{B}}_2}.
\end{align*}}

\noindent We first show that $\P[\textbf{\emph{B}}_2]=0$. By Lemma \ref{algo_behave}, \change{in any realization of Algorithm~\ref{alg.sqp_line_search_practical} for which event $\mathcal{E}$ occurs (see Assumption~\ref{ass.good_merit_paramter}),} for any iteration $k < T_{\varepsilon_{\Delta l}}$, it follows that $Z_{k+1} \leq Z_k - h(\tilde{\alpha}) + 2\bar\tau_{-1}\epsilon_f+\bar\tau_{-1}(E_k+E^+_k)\leq Z_k - h(\tilde{\alpha}) + 4\bar\tau_{-1}\epsilon_f$ if $U_kI_k\Theta_k = 1$,  and $Z_{k+1} \leq Z_k + 2\bar\tau_{-1}\epsilon_f+\bar\tau_{-1}(E_k+E^+_k)$ if $U_kI_k\Theta_k = 0$. 
By \change{Assumption~\ref{ass.good_merit_paramter}, $\E[E_k|\mathcal{E}]$ and $\E[E_k^+|\mathcal{E}]$} are bounded above by $\epsilon_f$ for all $k$. \change{Conditioned on event $\mathcal{E}$ (see Assumption~\ref{ass.good_merit_paramter}), the} event $T_{\varepsilon_{\Delta l}} > t$ implies that $Z_t>0$ (since $Z_t=0$ can only happen when  $T_{\varepsilon_{\Delta l}} \leq t$ by the proof of Lemma~\ref{lem_progress}). 
This together with $ \tfrac{1}{t}\sum_{k=0}^{t-1}(2\bar\tau_{-1}\epsilon_f+\bar\tau_{-1}(E_k+E^+_k)) \leq 4\bar\tau_{-1}\epsilon_f + s$ in turn implies  the event $\sum_{k=0}^{t-1} \Theta_kI_kU_k < \left(\hat{p}-\tfrac12\right)t - \tfrac{l}{2}$. To see this, assume that  $\sum_{k=0}^{t-1} \Theta_kI_kU_k \geq \left(\hat{p}-\tfrac12\right)t - \tfrac{l}{2}$, then 
\begin{align*}
    Z_t &\leq Z_0 - \left[\left(\left(\hat{p} - \tfrac12\right) t - \tfrac{l}{2}\right) h(\tilde{\alpha}) - \sum_{k=0}^{t-1}(2\bar\tau_{-1}\epsilon_f+\bar\tau_{-1}(E_k+E^+_k))\right] \\
    &\leq Z_0 - \left(\left(\hat{p} - \tfrac12\right) t - \tfrac{l}{2}\right) h(\tilde{\alpha}) + t(4\bar\tau_{-1}\epsilon_f+s) \\
    &= Z_0 -  \left(\left(\hat{p}-\tfrac12\right)h(\tilde\alpha)-(4\bar\tau_{-1}\epsilon_f+s)\right)t + \tfrac{l}{2} h(\tilde\alpha) \leq 0.
\end{align*}
The last inequality above is due to the  assumption that
$\hat{p} > \tfrac12 + \tfrac{4\bar\tau_{-1}\epsilon_f+s}{h(\tilde{\alpha})}$ and     $t \geq \tfrac{R}{\hat{p} - \tfrac12 - \tfrac{4\bar\tau_{-1}\epsilon_f+s}{h(\tilde{\alpha})}}$. Hence, $\P[\textbf{\emph{B}}_2]=0$. We now bound $\P[\textbf{\emph{B}}_1]$; by Lemmas~\ref{lem:azumaheoffding} and~\ref{lemma4},
\begin{align*}
    \P[\textbf{\emph{B}}_1]\leq \ & \P\left[T_{\varepsilon_{\Delta l}} > t,\; \sum_{k=0}^{t-1} \Theta_kI_kU_k < \left(\hat{p}-\tfrac12\right)t - \tfrac{l}{2} \change{\Bigg| \mathcal{E}}\right] \\
    = \ & \P\left[T_{\varepsilon_{\Delta l}} > t,\; \sum_{k=0}^{t-1} \Theta_kI_kU_k < \left(\hat{p}-\tfrac12\right)t - \tfrac{l}{2}, \sum_{k=0}^{t-1} I_k<\hat pt \change{\Bigg| \mathcal{E}} \right] \\
    & + \P\left[T_{\varepsilon_{\Delta l}} > t,\; \sum_{k=0}^{t-1} \Theta_kI_kU_k < \left(\hat{p}-\tfrac12\right)t - \tfrac{l}{2}, \sum_{k=0}^{t-1} I_k\geq \hat pt \change{\Bigg| \mathcal{E}} \right]\\
    \leq\ & \P\left[\sum_{k=0}^{t-1} I_k<\hat pt \change{\Bigg| \mathcal{E}} \right] + \P\left[T_{\varepsilon_{\Delta l}} > t,\; \sum_{k=0}^{t-1} \Theta_kI_kU_k < \left(\hat{p}-\tfrac12\right)t - \tfrac{l}{2}, \sum_{k=0}^{t-1} I_k\geq \hat pt \change{\Bigg| \mathcal{E}} \right]\\
    \leq\ & e^{-\tfrac{(p-\hat{p})^2}{2p^2}t}+0=e^{-\tfrac{(p-\hat{p})^2}{2p^2}t}. 
\end{align*}
Combining $\P[\textbf{\emph{A}}]$ and $\P[\textbf{\emph{B}}]$ completes the proof. 
\end{proof}

\begin{corollary}\label{cor_main}
Under the conditions of Theorem~\ref{thm:subexp_noise}, for any $s \geq 0$, $\hat{p} \in \left( \tfrac12 + \tfrac{4\bar\tau_{-1}\epsilon_{f}+s}{\tilde\alpha\theta \omega_p\varepsilon_{\Delta l}^2}, p\right)$ and $t \geq \tfrac{\hat R}{\hat{p} - \tfrac12 - \tfrac{{4\bar\tau_{-1}\epsilon_{f}}+s}{\tilde\alpha\theta \omega_p\varepsilon_{\Delta l}^2}}$, 
\begin{equation}\label{eq.success_events}
    \P\left[T_{\varepsilon_{\Delta l}} \leq t \change{| \mathcal{E}}\right] \geq 1 - e^{-\tfrac{(p-\hat{p})^2}{2p^2}t} - e^{-\min\left\{\tfrac{s^2t}{2(2\bar\tau_{-1}\nu)^2},\tfrac{st}{2(2\bar\tau_{-1}b)}\right\}},
\end{equation}
where 
$\hat{R} = \tfrac{\phi(x_0,\bar\tau_{-1}) - \phi_{\min} -(\bar\tau_{-1} - \bar\tau_{\min})f_{\inf}}{\tilde\alpha\theta \omega_p\varepsilon_{\Delta l}^2}+\max\left\{ \tfrac{\ln \tilde{\alpha}-\ln \alpha_0 }{2\ln\gamma}, 0\right\}$, equivalently, by Remark~\ref{rem.stationary_measure}, $ \hat{R}= \tfrac{\max\{\kappa_H^2,1\}}{\kappa_l\tau_{\min}}\tfrac{\phi(x_0,\bar\tau_{-1}) - \phi_{\min} -(\bar\tau_{-1} - \bar\tau_{\min})f_{\inf}}{\tilde\alpha\theta \omega_p\varepsilon^2} + \max\left\{ \tfrac{\ln \tilde{\alpha}-\ln \alpha_0 }{2\ln\gamma}, 0\right\}$,
\change{event $\mathcal{E}$ is defined in Assumption~\ref{ass.good_merit_paramter},}
$\omega_p=\min\left\{\tfrac{1-\eta\omega_4}{1+\epsilon_{\tau}\omega_1}, 1-\eta\omega_5, \tfrac{1}{1+\omega_2+\omega_3}\right\}$, and the rest of the constants are defined in Table~\ref{tbl:constants} (Appendix~\ref{app.constants}).
\end{corollary}

\begin{remark}\label{rem.best_complexity}
We make a few remarks about the main theoretical results of the paper (Theorem~\ref{thm:subexp_noise} and Corollary~\ref{cor_main}).
\begin{itemize}[wide]
    \item (Iteration $\varepsilon$-complexity) By Definition~\ref{def.iter_term} (and Remark~\ref{rem.stationary_measure}) and  Corollary~\ref{cor_main}, we conclude that \change{conditioned on 
    event $\mathcal{E}$ (see Assumption~\ref{ass.good_merit_paramter})}, with overwhelmingly high probability, the iteration $\varepsilon$-complexity of Algorithm~\ref{alg.sqp_line_search_practical} to generate a primal-dual iterate $(x_k,y_k)\in\R{n}\times\R{m}$ that satisfies $\max\{\|\nabla f_k + J_k^Ty_k\|,\sqrt{\|c_k\|}\} \leq \varepsilon$  is $\mathcal{O}(\varepsilon^{-2})$. This iteration $\varepsilon$-complexity is of the same order in terms of the dependence on $\varepsilon$ as the iteration complexity that can be derived for the deterministic counterpart (e.g., ~\cite{CurtONeiRobi21}), with the additional restriction that  $\varepsilon$ is bounded away from zero (Assumption~\ref{ass.main_accuracy}) due to the noise and bias in the oracles (Oracles~\ref{oracle.zero} and \ref{oracle.first}).
    \item (Almost-sure convergence) We note that \change{under 
    event $\mathcal{E}$ occurring (see Assumption~\ref{ass.good_merit_paramter}),} Algorithm~\ref{alg.sqp_line_search_practical} finds an 
    $\varepsilon$-stationary iterate in a finite number of iterations with probability 1, i.e., $\Pmbb[\cap_{k=1}^{\infty}\cup_{t=k}^{\infty} \left(T_{\varepsilon_{\Delta l}} > t\right)\change{|\mathcal{E}}] = 0$. This is a direct consequence of the Borel–Cantelli lemma, since it follows from \eqref{eq.success_events} that the probability of failure events is summable, i.e., $\sum_{t=1}^{\infty} \Pmbb[T_{\varepsilon_{\Delta l}} > t \change{|\mathcal{E}}] = \sum_{t=1}^{\infty} \left(1 - \Pmbb[T_{\varepsilon_{\Delta l}} \leq t \change{|\mathcal{E}}]\right) < \infty$.
    \item (Unconstrained setting) The high probability $\varepsilon$-complexity bound in this paper is a generalization of the unconstrained version (e.g., \cite{jin2022high}). In the unconstrained setting, the parameters reduce to $\sigma=0$, $\zeta_y = 0$, 
    $\omega_1=1$, $\Gamma=0$, $\zeta=1$, $\kappa_H = 1$, $\kappa_l=1$, $\epsilon_{\tau}=0$, and $\bar\tau_k = 1$ for all $k\in\N{}$. While using these values in the result of Corollary~\ref{cor_main} does not exactly recover the result from the unconstrained setting (e.g., \cite{jin2022high}), the order of the results is the same in terms of the dependence on $\varepsilon$. The gap is due to the additional complexity that arises due to the adaptive merit parameter. 
    We emphasize that although there is a constant difference in the results as compared to \cite{jin2022high}, our algorithm recovers the complexity bound of the deterministic variant \cite{CurtONeiRobi21}.
\end{itemize}
\end{remark}

\section{Numerical Results}\label{sec.numerical}
In this section, we present numerical results for our proposed algorithm on standard equality constrained nonlinear optimization problems. The goal of the numerical experiments is to investigate the efficiency and robustness of the \SSSPQ{} algorithm across a diverse set of test problems with different levels of noise in the objective function and gradient evaluations. All experiments were conducted in MATLAB. 
Before we present the numerical results, we describe the test problems, implementation details, and evaluation metrics.

\subsection{Test Problems}
We ran the numerical experiments on a subset of the equality constrained optimization problems from the CUTEst collection \cite{gould2015cutest}. We selected the problems that satisfy the following criteria: $(i)$ the objective function is \emph{not} a constant function, $(ii)$ the total number of variables and constraints are not larger than $10^3$, and $(iii)$ the singular values of Jacobians of the constraints at all iterates in all runs were greater than $10^{-8}$.  This resulted in 35 test problems of various dimensions.

We considered noisy (noisy objective function and gradient evaluations) versions of the 35 CUTEst problems. Specifically, whenever an objective function or objective gradient evaluation was required, approximations, $\bar{f}(x;\xi) = \mathcal{N}\left(f(x),\epsilon_{f,N}^2\right)$ and $\bar{g}(x;\xi^\prime) = \mathcal{N}\left(\nabla f(x),\tfrac{\epsilon_{g,N}^2}{n}I\right)$, respectively, were utilized. We considered 4 different noise levels in the objective function and gradient evaluations, dictated by the constants $\epsilon_{f,N}\in\left\{0,10^{-4},10^{-2},10^{-1}\right\}$ and $\epsilon_{g,N}\in\left\{0,10^{-4},10^{-2},10^{-1}\right\}$, respectively. Each CUTEst problem has a unique initial starting point, which was used as the starting point of all runs of all algorithms. Moreover, for each selected tuple of noise levels $(\epsilon_{f,N},\epsilon_{g,N})\in\left\{0,10^{-4},10^{-2},10^{-1}\right\} \times \left\{10^{-4},10^{-2},10^{-1}\right\} \cup \{0\}\times \{0\}$, where appropriate, we ran each problem with five different random seeds.

\subsection{Implementation Details}
We compared \SSSPQ{} (Algorithm~\ref{alg.sqp_line_search_practical}) to the adaptive stochastic SQP algorithm proposed in \cite{BeraCurtRobiZhou21} (which we call \ASSPQ{}) on the previously described noisy CUTEst problems. 
We set user-defined parameters for \SSSPQ{} as follows: $\epsilon_f = \epsilon_{f,N},\epsilon_g = \epsilon_{g,N}$, $\epsilon_{\tau} = 10^{-2}$, $\bar\tau_{-1} = \sigma = 0.1$, $\gamma = 0.5$, $\theta = 10^{-4}$, $\alpha_0 = \alpha_{\max} = 1$, and $H_k = I$ for all $k\in\mathbb{N}$. For \ASSPQ{} \cite{BeraCurtRobiZhou21} we set the parameters as follows (this parameter selection was guided by the choice of parameters in~\cite{BeraCurtRobiZhou21}): $\bar\tau_{-1} = \sigma = 0.1$, $\bar\xi_{-1} = 1$, $\epsilon = 10^{-2}$, $\theta = 10^4$, $H_k = I$ and $\beta_k = 1$ for all $k\in\mathbb{N}$. The \ASSPQ{} step size rule  requires knowledge (or estimates) of the Lipschitz constants $L$ and $\Gamma$. To this end, we estimated these constants using gradient differences near the initial point, and set $L_k = L$ and $\Gamma_k = \Gamma$ for all $k\in\mathbb{N}$. 
We note that while the analysis of the  \SSSPQ{} algorithm  requires that the condition of Oracles~\ref{oracle.first} hold, such conditions are not enforced or checked, 
and rather in each experiment, the algorithms were given random gradient estimates with the same, fixed, pre-specified accuracy (as described above). 
That being said, \SSSPQ{} and \ASSPQ{} differ in that the former requires estimates of the objective function 
whereas the latter  does not (\ASSPQ{} is an objective-function-free method). Specifically,  \SSSPQ{} requires 2 function evaluations and 1 gradient evaluation per iteration and \ASSPQ{} only requires a single gradient evaluation per iteration. We discuss this further when presenting the numerical results.

\subsection{Termination Conditions and Evaluation Metrics}
In all of our experiments, results are given in terms of infeasibility ($\|c(x_k)\|_{\infty}$) and stationarity (KKT) ($\max\{\|c(x_k)\|_{\infty}, \min_{y\in\R{m}}\|\nabla f(x_k) + \nabla c(x_k)y\|_{\infty}\}$) with respect to different evaluation metrics (iterations and work). All algorithms were run with a budget of  $10^3$ iterations, and only terminated a run early if an approximate stationary point was found, which is defined as $x_*\in\R{n}$ such that $\|c(x_*)\|_{\infty} \leq 10^{-6}$ and $\min_{y\in\R{m}}\|\nabla f(x_*) + \nabla c(x_*)y\|_{\infty} \leq 10^{-4}$. 

We present results in the form of performance profiles with respect to iterations and work (defined as the total number of function and gradient evaluations, or equivalently, the total number of probabilitic oracle calls). At every iteration, \SSSPQ{} requires both Oracles~\ref{oracle.zero} and~\ref{oracle.first} while \ASSPQ{} only requires Oracle~\ref{oracle.first}, which means \SSSPQ{} is more expensive per iteration. Moreover, we 
use the convergence metric as described in~\cite{more2009benchmarking}, i.e., 
$m(x_0) - m(x) \geq (1-\epsilon_{pp})(m(x_0)-m_b)$, 
where $m(x)$ is either $\|c(x)\|_{\infty}$ (infeasibility) or $\max\{\|c(x)\|_{\infty}, \min_{y\in\R{m}}\|\nabla f(x) + \nabla c(x)y\|_{\infty}\}$ (stationarity (KKT)), $x_0$ is the initial iterate, and $m_b$ is the best value of the metric found by any algorithm for a given problem instance within the budget, and $\epsilon_{pp} \in (0,1)$ is the tolerance. For all experiments presented, we chose $\epsilon_{pp} = 10^{-3}$.

\subsection{Noisy Gradients, Exact Functions ($\epsilon_f = 0$)} In our first set of experiments, we consider problems with exact objective function evaluations and noisy objective gradient evaluations and compare \SSSPQ{} and \ASSPQ{}. The goal of this experiment is to show the effect of noise in the gradient and the advantages of using (exact) function values. Each row in Figure~\ref{fig.1} shows performance profiles for a different noise level in the gradient (bottom row, highest noise level) and each column shows a different evaluation metric. Starting from the noise-less benchmark case ($\epsilon_f = 0$ and $\epsilon_g = 0$, the first row of Figure~\ref{fig.1}), it is clear that the performance of the methods in both infeasibility error and KKT error is similar with a slight advantage in effectiveness (total problems that can be solved) for \SSSPQ{} in terms of KKT error. As the noise in the gradient is increased, the gap between the performance of the two methods (in terms of all metrics) increases favoring \SSSPQ{}. This is not surprising as \SSSPQ{} uses additional information (exact function values). These results highlight the effect reliable function information can have on the performance of the methods.

\begin{figure}[ht]
   \centering
  \includegraphics[width=0.24\textwidth,clip=true,trim=30 180 50 180]{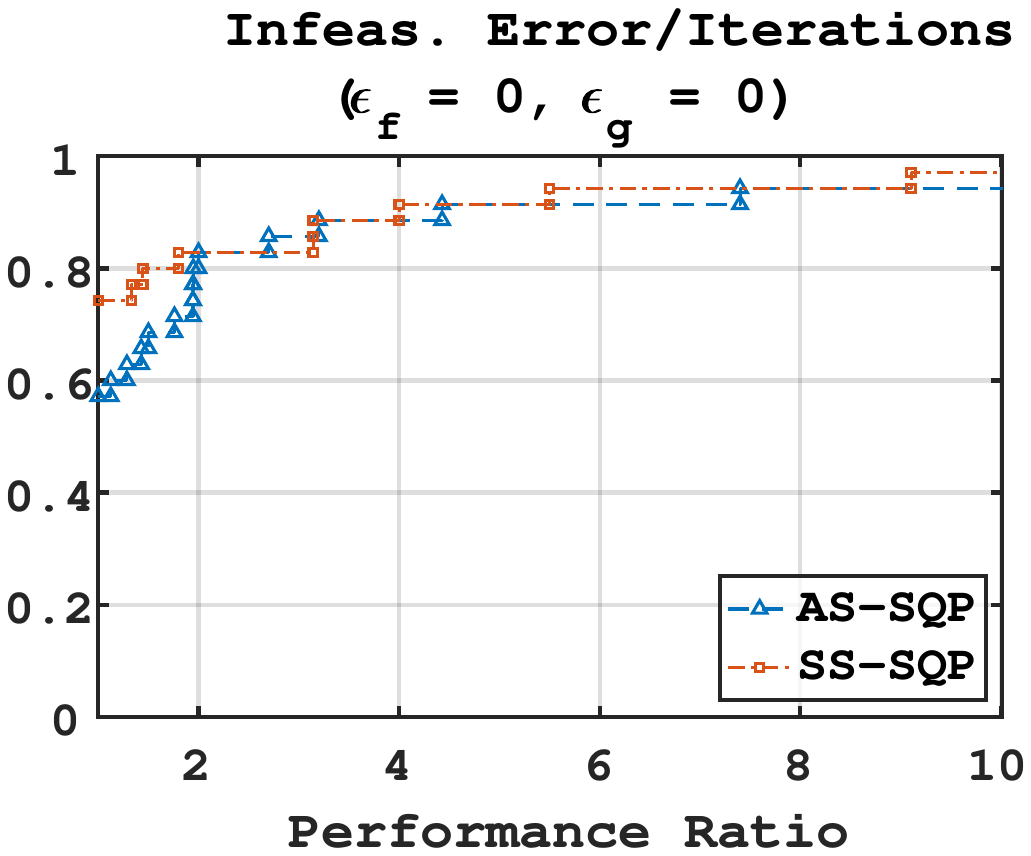}
  \includegraphics[width=0.24\textwidth,clip=true,trim=30 180 50 180]{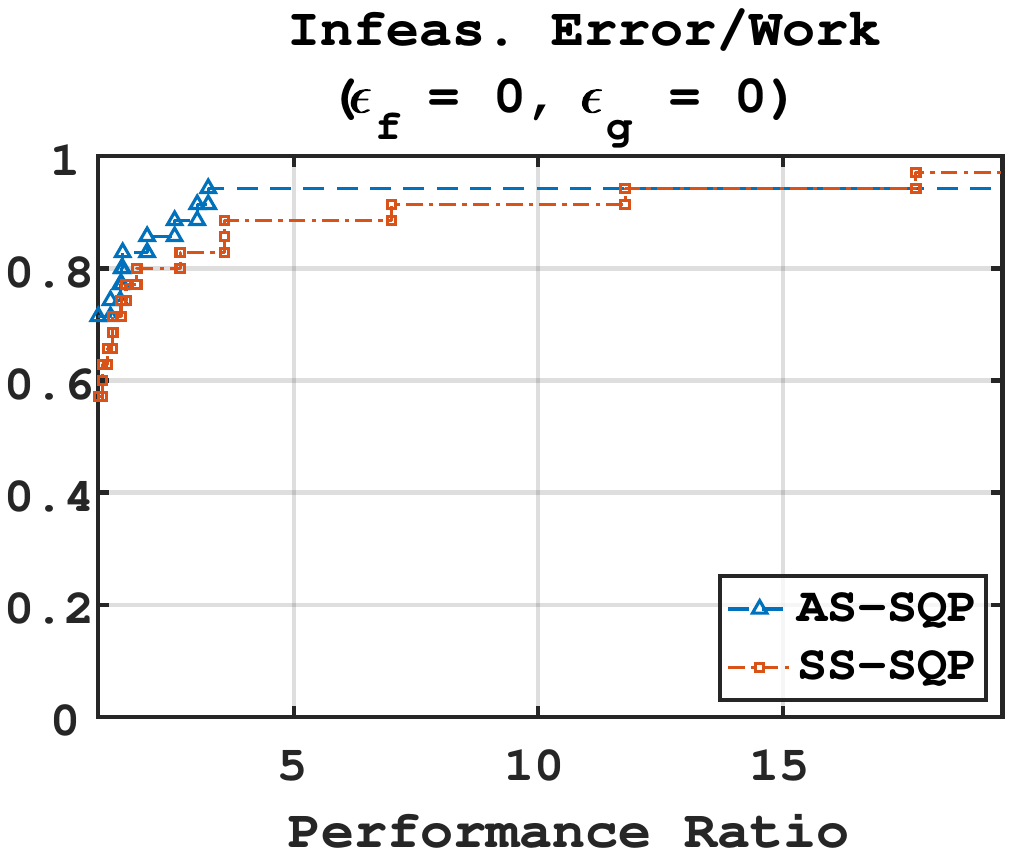}
  \includegraphics[width=0.24\textwidth,clip=true,trim=30 180 50 180]{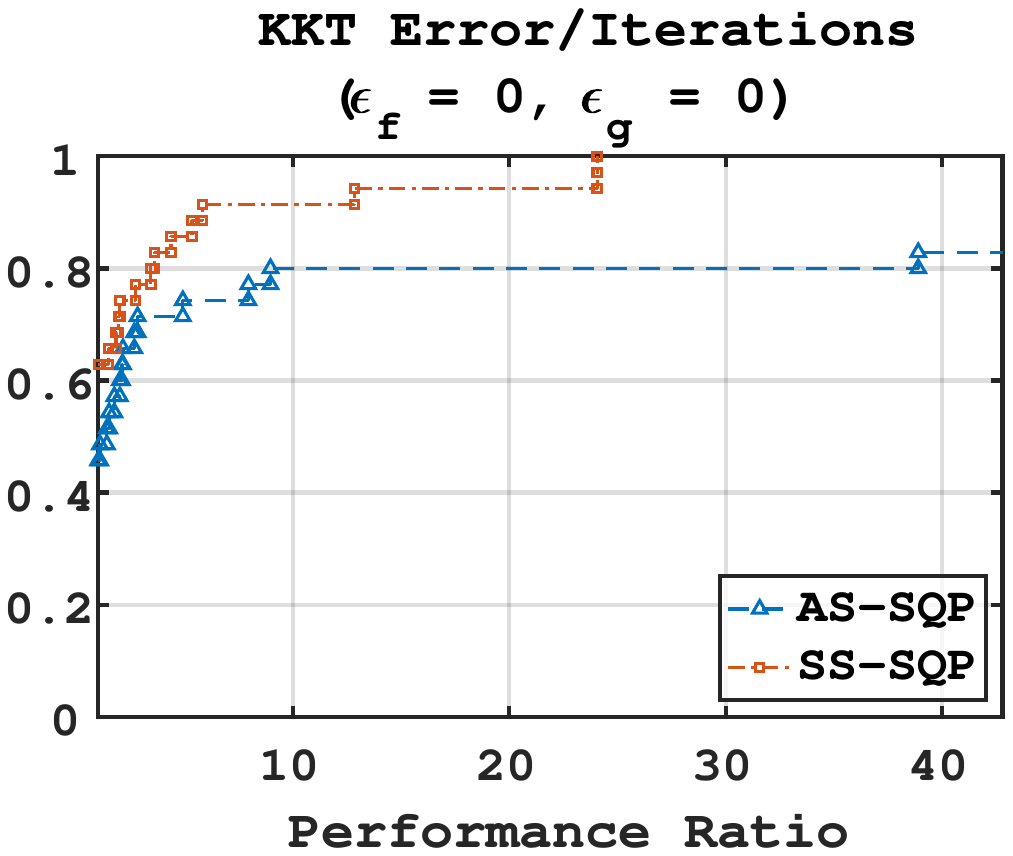}
  \includegraphics[width=0.24\textwidth,clip=true,trim=30 180 50 180]{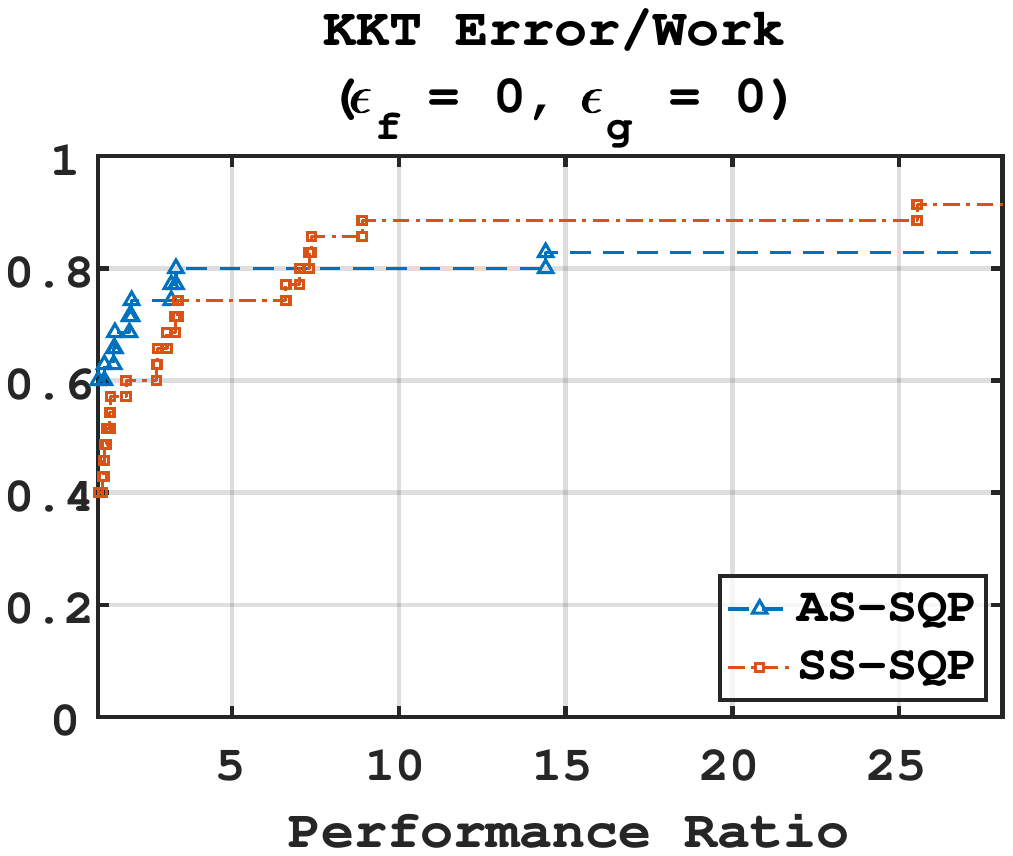}
  
  \includegraphics[width=0.24\textwidth,clip=true,trim=30 180 50 180]{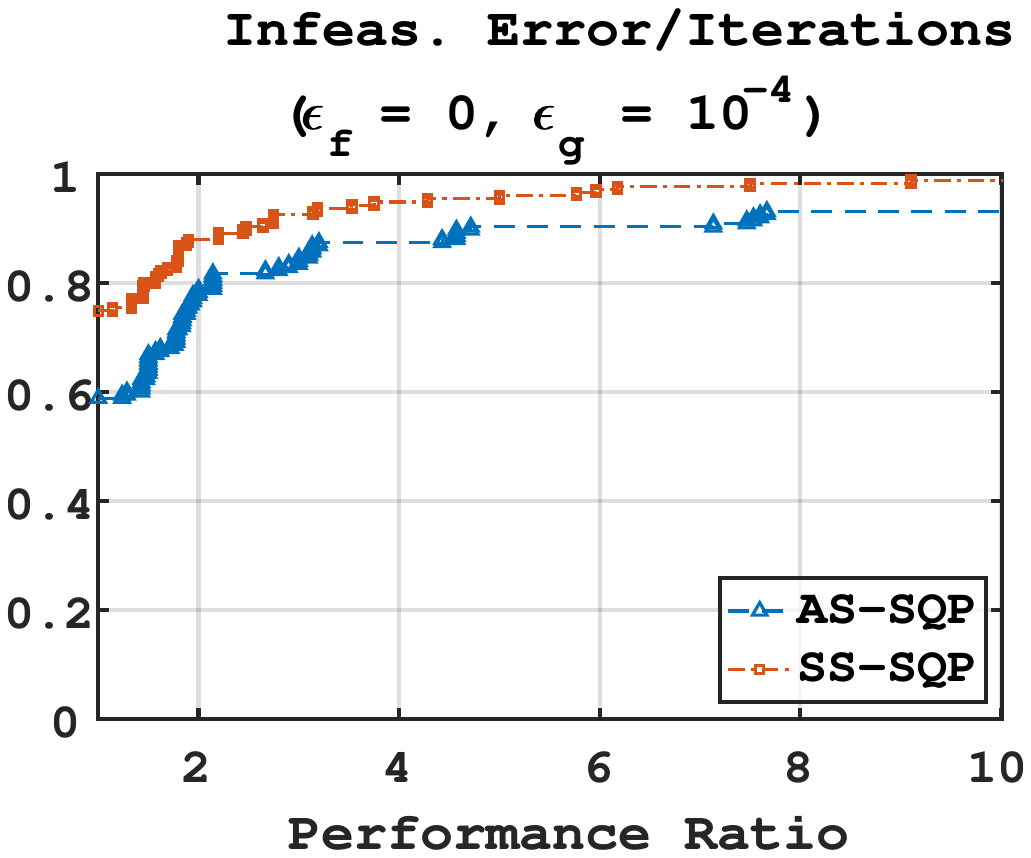}
  \includegraphics[width=0.24\textwidth,clip=true,trim=30 180 50 180]{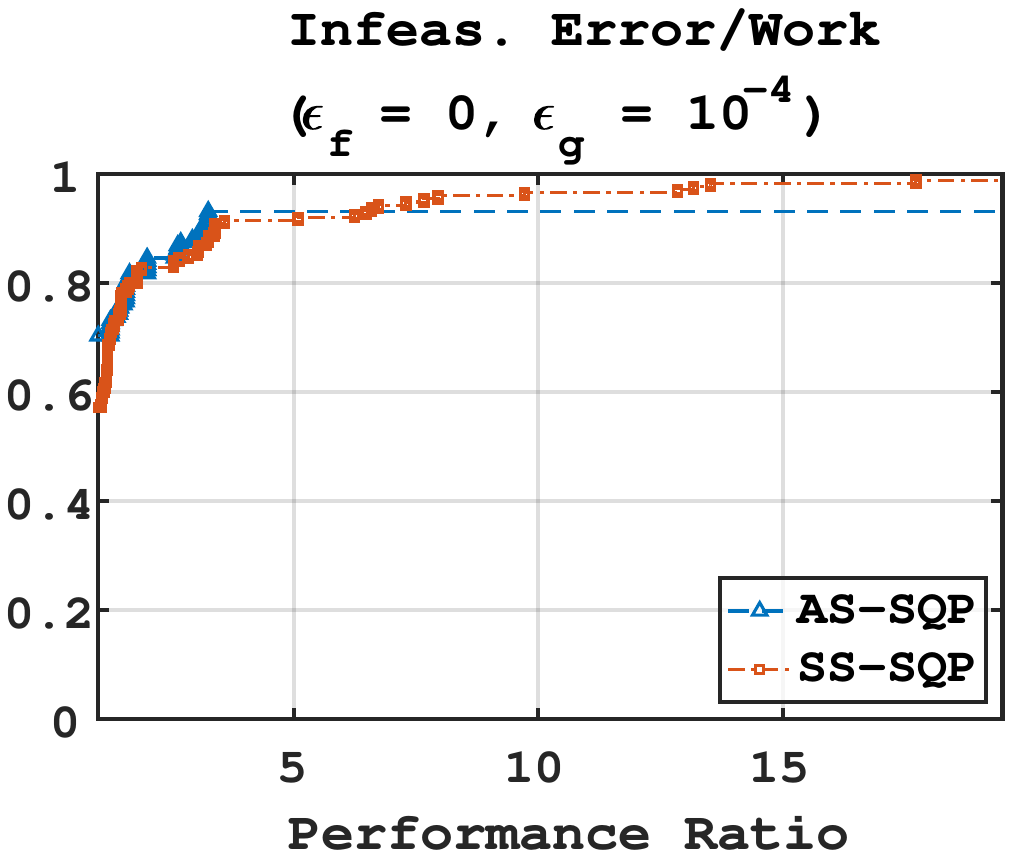}
  \includegraphics[width=0.24\textwidth,clip=true,trim=30 180 50 180]{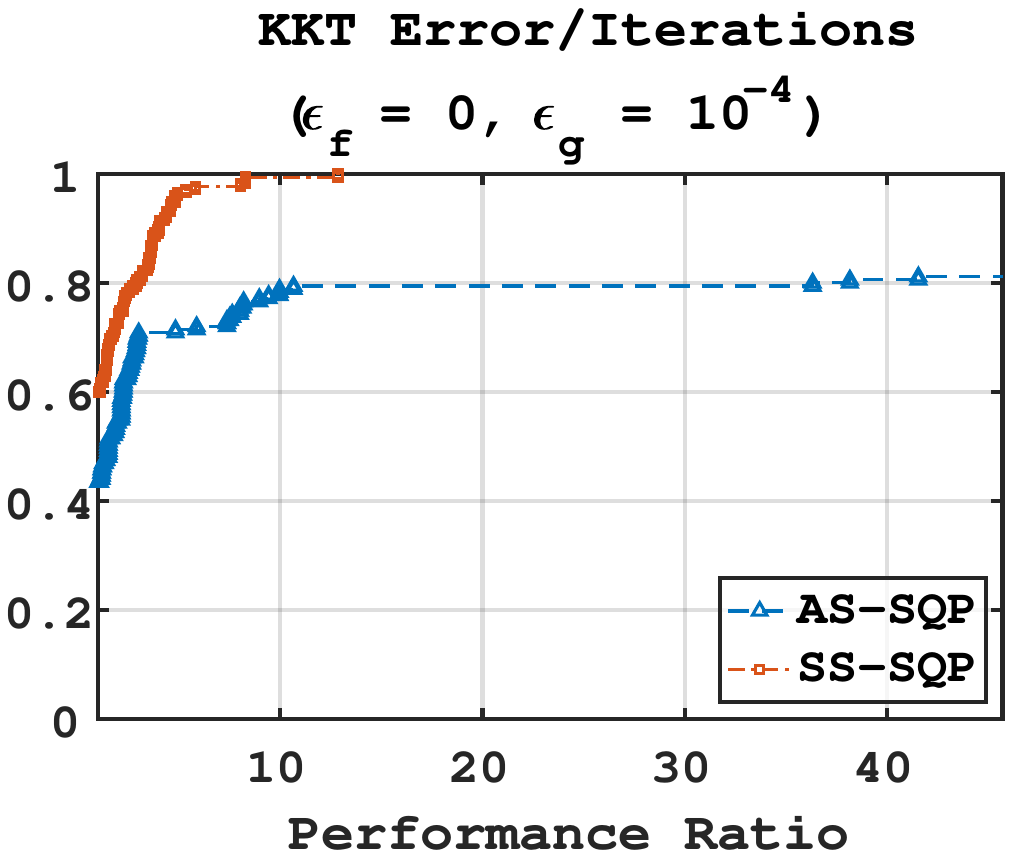}
  \includegraphics[width=0.24\textwidth,clip=true,trim=30 180 50 180]{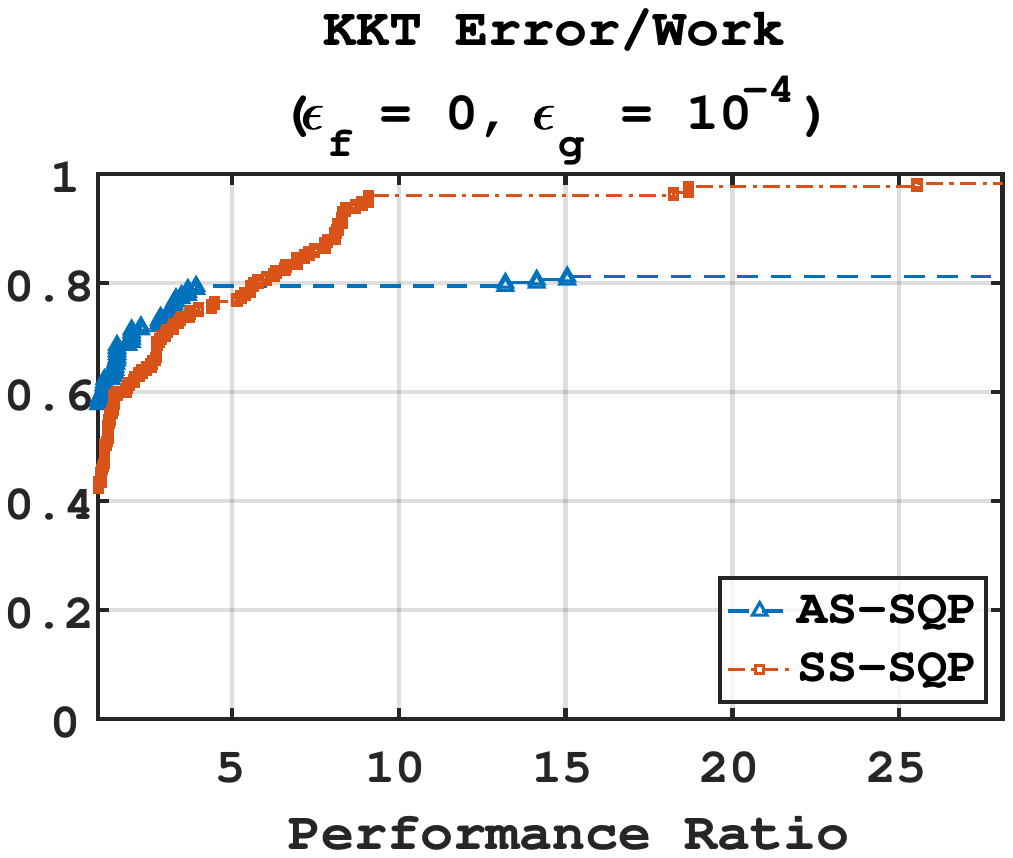}
  
  \includegraphics[width=0.24\textwidth,clip=true,trim=30 180 50 180]{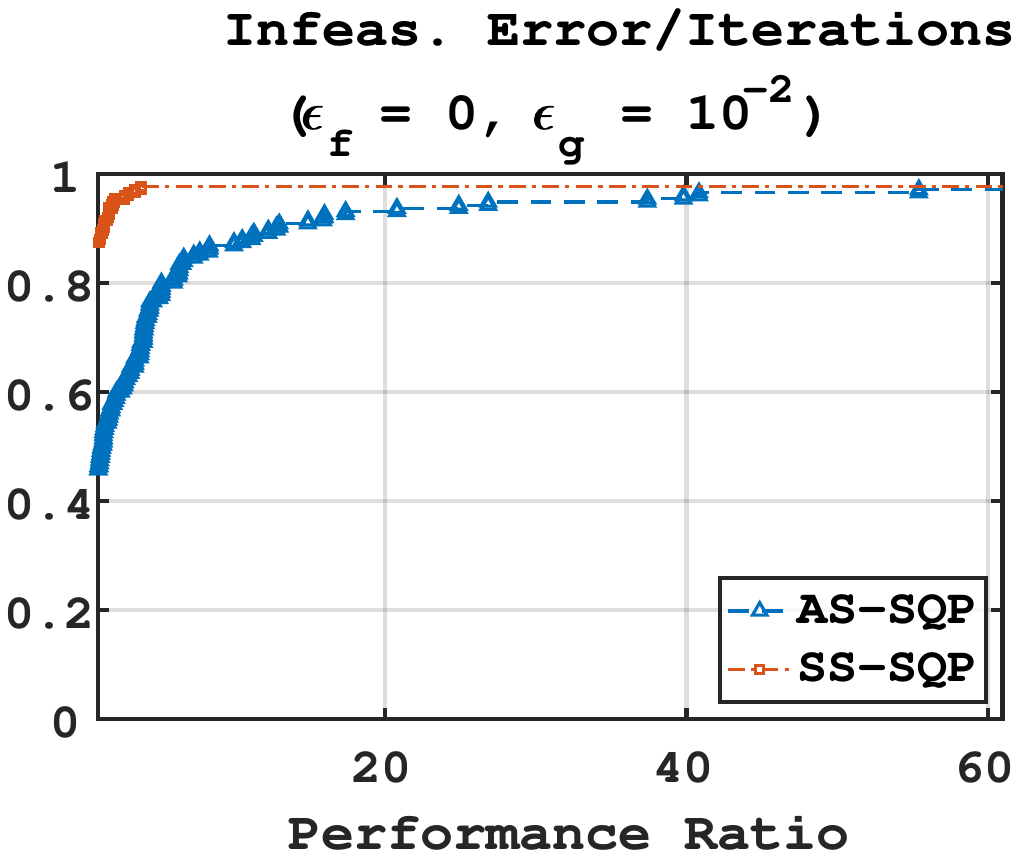}
  \includegraphics[width=0.24\textwidth,clip=true,trim=30 180 50 180]{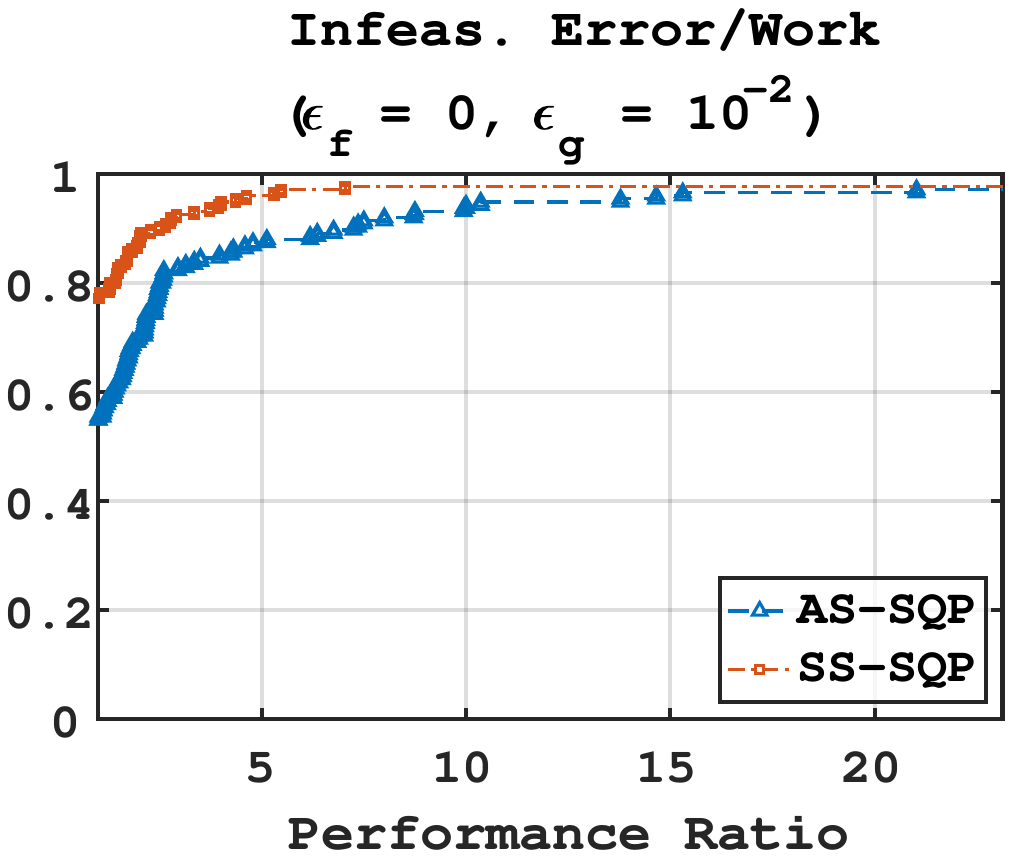}
  \includegraphics[width=0.24\textwidth,clip=true,trim=30 180 50 180]{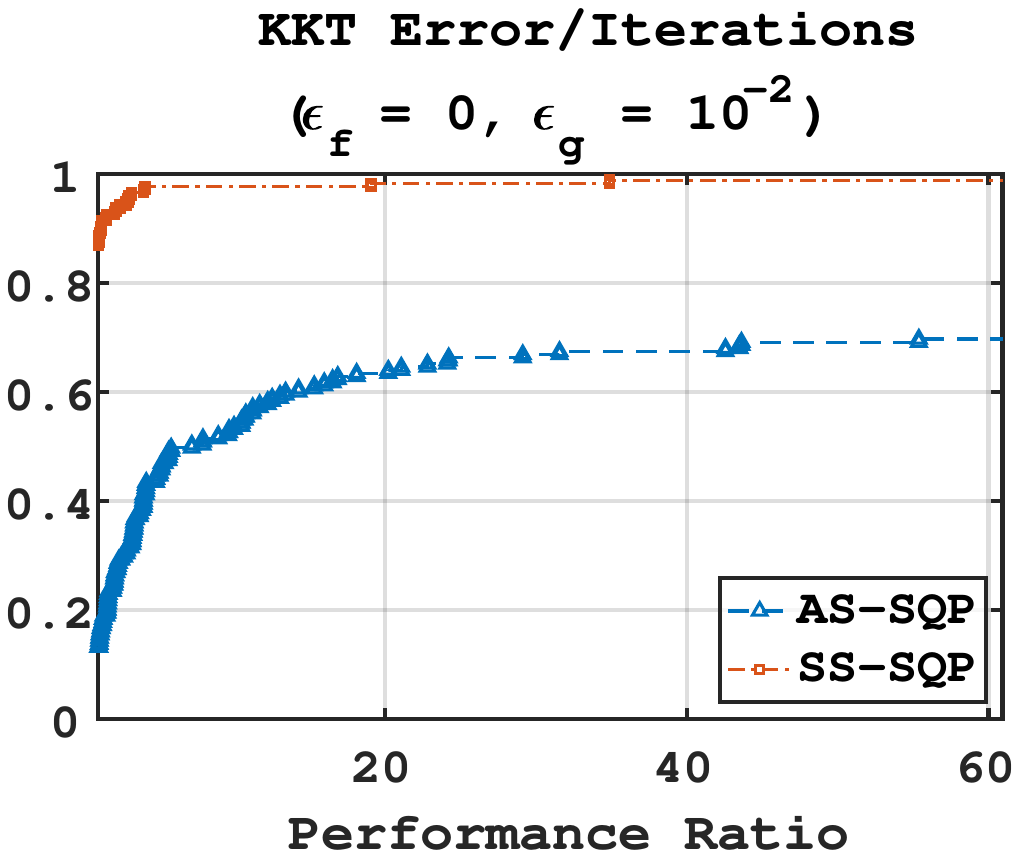}
  \includegraphics[width=0.24\textwidth,clip=true,trim=30 180 50 180]{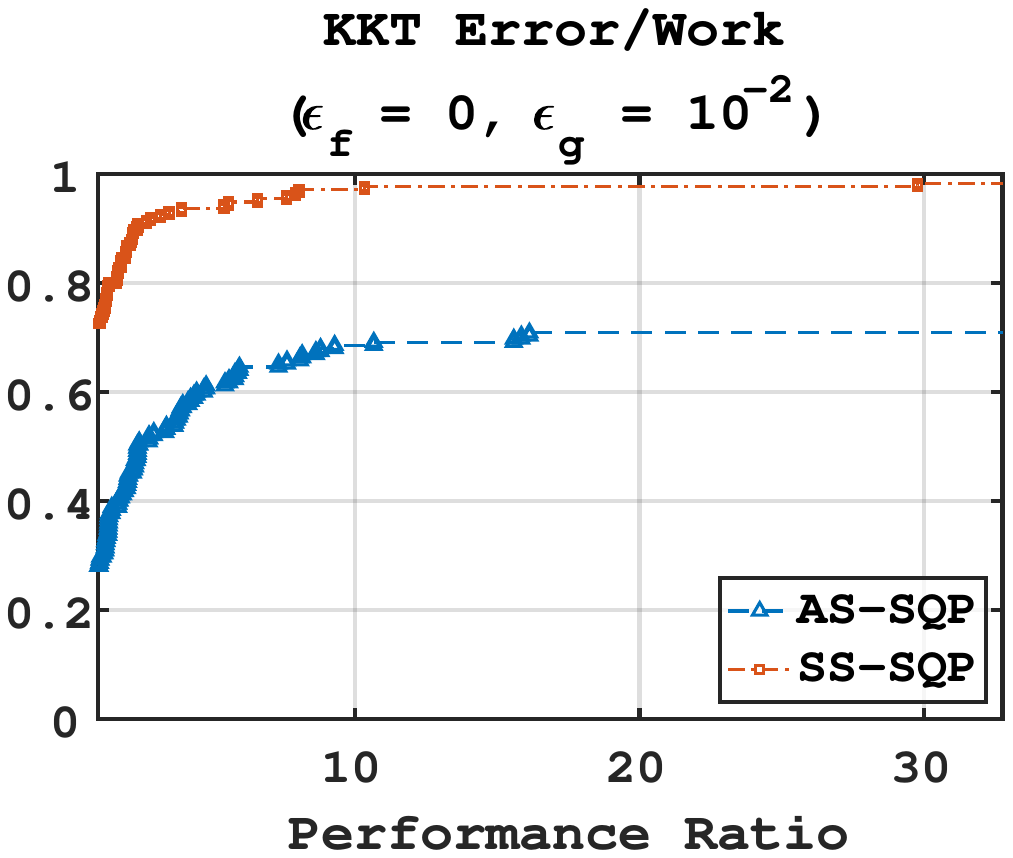}
  
  \includegraphics[width=0.24\textwidth,clip=true,trim=30 180 50 180]{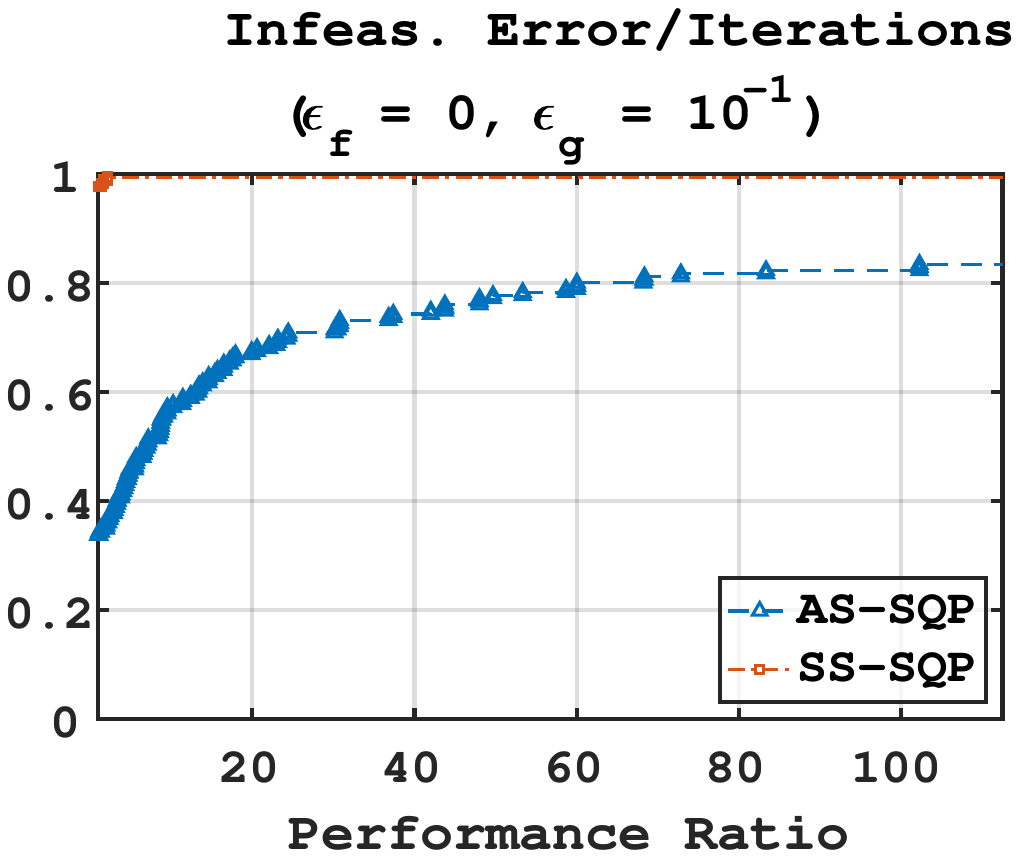}
  \includegraphics[width=0.24\textwidth,clip=true,trim=30 180 50 180]{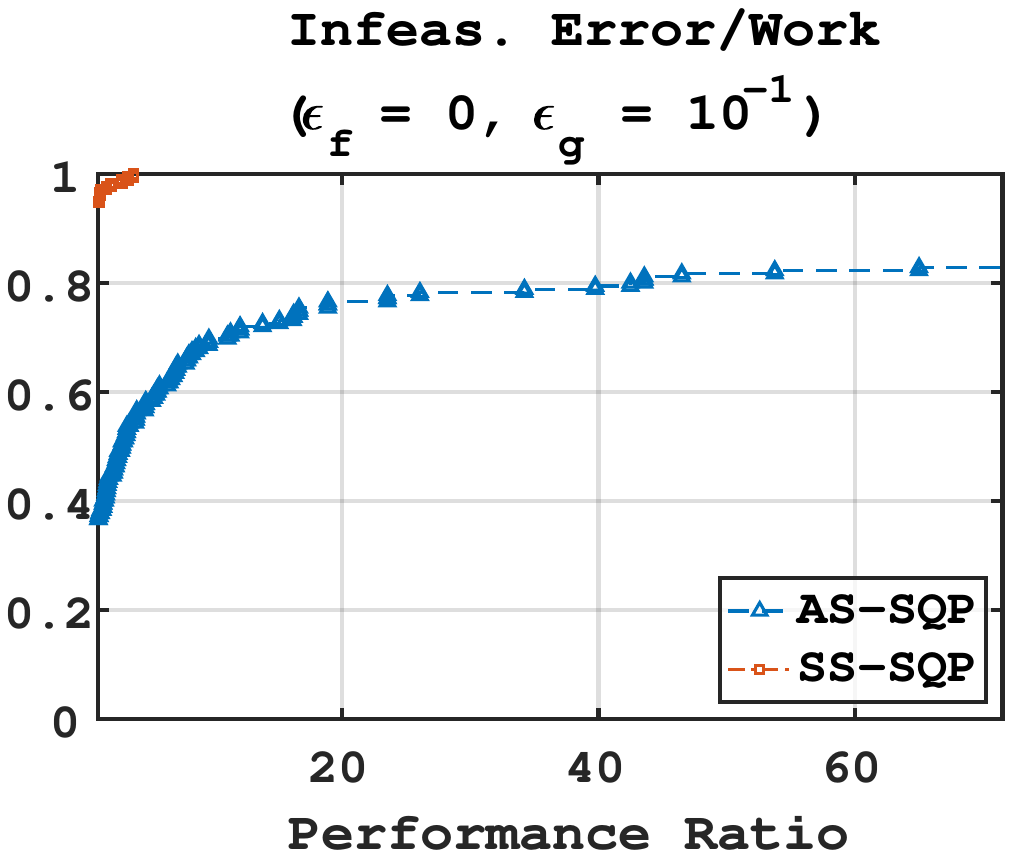}
  \includegraphics[width=0.24\textwidth,clip=true,trim=30 180 50 180]{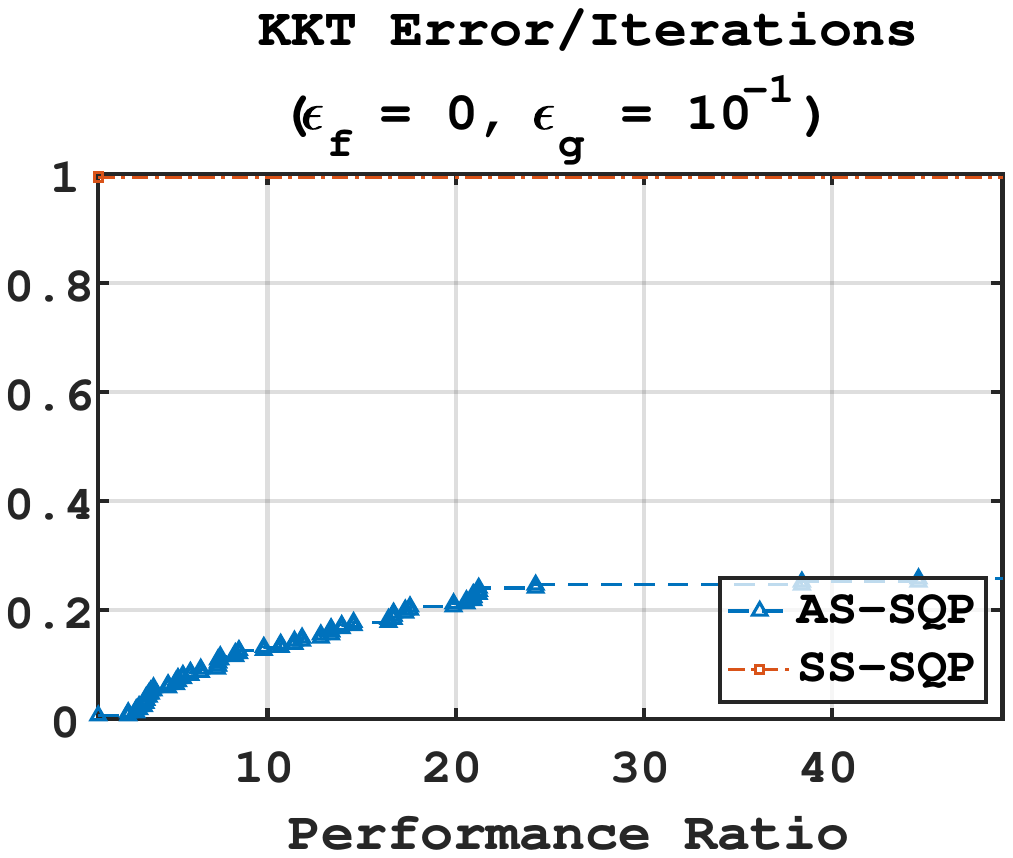}
  \includegraphics[width=0.24\textwidth,clip=true,trim=30 180 50 180]{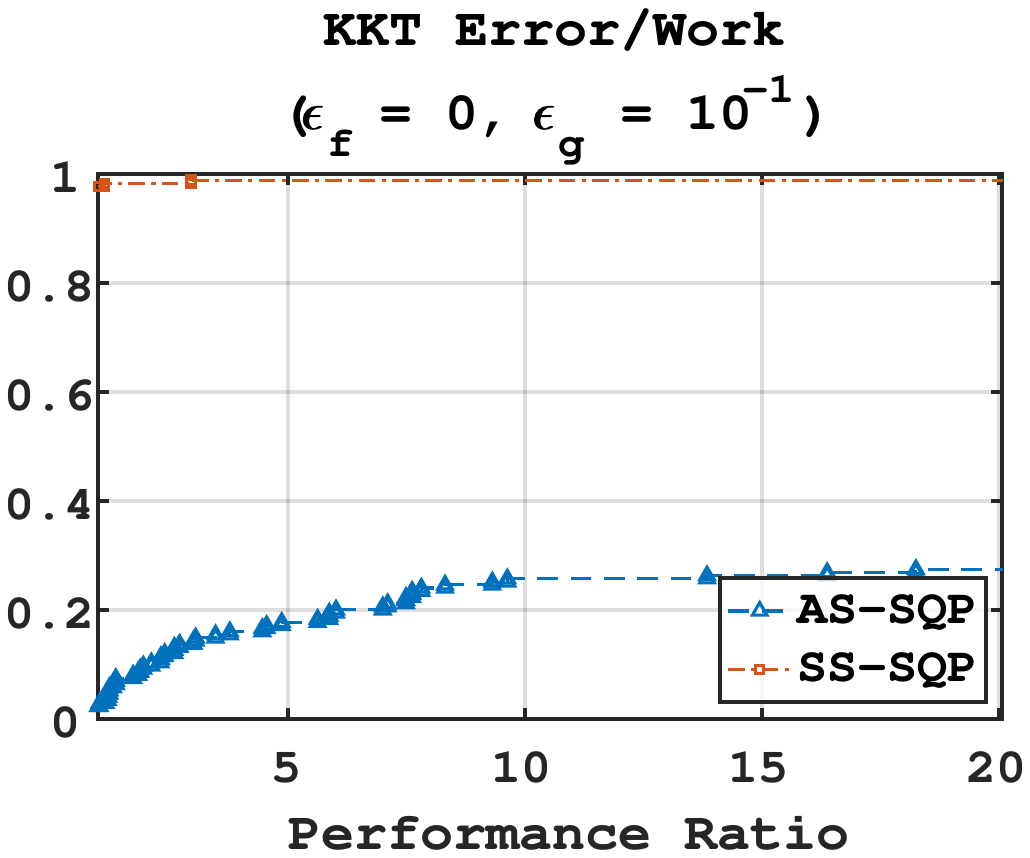}
  \caption{Performance profiles for \ASSPQ{} and \SSSPQ{} on CUTEst collection~\cite{gould2015cutest} with deterministic objective function evaluations ($\epsilon_f = 0$) and noisy objective gradient evaluations. Each column corresponds to a different evaluation metric (infeasibility and KKT errors vs. iterations and work). The noise in the objective gradient evaluations $\epsilon_g$ increases from top to bottom (First row: $\epsilon_g = 0$; Second row: $\epsilon_g = 10^{-4}$; Third row: $\epsilon_g = 10^{-2}$; Fourth row: $\epsilon_g = 10^{-1}$). \label{fig.1}}
\end{figure}

\subsection{Noisy Functions and Gradients}\label{sec.noisy_noisy}
Here we present results with noise in both the objective function and gradient evaluations. As in Figure~\ref{fig.1}, in Figure~\ref{fig.2} different rows show results for different noise levels in the gradient (the bottom row has the highest noise) and different columns show results for different evaluation metrics. Each performance profile has 4 lines: the \ASSPQ{} (that is objective-function-free and is not affected by the noise in the function evaluations) and three variants of the \SSSPQ{} method with different levels of noise in the objective function evaluations. One can make the following observations. First, not surprisingly, the performance of the \SSSPQ{} method degrades as the noise in the objective function evaluations increases. Second, \ASSPQ{} and \SSSPQ{} are competitive and achieve similar robustness levels with respect to infeasibility errors.  
Third, and most interestingly, the performance of the methods depends on the  relative errors of the function and gradient evaluations. In particular, when the objective function noise level is sufficiently small compared to the objective gradient bias, \SSSPQ{} performs better. On the other hand, when the function estimations are too noisy compared to the noise level in the gradient evaluations, \ASSPQ{} performs slightly better.  
These results highlight the power of objective-function-free optimization methods in the presence of noise (especially high noise in the objective function evaluations) and the value of quality (or at least relative quality) function evaluations in methods that require zeroth-order information.

We conclude this section by making a few remarks about the behavior of the merit parameter. In the noise-free setting ($\epsilon_f = \epsilon_g = 0$), over all problems, the minimum merit parameter value was (of the order of) $10^{-4}$. In the noisy setting, over all 2100 problem instances (35 problems, 12 noise levels, 5 replications), the minimum value for the merit parameter was (of the order of) $10^{-6}$ (this value was attained for a problem with the highest noise level), and less than 5\% of the final merit parameter values were less that (of the order of) $10^{-4}$. Moreover, across different realizations of the same problem with the same noise levels, the variance in the final merit parameter value was small (and dependent on the noise levels).

\begin{figure}[ht]
   \centering
  \includegraphics[width=0.24\textwidth,clip=true,trim=30 180 50 180]{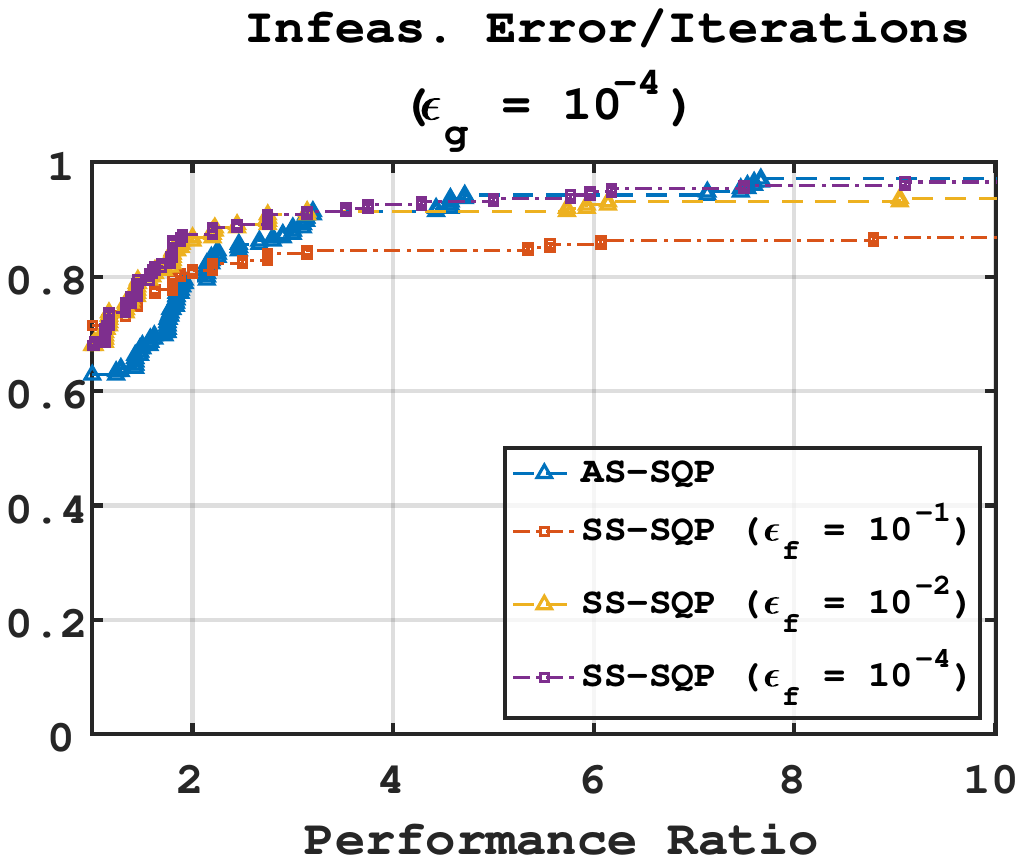}
  \includegraphics[width=0.24\textwidth,clip=true,trim=30 180 50 180]{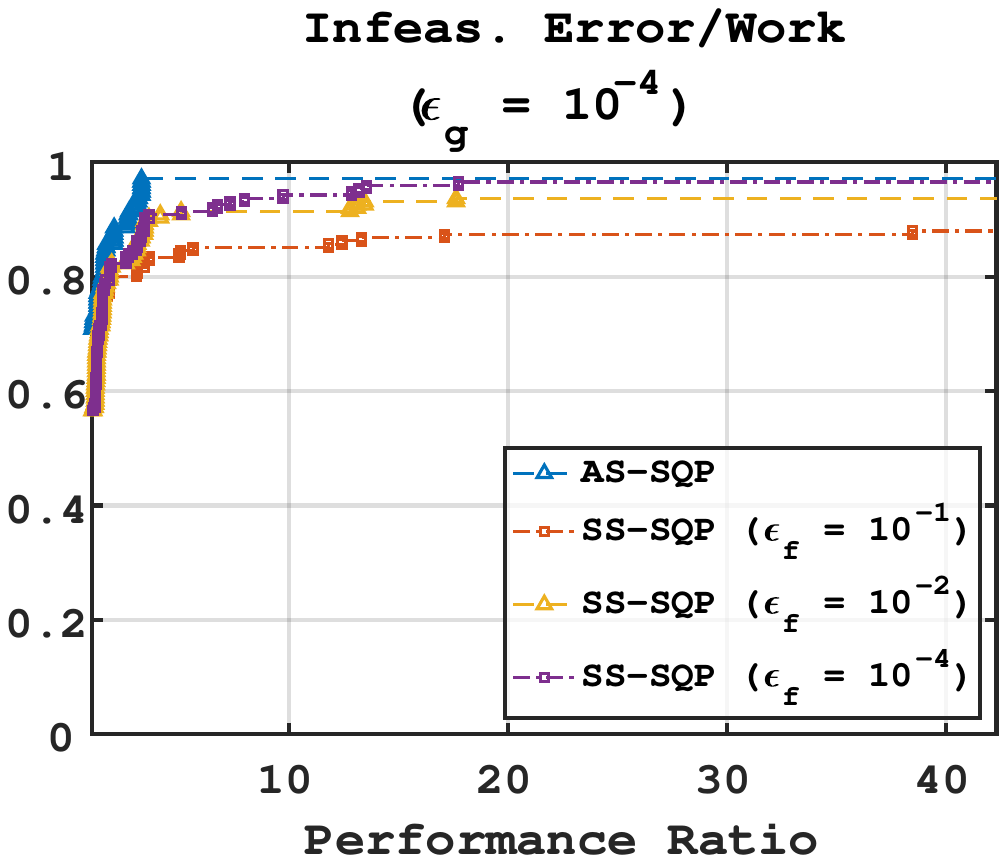}
  \includegraphics[width=0.24\textwidth,clip=true,trim=30 180 50 180]{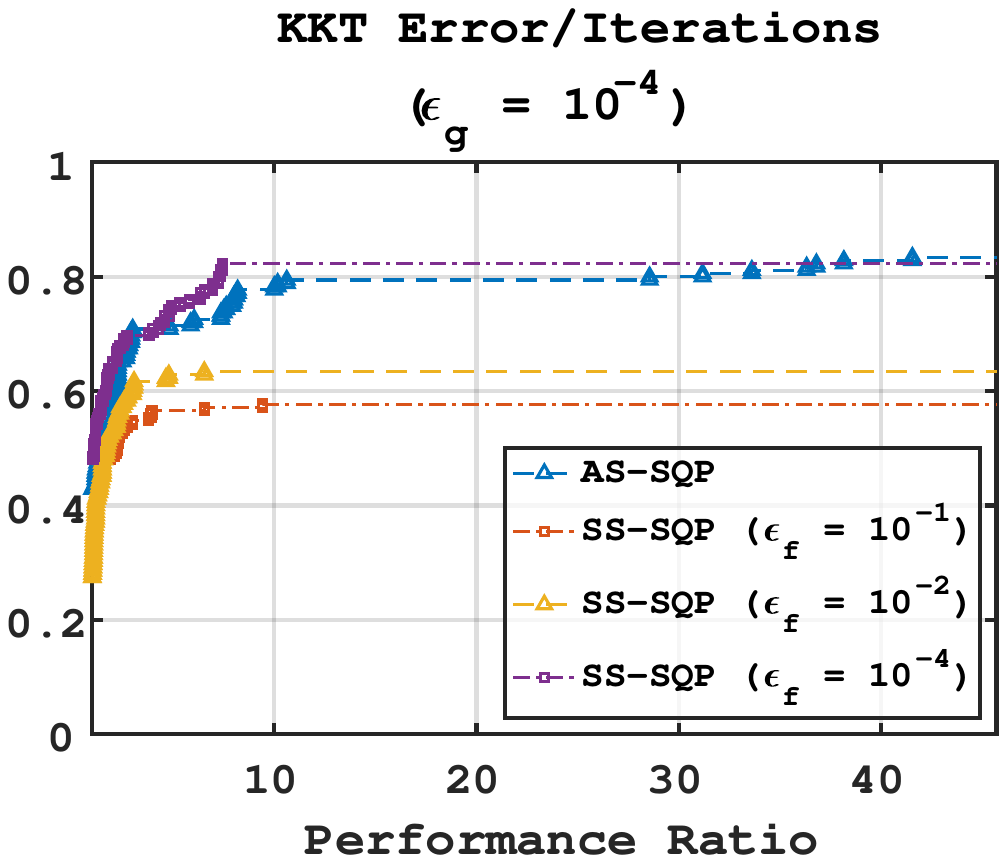}
  \includegraphics[width=0.24\textwidth,clip=true,trim=30 180 50 180]{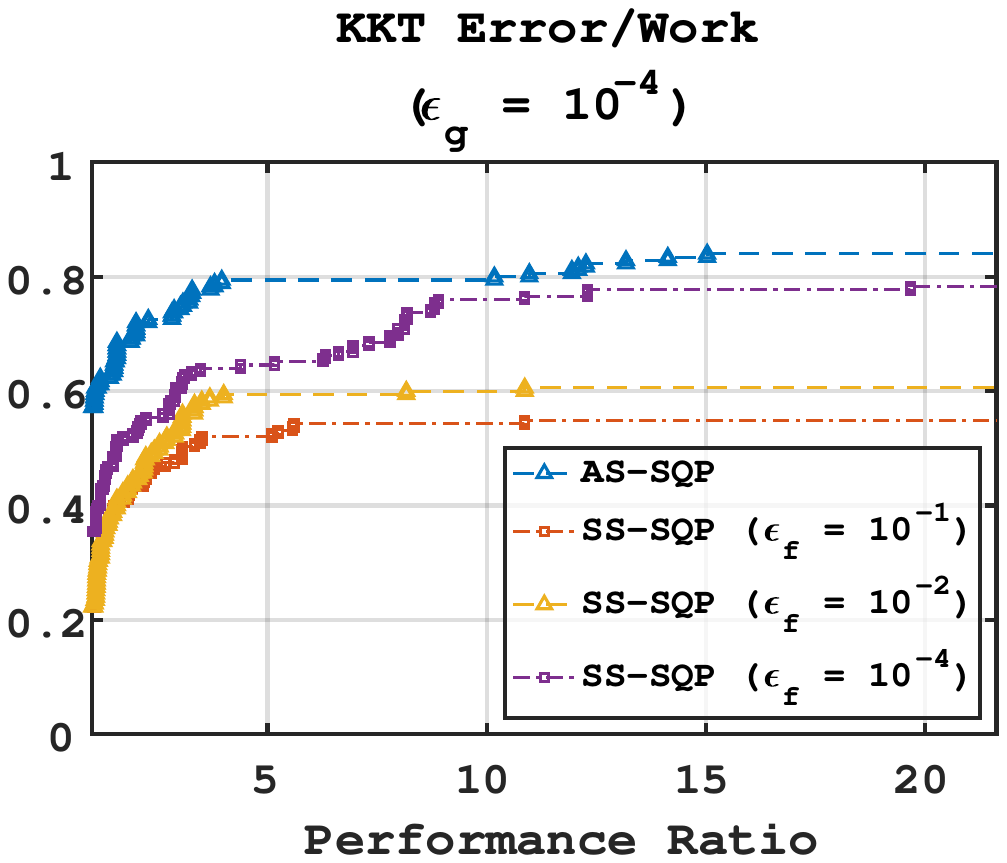}
 
 \includegraphics[width=0.24\textwidth,clip=true,trim=30 180 50 180]{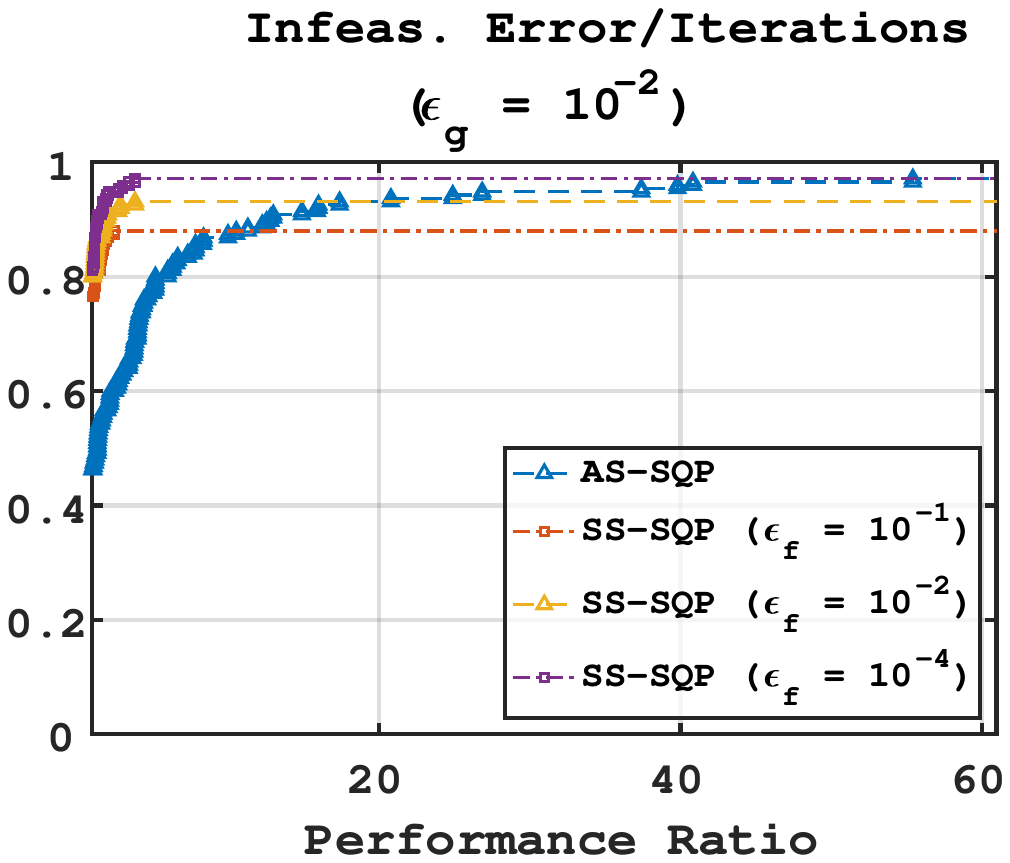}
  \includegraphics[width=0.24\textwidth,clip=true,trim=30 180 50 180]{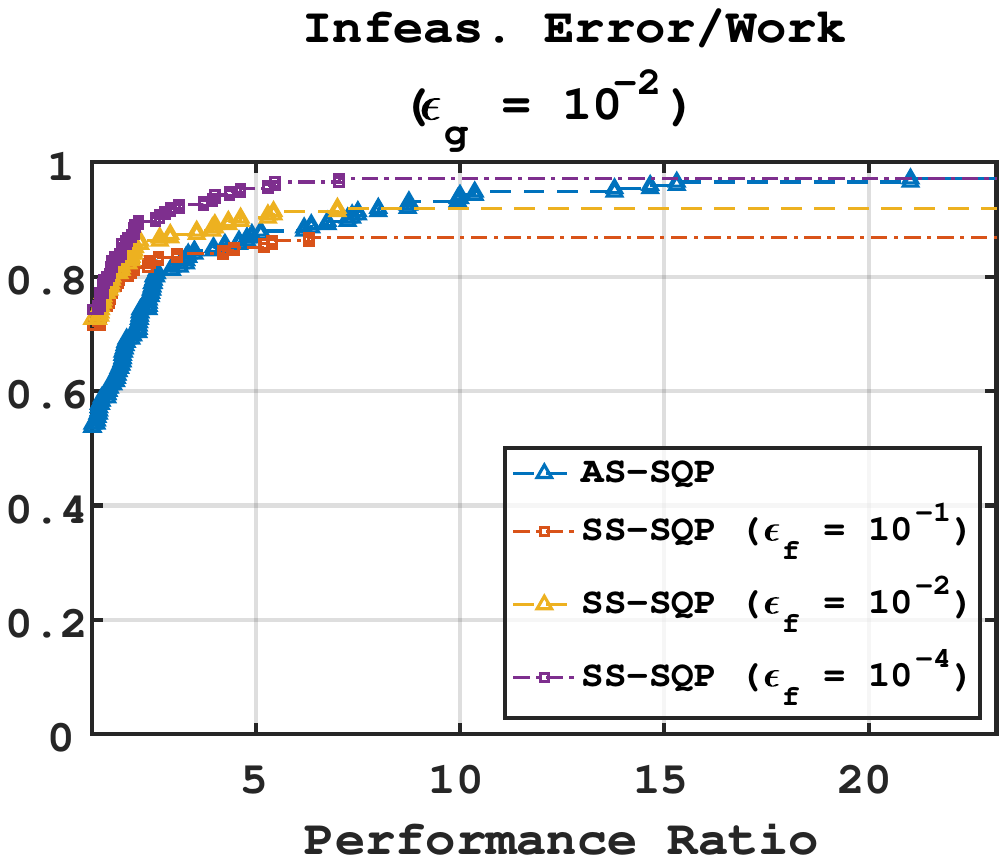}
  \includegraphics[width=0.24\textwidth,clip=true,trim=30 180 50 180]{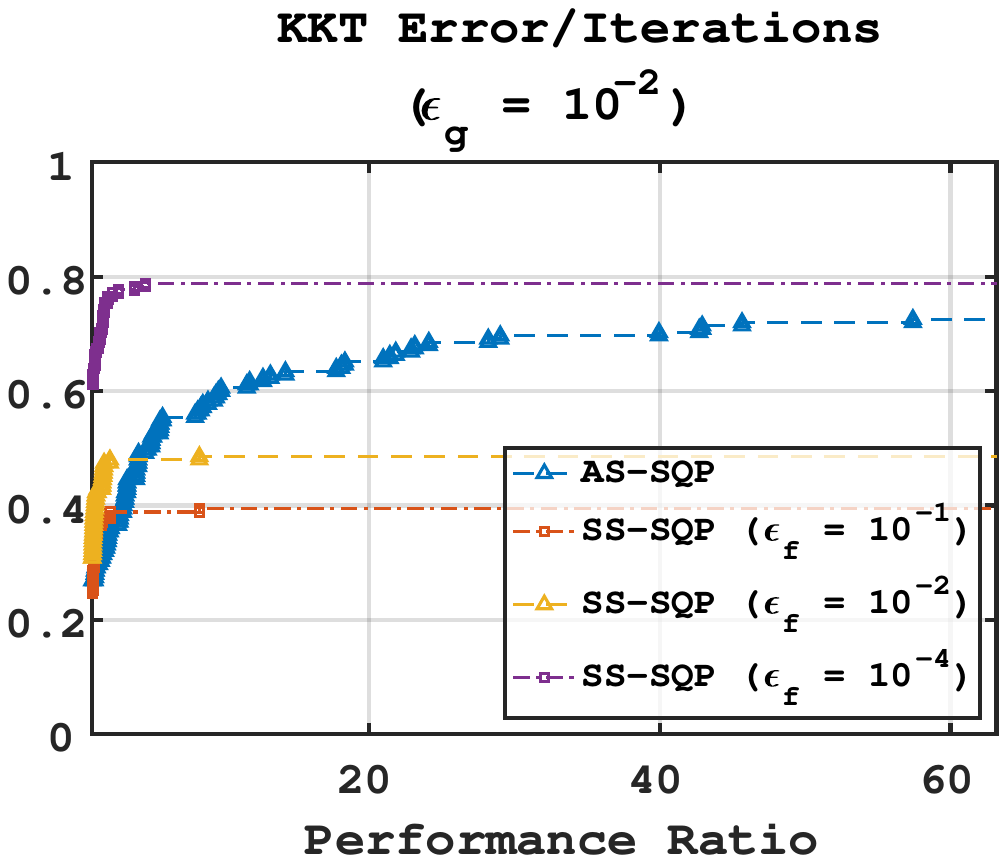}
  \includegraphics[width=0.24\textwidth,clip=true,trim=30 180 50 180]{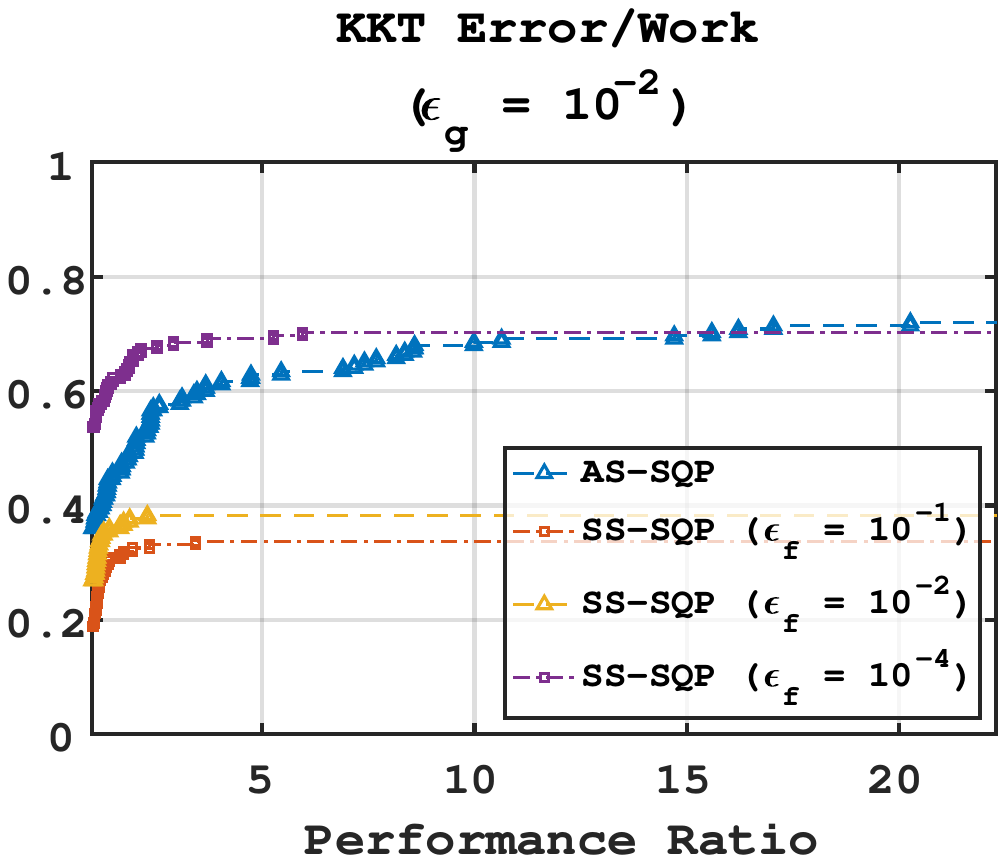}
  
  \includegraphics[width=0.24\textwidth,clip=true,trim=30 180 50 180]{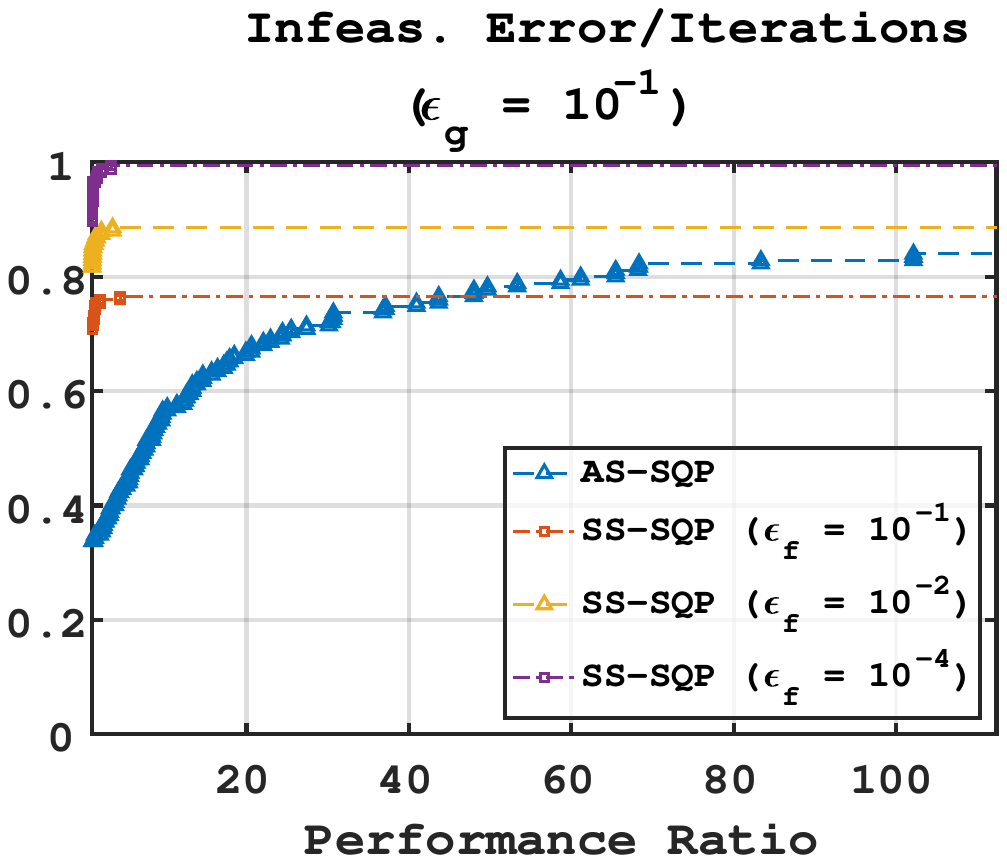}
  \includegraphics[width=0.24\textwidth,clip=true,trim=30 180 50 180]{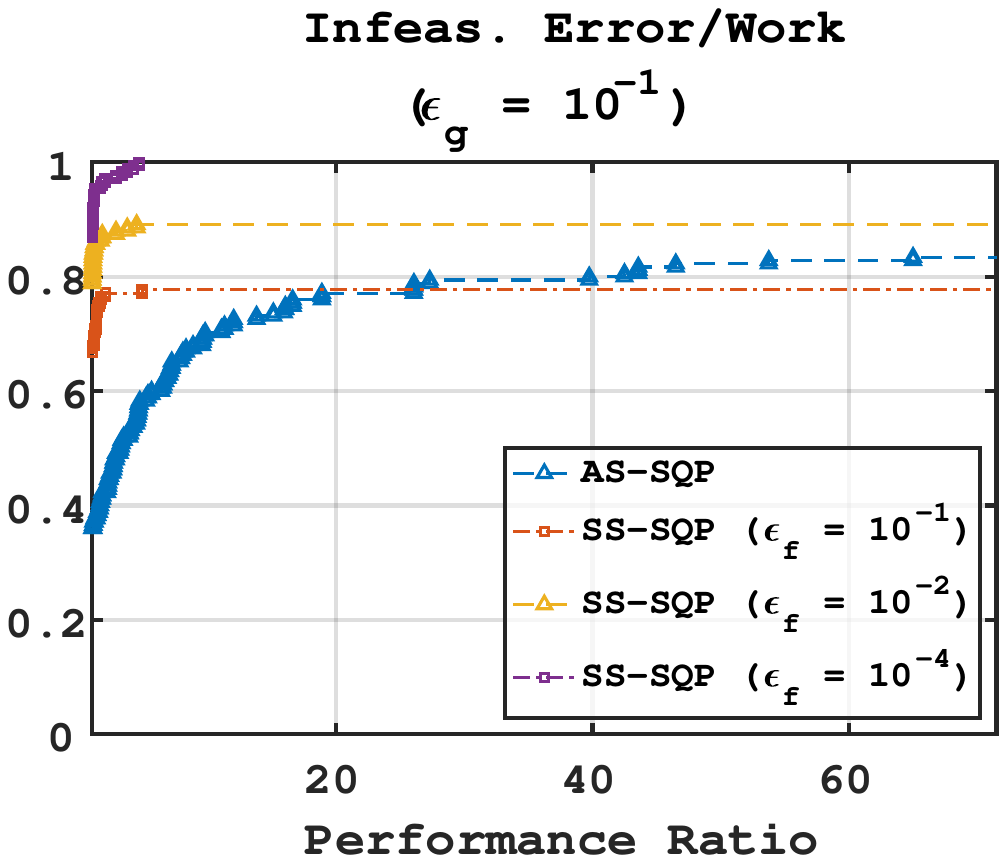}
  \includegraphics[width=0.24\textwidth,clip=true,trim=30 180 50 180]{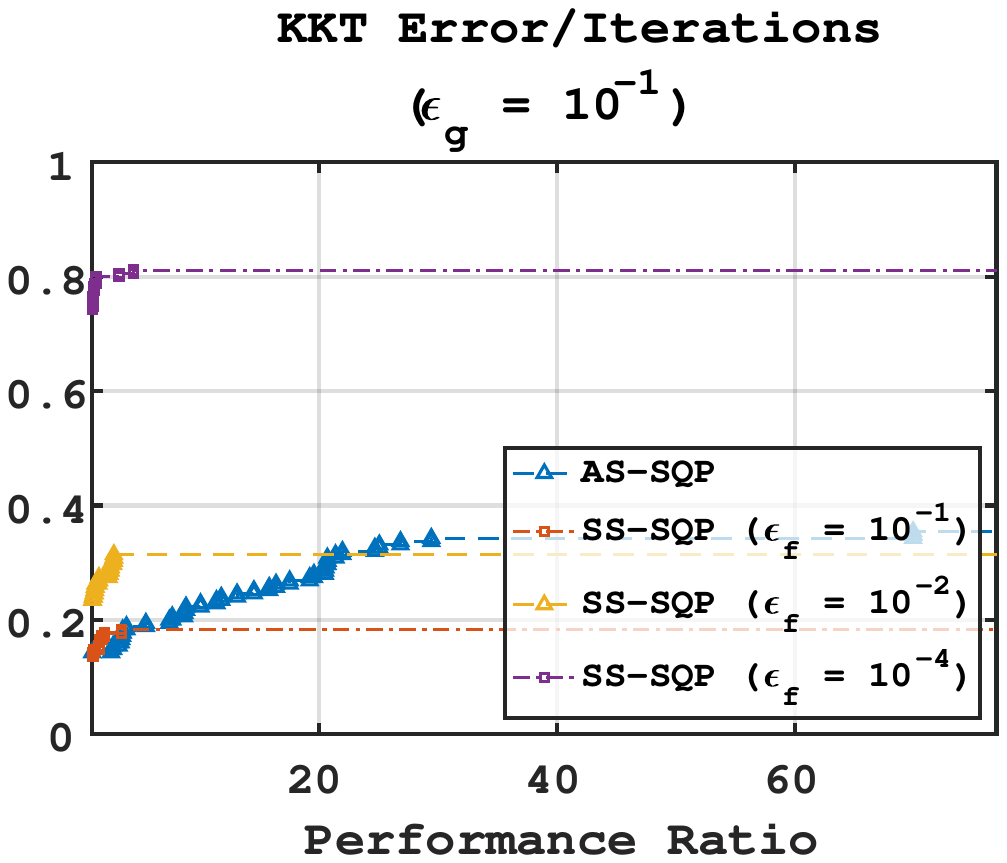}
  \includegraphics[width=0.24\textwidth,clip=true,trim=30 180 50 180]{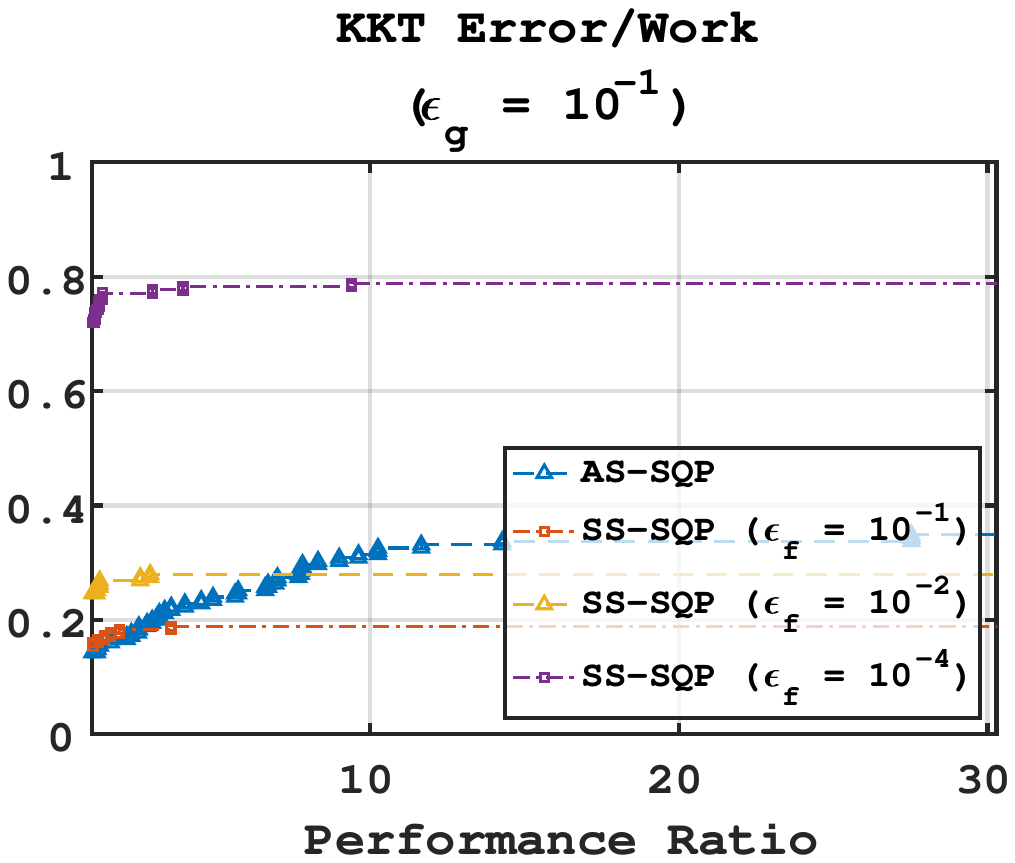}
  \caption{Performance profiles for \ASSPQ{} and \SSSPQ{} on CUTEst collection~\cite{gould2015cutest} with noise in both the objective function and gradient evaluations. Each column corresponds to a different evaluation metric (infeasibility and KKT vs. iterations and work). The noise in the objective gradient evaluations $\epsilon_g$ increases from top to bottom (First row: $\epsilon_g = 10^{-4}$; Second row: $\epsilon_g = 10^{-2}$; Third row: $\epsilon_g = 10^{-1}$). The different variants of \SSSPQ{} correspond to different levels of noise in the objective function evaluations. \label{fig.2}}
\end{figure}

\section{Conclusion}\label{sec.conclusion}
We have proposed a step-search SQP algorithm (\SSSPQ{}) for solving stochastic optimization problems with deterministic equality constraints.
We showed that under reasonable assumptions on the inexact probabilistic zeroth- and first-order oracles, for any $\varepsilon$ greater than a positive quantity dictated by the noise and bias in the oracles, with overwhelmingly high probability, in $O(\varepsilon^{-2})$ iterations our algorithm can produce an iterate that satisfies the first-order $\varepsilon$-stationarity, which matches the iteration complexity of the deterministic counterparts of the SQP algorithm \cite{CurtONeiRobi21}. Numerical results provide strong evidence for the efficiency and efficacy of the proposed method. Some future directions include but are not limited to, $(1)$ incorporating stochastic constraint evaluations into the algorithm design and analysis, and $(2)$ extending the framework to the setting with inequality constraints. Both avenues above are subjects of future work as they require significant adaptations in the design, analysis, and implementation of the algorithm.

%% file: acknowledgements.tex
\section*{Acknowledgments}
This material is based upon work supported by the Office of Naval Research under award number N00014-21-1-2532. We would like to thank Professors Frank E. Curtis and Katya Scheinberg for their invaluable support and feedback.

%% file: appendix.tex
\section{Constants in Assumption~\ref{ass.main_accuracy} and Lemma~\ref{algo_behave}}\label{app.constants}

In this appendix, we provide the definitions for all constants that appear in Assumption~\ref{ass.main_accuracy} and Lemma~\ref{algo_behave}.

\begin{table}[h!]
\caption{Table of Constants from Assumption~\ref{ass.main_accuracy} and Lemma~\ref{algo_behave}.}
\label{tbl:constants}
\centering 
\resizebox{\columnwidth}{!}{\begin{tabular}{lc}
    \toprule
    \textbf{Constant} & \textbf{Definition} \\  \midrule
$\eta_0$ &  $1 + \tfrac{\kappa_H}{\zeta}$\\ 
\hdashline
$\eta_1$&  $\tfrac{(1-\theta)(1+\epsilon_{\tau})\bar\tau_{-1}\eta_0}{\sqrt{\kappa_l\tau_{\min}}}$   \\ 
\hdashline
$\eta_2$&  $\sqrt{(1-\theta)^2\bar\tau_{-1}\eta_0^2\left(\tfrac{(1+\epsilon_{\tau})^2\bar\tau_{-1}}{\kappa_l\tau_{\min}} + \epsilon_{\tau}\right) + 4\bar\tau_{-1}\left(\tfrac{1+\epsilon_{\tau}\omega_1}{\kappa_l} + \tfrac{(1-\theta)^2(1+\epsilon_{\tau})}{\zeta}\right)}$   \\ 
\hdashline
$\eta_3$ & $\tfrac{(1-\theta)\bar\tau_{-1}\left(\bar\tau_{-1}\left(3\eta_0 - 2\right) + (1-\sigma)\tau_{\min}\eta_0\right)}{(1-\sigma)\tau_{\min}\sqrt{\kappa_l\tau_{\min}}}$   \\
\hdashline
$\eta_4$ & $\sqrt{\tfrac{(1-\theta)^2\bar\tau_{-1}^2\left(\bar\tau_{-1}\left(3\eta_0 - 2\right) + (1-\sigma)\tau_{\min}\eta_0\right)^2}{(1-\sigma)^2\tau_{\min}^3\kappa_l} + \tfrac{4\bar\tau_{-1}}{\kappa_l} +  \tfrac{4(1-\theta)^2\bar\tau_{-1}}{\zeta}\left(\tfrac{\bar\tau_{-1}\eta_0}{(1-\sigma)\tau_{\min}} + 1 \right)}$ \\
\hdashline
$\omega_1$ & $\tfrac{\max\{\kappa_H,\kappa_y\} + 1}{\kappa_l}$   \\
\hdashline
$\omega_2$ & $\tfrac{\eta_0\kappa_{\mathrm{FO}}\sqrt{\bar\tau_{-1}}\alpha_{\max}}{\sqrt{\kappa_l}} + \tfrac{\bar\tau_{-1}\kappa_{\mathrm{FO}}^2\alpha_{\max}^2}{\zeta}$   \\
\hdashline
$\omega_3$ & \begin{tabular}[c]{@{}c@{}}$\max\left\{\epsilon_{\tau}\left(\tfrac{\max\{\kappa_H,\kappa_y\}}{\kappa_l} + \tfrac{\sqrt{\bar\tau_{-1}}\eta_0\kappa_{\mathrm{FO}}\alpha_{\max}}{\sqrt{\kappa_l}} + \omega_2\right)\right.$,\\ $\qquad\qquad \qquad\quad  \left.\tfrac{\bar\tau_{-1}}{(1-\sigma)\tauTrue_{\min}}\left(\tfrac{(3\eta_0 - 2)\kappa_{\mathrm{FO}}\sqrt{\bar\tau_{-1}}\alpha_{\max}}{\sqrt{\kappa_l}} + \tfrac{\eta_0\bar\tau_{-1}\kappa_{\mathrm{FO}}^2\alpha_{\max}^2}{\zeta}\right) \right\}$\end{tabular}  
\\
\hdashline
$\omega_4$ & $(1 + \epsilon_{\tau})\bar\tau_{-1}\left(\tfrac{\eta}{\zeta} + \tfrac{\eta_0}{\sqrt{\kappa_l\tauTrue_{\min}}}\right) + \tfrac{\epsilon_{\tau}\bar\tau_{-1}\eta_0^2\eta}{4}$   \\
\hdashline
$\omega_5$ & $\tfrac{\bar\tau_{-1}^2\left(\tfrac{\eta_0\eta}{\zeta} + \tfrac{3\eta_0 - 2}{\sqrt{\kappa_l\tauTrue_{\min}}} \right)}{(1-\sigma)\tauTrue_{\min}} + \bar\tau_{-1}\left(\tfrac{\eta}{\zeta} + \tfrac{\eta_0}{\sqrt{\kappa_l\tauTrue_{\min}}}\right)$   \\
\hdashline
$\omega_6$ & $\sqrt{\tfrac{4\bar\tau_{-1}}{(p-\tfrac 12)\theta}\max\left\{\tfrac{1+\epsilon_{\tau}\omega_1}{1-\eta\omega_4},\tfrac{1}{1-\eta\omega_5},1+\omega_2 + \omega_3\right\}}$   \\
\hdashline
$\omega_7$ & $\sqrt{\tfrac{\bar\tau_{-1}}{\kappa_l}}\kappa_{\mathrm{FO}} + \tfrac{\bar\tau_{\min}L+ \Gamma}{2\bar\tau_{\min}\kappa_l}$\\
\hdashline
$\omega_8$ & $2\bar\tau_{\min}\kappa_l\left(1-\theta-\eta\sqrt{\tfrac{\bar\tau_{-1}}{\kappa_l}}\max\left\{\sqrt{\tfrac{1+\epsilon_{\tau}\omega_1}{1-\eta\omega_4}},\tfrac{1}{\sqrt{1-\eta\omega_5}}\right\}\right)$\\
\hdashline
$\omega_9$ & $\sqrt{\max\left\{\tfrac{\omega_7 }{1-\theta}, \tfrac{\bar\tau_{\min}L+\Gamma}{\omega_8}\right\}}$   \\
\hdashline
$\omega_{10}$ & $\omega_6\omega_9$\\
    \bottomrule
\end{tabular}}
\end{table}

\section{Proof of Lemma~\ref{algo_behave}}\label{app.proof_lemma}

In this appendix, we provide the proof of Lemma~\ref{algo_behave}. From the lemma statement, 
we note that: $(1)$ due to the constants and the form of $p$, $p$ is a valid probability, i.e., $p \in (\tfrac{1}{2},1]$, $(2)$ $\tilde\alpha > 0$ is guaranteed by the restriction on $\eta$ in Assumption~\ref{ass.main_accuracy}, and $(3)$ $h:\mathbb{R}_{>0}\to\mathbb{R}_{>0}$ is a positive function that measures the potential progress made if iterations are \emph{\textbf{true}} and \emph{\textbf{successful}}. Next, we separate the proofs of statements $(i)$--$(v)$ as follows.
\begin{proof}[Proof of (i)]
This result follows directly from the definition of $h(\tilde\alpha)$ and the lower bound on $\varepsilon_{\Delta l}$; see Assumption~\ref{ass.main_accuracy}.
\end{proof}
\begin{proof}[Proof of (ii)]
This proof is essentially the same as that from \cite[Proposition 3(ii)]{jin2022high}. Let
    \begin{equation*}
        J_k :=\mathds{1}\left\{\|{G}_k - \nabla f(X_k)\| \leq \max\left\{ \epsilon_g, \kappa_{\mathrm{FO}}A_k\sqrt{\Delta l(X_k,\mathcal{T}_k, G_k,D_k)} \right\}\right\}. 
    \end{equation*}
    Clearly, by Definition~\ref{def.true}, 
    \begin{equation*}
    \begin{aligned}
	    \P\left[I_k = 0 \mid \change{\sF_{k-1} \cap \mathcal{E}}\right]&= \P\left[J_k = 0 \;\;\text{or}\;\; E_{k}+E^+_{k} > 2\epsilon_f\mid \change{\sF_{k-1} \cap \mathcal{E}}\right] \\
	    &\leq \P\left[J_k = 0 \mid \change{\sF_{k-1} \cap \mathcal{E}}\right] + \P\left[ E_{k}+E^+_{k}> 2\epsilon_f\mid \change{\sF_{k-1} \cap \mathcal{E}}\right].
    \end{aligned}
    \end{equation*}
    The first term on the right-hand-side of the inequality is bounded above by $\delta$, by \change{Assumption~\ref{ass.good_merit_paramter} and} the first-order probabilistic oracle (Oracle~\ref{oracle.first}). The second term is zero in the case where $\epsilon_f$ is a deterministic bound on the noise. Otherwise, since $E_k$ and $E_k^+$ individually satisfy the one-sided sub-exponential bound in \eqref{eq:zero_order} with parameters $\epsilon_f$ and $(\nu, b)$ \change{conditioned on 
    event $\mathcal{E}$ (see Assumption~\ref{ass.good_merit_paramter})}, one can show that \change{conditioned on 
    event $\mathcal{E}$,} $E_k + E_k^+$ satisfies \eqref{eq:zero_order} with parameters $2\epsilon_f$ and $(2\nu, 2b)$. Hence by the one-sided Bernstein inequality, the second term is bounded above by $e^{-\min\left\{u^2/2\nu^2,u/2b\right\}}$, with $u = \inf_{x\in\mathcal{X}} \{\epsilon_f - \E[E(x)]\}$. As a result, $\P\left[I_k = 1 \mid \change{\sF_{k-1} \cap \mathcal{E}}\right] \geq p$ 
    for all $k$, for $p$ as defined in the statement. The range of $p \in \left(\tfrac12 + \tfrac{4\bar\tau_{-1}\epsilon_f}{h(\tilde\alpha)}, 1\right]$ follows from the definitions of $h(\cdot)$ and $\tilde\alpha$ in the statement, together with the inequality on $\varepsilon_{\Delta l}$ in Assumption~\ref{ass.main_accuracy}.
\end{proof}
\begin{proof}[Proof of (iii)]
Suppose iteration $k$ is \emph{\textbf{true}} and \emph{\textbf{successful}}. By Definition~\ref{def.true}, there are two cases, $\|\bar{g}_k - \nabla f_k\| \leq \kappa_{\mathrm{FO}}\alpha_k\sqrt{\Delta l(x_k,\bar\tau_{k},\bar{g}_k,\bar{d}_k)}$ and $\|\bar{g}_k - \nabla f_k\| \leq \epsilon_g$, that we consider separately.
    We further subdivide the analysis into the case where $\nabla f_k^T\dTrue_k \leq 0$ and $\nabla f_k^T\dTrue_k > 0$. \\    
    \textbf{Case A} \  When $\|\bar{g}_k - \nabla f(x_k)\| \leq \kappa_{\mathrm{FO}}\alpha_k\sqrt{\Delta l(x_k,\bar\tau_{k},\bar{g}_k,\bar{d}_k)}$, by Lemma~\ref{lem.step_diff_bound},
        \begin{equation*}
            \|\bar{d}_k - \dTrue_k\| \leq \zeta^{-1}\|\bar{g}_k - \nabla f(x_k)\| \leq \zeta^{-1}\kappa_{\mathrm{FO}}\alpha_k\sqrt{\Delta l(x_k,\bar\tau_{k},\bar{g}_k,\bar{d}_k)}.
        \end{equation*}
    \textbf{Case A.1} \ If $\nabla f_k^T\dTrue_k \leq 0$, by the fact that $\bar\tau_k \geq \tauTrue_k$, the triangle inequality, \eqref{eq.model_reduction} and Lemma~\ref{lem.gd_dHd_diff_case_a_2}, it follows that
        \begin{equation}\label{eq.case_a1}
        \begin{aligned}
            &\Delta l(x_k,\tauTrue_k,\nabla f_k,\dTrue_k) - \Delta l(x_k,\bar\tau_k,\bar{g}_k,\bar{d}_k) \\
            =\ & \bar\tau_k\bar{g}_k^T\bar{d}_k - \tauTrue_k \nabla f_k^T\dTrue_k \leq \bar\tau_k(\bar{g}_k^T\bar{d}_k - \nabla f_k^T\dTrue_k) \leq \bar\tau_k|\bar{g}_k^T\bar{d}_k - \nabla f_k^T\dTrue_k| \\
            \leq \ &\bar\tau_k\left(\tfrac{(1+\kappa_H\zeta^{-1})\kappa_{\mathrm{FO}}\alpha_k}{\sqrt{\kappa_l\bar\tau_k}} + \tfrac{\kappa_{\mathrm{FO}}^2\alpha_k^2}{\zeta}\right)\Delta l(x_k,\bar\tau_k,\bar{g}_k,\bar{d}_k).
        \end{aligned}
        \end{equation}
    \textbf{Case A.2} \ If $\nabla f_k^T\dTrue_k > 0$, by the triangle inequality, \eqref{eq.model_reduction} and Lemma~\ref{lem.gd_dHd_diff_case_a_2}, 
        \begin{equation}\label{eq.case_a2}
        \begin{aligned}
            &\Delta l(x_k,\tauTrue_k,\nabla f_k,\dTrue_k) - \Delta l(x_k,\bar\tau_k,\bar{g}_k,\bar{d}_k) \\
            = \ & \bar\tau_k\bar{g}_k^T\bar{d}_k - \tauTrue_k \nabla f_k^T\dTrue_k \leq |\bar\tau_k\bar{g}_k^T\bar{d}_k - \tauTrue_k \nabla f_k^T\dTrue_k| \leq |(\bar\tau_k - \tauTrue_k)\nabla f_k^T\dTrue_k| + \bar\tau_k|\bar{g}_k^T\bar{d}_k - \nabla f_k^T\dTrue_k| \\
            \leq \ &|(\bar\tau_k - \tauTrue_k)\nabla f_k^T\dTrue_k| + \bar\tau_k\left(\tfrac{(1+\kappa_H\zeta^{-1})\kappa_{\mathrm{FO}}\alpha_k}{\sqrt{\kappa_l\bar\tau_k}} + \tfrac{\kappa_{\mathrm{FO}}^2\alpha_k^2}{\zeta}\right)\Delta l(x_k,\bar\tau_k,\bar{g}_k,\bar{d}_k).
        \end{aligned}
        \end{equation}
        We now bound the term $|(\bar\tau_k - \tauTrue_k)\nabla f_k^T\dTrue_k|$; we consider three cases due to the merit parameter updating formulae  (\eqref{eq.tau_update}--\eqref{eq.tau_trial} and \eqref{eq.tautrue_update}--\eqref{eq.tautrue_trial}). \\
        \textbf{Case A.2.1} \ If $\tauTrue_k = \bar\tau_k$, then $|(\bar\tau_k - \tauTrue_k)\nabla f_k^T\dTrue_k| = 0$. \\
        \textbf{Case A.2.2} \ If $\tauTrue_k = (1-\epsilon_{\tau})\bar\tau_{k}$, by the triangle inequality and Lemmas~\ref{lem.|gd|_2} and~\ref{lem.gd_dHd_diff_case_a_2},
        \begin{align*}
         |(\bar\tau_k - \tauTrue_k)\nabla f_k^T\dTrue_k| = \ &  \epsilon_{\tau}\bar\tau_k|\nabla f_k^T\dTrue_k| \leq \epsilon_{\tau}\bar\tau_k(|\bar{g}_k^T\bar{d}_k| + |\nabla f_k^T\dTrue_k - \bar{g}_k^T\bar{d}_k|) \\
            \leq \ &\epsilon_{\tau}\left(\tfrac{\max\{\kappa_H,\kappa_y\}}{\kappa_l} + \tfrac{\sqrt{\bar\tau_k}\left(1 + \frac{\kappa_H}{\zeta}\right)\kappa_{\mathrm{FO}}\alpha_k}{\sqrt{\kappa_l}}\right)\Delta l(x_k,\bar\tau_k,\bar{g}_k,\bar{d}_k) \\
            & +  \epsilon_{\tau}\bar\tau_k\left(\tfrac{(1+\frac{\kappa_H}{\zeta})\kappa_{\mathrm{FO}}\alpha_k}{\sqrt{\kappa_l\bar\tau_k}} + \tfrac{\kappa_{\mathrm{FO}}^2\alpha_k^2}{\zeta}\right)\Delta l(x_k,\bar\tau_k,\bar{g}_k,\bar{d}_k).
        \end{align*}
        \textbf{Case A.2.3} \ If $\bar\tau_{k} > \tauTrue_k = \tfrac{(1-\sigma)\|c_k\|_1}{\nabla f_k^T\dTrue_k + \max\left\{\dTrue_k^TH_k\dTrue_k,0\right\}}$, by  \eqref{eq.tau_update}--\eqref{eq.tau_trial},
        \begin{equation}\label{eq.tau_diff_bound}
        \begin{aligned}
            \nabla f_k^T\dTrue_k + \max\left\{\dTrue_k^TH_k\dTrue_k,0\right\} > \tfrac{(1-\sigma)\|c_k\|_1}{\bar\tau_{k}} \geq \bar{g}_k^T\bar{d}_k + \max\left\{\bar{d}_k^TH_k\bar{d}_k,0\right\}.
        \end{aligned}
        \end{equation}
        \change{Conditioned on 
        event $\mathcal{E}$ (see Assumption~\ref{ass.good_merit_paramter}), by} Lemma~\ref{lem.deter_tau_lb}, 
        we have $\tauTrue_k \geq \tauTrue_{\min}$ for all $k\in\N{}$. Moreover, it follows from \eqref{eq.model_reduction} and Lemma~\ref{lem.Deltal_lb_det} that $0 \leq \Delta l(x_k,\tauTrue_k,\nabla f_k,\dTrue_k)$, which implies $\tauTrue_k\nabla f_k^T\dTrue_k \leq \|c_k\|_1$. Using the fact that $\tauTrue_k\in\R{}_{>0}$ and $\nabla f_k^T\dTrue_k > 0$, 
        \begin{equation}\label{eq.gd/c_bound}
            \tfrac{|\nabla f_k^T\dTrue_k|}{\|c_k\|_1} = \tfrac{\nabla f_k^T\dTrue_k}{\|c_k\|_1} \leq \tfrac{1}{\tauTrue_k}.
        \end{equation}
        By Lemma~\ref{lem.gd_dHd_diff_case_a_2}, \eqref{eq.tau_diff_bound} and \eqref{eq.gd/c_bound}, it follows that
        \begin{align*}
            &|(\bar\tau_k - \tauTrue_k)\nabla f_k^T\dTrue_k| \\
            = \ & \left(\bar\tau_k - \tfrac{(1-\sigma)\|c_k\|_1}{\nabla f_k^T\dTrue_k + \max\left\{\dTrue_k^TH_k\dTrue_k,0\right\}}\right)|\nabla f_k^T\dTrue_k| \\
            \leq \ &\tfrac{(\nabla f_k^T\dTrue_k + \max\left\{\dTrue_k^TH_k\dTrue_k,0\right\}) - (\bar{g}_k^T\bar{d}_k + \max\{\bar{d}_k^TH_k\bar{d}_k,0\})}{\nabla f_k^T\dTrue_k + \max\left\{\dTrue_k^TH_k\dTrue_k,0\right\}}\bar\tau_k|\nabla f_k^T\dTrue_k| \\
            \leq \ &\tfrac{|\nabla f_k^T\dTrue_k - \bar{g}_k^T\bar{d}_k| + |\max\left\{\dTrue_k^TH_k\dTrue_k,0\right\} - \max\{\bar{d}_k^TH_k\bar{d}_k,0\}|}{(1-\sigma)\|c_k\|_1} \bar\tau_k^2|\nabla f_k^T\dTrue_k| \\
            \leq \ &\tfrac{\bar\tau_k^2}{(1-\sigma)\tauTrue_k} \left(|\nabla f_k^T\dTrue_k - \bar{g}_k^T\bar{d}_k| + |\max\left\{\dTrue_k^TH_k\dTrue_k,0\right\} - \max\{\bar{d}_k^TH_k\bar{d}_k,0\}|\right) \\
            \leq \ &\tfrac{\bar\tau_k^2}{(1-\sigma)\tauTrue_k} \left(|\nabla f_k^T\dTrue_k - \bar{g}_k^T\bar{d}_k| + |\dTrue_k^TH_k\dTrue_k - \bar{d}_k^TH_k\bar{d}_k|\right) \\
            \leq \ &\tfrac{\bar\tau_k^2}{(1-\sigma)\tauTrue_{\min}}\left(\tfrac{(1 + 3\kappa_H\zeta^{-1})\kappa_{\mathrm{FO}}\alpha_k}{\sqrt{\kappa_l\bar\tau_k}} + \frac{(1+\kappa_H\zeta^{-1})}{\zeta}\kappa_{\mathrm{FO}}^2\alpha_k^2 \right)\Delta l(x_k,\bar\tau_k,\bar{g}_k,\bar{d}_k).
        \end{align*}
        Combining \textbf{Cases A.2.1--A.2.3}, and by \eqref{eq.case_a1}, \eqref{eq.case_a2}, and the definitions of $\{\omega_2,\omega_3\}\subset\R{}_{>0}$ (Assumption~\ref{ass.main_accuracy} and Table~\ref{tbl:constants}), it follows that
        \begin{equation*}
            \Delta l(x_k,\tauTrue_k,\nabla f_k,\dTrue_k) - \Delta l(x_k,\bar\tau_k,\bar{g}_k,\bar{d}_k) \leq (\omega_2 + \omega_3)\Delta l(x_k,\bar\tau_k,\bar{g}_k,\bar{d}_k).
        \end{equation*}
        By  $\{\omega_2,\omega_3\}\subset\R{}_{>0}$, $\tfrac{\Delta l(x_k,\tauTrue_k,\nabla f_k,\dTrue_k)}{1+\omega_2 + \omega_3} \leq \Delta l(x_k,\bar\tau_{k},\bar{g}_k,\bar{d}_k).$ 
        By the fact that iteration $k$ is \textbf{\emph{successful}} and Definition~\ref{def.successful}, it follows that
        \begin{align*}
            \bar\phi(x_k^+,\bar\tau_k;\xi_k^+) - \bar\phi(x_k,\bar\tau_k;\xi_k) &\leq -\alpha_k\theta\Delta l(x_k,\bar\tau_{k},\bar{g}_k,\bar{d}_k) + 2\bar\tau_{k}\epsilon_f \\
            &\leq -\alpha_k\theta\tfrac{\Delta l(x_k,\tauTrue_k,\nabla f_k,\dTrue_k)}{1+\omega_2 + \omega_3} + 2\bar\tau_{-1}\epsilon_f .
        \end{align*}
        Hence, it follows that
        \begin{equation}\label{eq.parta}
        \begin{aligned}
        &Z_{k+1} - Z_k \\
        = \ &\phi(x_{k+1},\bar\tau_{k+1}) - \phi(x_k,\bar\tau_k) - \bar\tau_{k+1}f_{\inf} + \bar\tau_kf_{\inf} \\
        \leq \ &\phi(x_{k+1},\bar\tau_{k+1}) - \bar\phi(x_k,\bar\tau_k;\xi_k) - \bar\tau_{k+1}f_{\inf} + \bar\tau_kf_{\inf} + \bar\tau_k e_k \\
        = \ &\phi(x_{k+1},\bar\tau_{k+1}) - \bar\phi(x_{k+1},\bar\tau_k;\xi_k^+) + \bar\phi(x_{k+1},\bar\tau_k;\xi_k^+) - \bar\phi(x_k,\bar\tau_k;\xi_k) - \bar\tau_{k+1}f_{\inf} + \bar\tau_kf_{\inf} + \bar\tau_k e_k \\
        \leq \ &-\alpha_k\theta\tfrac{\Delta l(x_k,\tauTrue_k,\nabla f_k,\dTrue_k)}{1+\omega_2 + \omega_3} + 2\bar\tau_{-1}\epsilon_f + (\bar\tau_{k+1} - \bar\tau_k)(f(x_{k+1}) - f_{\inf}) + \bar\tau_k (e_k + e_k^+) \\
        \leq \ &-\alpha_k\theta\tfrac{\Delta l(x_k,\tauTrue_k,\nabla f_k,\dTrue_k)}{1+\omega_2 + \omega_3} + 2\bar\tau_{-1}\epsilon_f + \bar\tau_k (e_k + e_k^+).
        \end{aligned}
        \end{equation}
        \textbf{Case B} \ When $\|\bar{g}_k - \nabla f(x_k)\| \leq \epsilon_g$, by $k < T_{\varepsilon_{\Delta l}}$ and Definition~\ref{def.iter_term}, it follows that $\sqrt{\Delta l(x_k,\tauTrue_k,\nabla f_k,\dTrue_k)} > \varepsilon_{\Delta l} > \tfrac{\epsilon_g}{\eta}$. By Lemma~\ref{lem.step_diff_bound}, 
        \begin{equation*}
            \|\bar{d}_k - \dTrue_k\| \leq \zeta^{-1}\|\bar{g}_k - \nabla f_k\| \leq \zeta^{-1}\epsilon_g < \zeta^{-1}\eta\sqrt{\Delta l(x_k,\tauTrue_k,\nabla f_k,\dTrue_k)}.
        \end{equation*}
        \textbf{Case B.1} \ Similarly to \textbf{Case A.1}, if $\nabla f_k^T\dTrue_k \leq 0$, by the fact that $\bar\tau_k \geq \tauTrue_k$, the triangle inequality, \eqref{eq.model_reduction} and Lemma~\ref{lem.gd_dHd_diff_case_a_2}, it follows that
        \begin{equation}\label{eq.case_b1}
        \begin{aligned}
            &\Delta l(x_k,\tauTrue_k,\nabla f_k,\dTrue_k) - \Delta l(x_k,\bar\tau_k,\bar{g}_k,\bar{d}_k) \\
            \leq \ & \bar\tau_k|\bar{g}_k^T\bar{d}_k - \nabla f_k^T\dTrue_k| \leq \bar\tau_k\left(\zeta^{-1}\epsilon_g^2 + \tfrac{(1+\kappa_H\zeta^{-1})\epsilon_g}{\sqrt{\kappa_l\tauTrue_k}}\sqrt{\Delta l(x_k,\tauTrue_k,\nabla f_k,\dTrue_k)}\right) \\
            \leq \ &\bar\tau_k\left(\tfrac{\eta}{\zeta} + \tfrac{1+\kappa_H\zeta^{-1}}{\sqrt{\kappa_l\tauTrue_k}}\right) \eta\Delta l(x_k,\tauTrue_k,\nabla f_k,\dTrue_k).
        \end{aligned}
        \end{equation}
        \textbf{Case B.2} \ If $\nabla f_k^T\dTrue_k > 0$, using previous arguments, \eqref{eq.model_reduction} and Lemma~\ref{lem.gd_dHd_diff_case_a_2},
        \begin{equation}\label{eq.case_b2}
        \begin{aligned}
            &\Delta l(x_k,\tauTrue_k,\nabla f_k,\dTrue_k) - \Delta l(x_k,\bar\tau_k,\bar{g}_k,\bar{d}_k) \\
            \leq \ & |(\bar\tau_k - \tauTrue_k)\nabla f_k^T\dTrue_k| + \bar\tau_k|\bar{g}_k^T\bar{d}_k - \nabla f_k^T\dTrue_k| \\
            \leq \ &|(\bar\tau_k - \tauTrue_k)\nabla f_k^T\dTrue_k| + \bar\tau_k\left(\tfrac{\epsilon_g^2}{\zeta} + \tfrac{(1+\kappa_H\zeta^{-1})\epsilon_g}{\sqrt{\kappa_l\tauTrue_k}}\sqrt{\Delta l(x_k,\tauTrue_k,\nabla f_k,\dTrue_k)}\right) \\
            \leq \ &|(\bar\tau_k - \tauTrue_k)\nabla f_k^T\dTrue_k| + \bar\tau_k\left(\tfrac{\eta}{\zeta} + \tfrac{1+\kappa_H\zeta^{-1}}{\sqrt{\kappa_l\tauTrue_k}}\right) \eta\Delta l(x_k,\tauTrue_k,\nabla f_k,\dTrue_k).
        \end{aligned}
        \end{equation}
        We proceed to bound the term $|(\bar\tau_k - \tauTrue_k)\nabla f_k^T\dTrue_k|$. \\
        \textbf{Case B.2.1} \ If $\tauTrue_k = \bar\tau_k$, then $|(\bar\tau_k - \tauTrue_k)\nabla f_k^T\dTrue_k| = 0$. \\
        \textbf{Case B.2.2} \ If $\tauTrue_k = (1-\epsilon_{\tau})\bar\tau_{k}$, then by Lemmas~\ref{lem.|gd|_2} and \ref{lem.gd_dHd_diff_case_a_2} and Assumption~\ref{ass.main_accuracy},
        \begin{align*}
            &|(\bar\tau_k - \tauTrue_k)\nabla f_k^T\dTrue_k| \\
            = \ &  \epsilon_{\tau}\bar\tau_k|\nabla f_k^T\dTrue_k| \leq \epsilon_{\tau}\bar\tau_k\left(|\bar{g}_k^T\bar{d}_k| + |\nabla f_k^T\dTrue_k - \bar{g}_k^T\bar{d}_k|\right) \\
            \leq \ &\epsilon_{\tau}\omega_2\Delta l(x_k,\bar\tau_k,\bar{g}_k,\bar{d}_k) + \tfrac{\epsilon_{\tau}\bar\tau_k(1+\kappa_H\zeta^{-1})^2}{4}\epsilon_g^2 \\
            & + \epsilon_{\tau}\bar\tau_k\left( \tfrac{\epsilon_g^2}{\zeta} + \tfrac{(1+\kappa_H\zeta^{-1})\epsilon_g}{\sqrt{\kappa_l\tauTrue_k}}\sqrt{\Delta l(x_k,\tauTrue_k,\nabla f_k,\dTrue_k)} \right) \\
            \leq \ &\epsilon_{\tau}\omega_2\Delta l(x_k,\bar\tau_k,\bar{g}_k,\bar{d}_k) + \epsilon_{\tau}\bar\tau_k\eta\left( \tfrac{(1+\kappa_H\zeta^{-1})^2\eta}{4} + \tfrac{\eta}{\zeta} + \tfrac{1+\kappa_H\zeta^{-1}}{\sqrt{\kappa_l\tauTrue_k}} \right)\Delta l(x_k,\tauTrue_k,\nabla f_k,\dTrue_k).
        \end{align*}
        \textbf{Case B.2.3} \ If $\bar\tau_{k} > \tauTrue_k = \tfrac{(1-\sigma)\|c_k\|_1}{\nabla f_k^T\dTrue_k + \max\left\{\dTrue_k^TH_k\dTrue_k,0\right\}}$, following the same logic as in \textbf{Case A.2.3}, by Lemma~\ref{lem.gd_dHd_diff_case_a_2},  \eqref{eq.tau_diff_bound} and \eqref{eq.gd/c_bound},
        \begin{align*}
            & |(\bar\tau_k - \tauTrue_k)\nabla f_k^T\dTrue_k| \\ 
            \leq \ &\tfrac{\bar\tau_k^2}{(1-\sigma)\tauTrue_k} \left(|\nabla f_k^T\dTrue_k - \bar{g}_k^T\bar{d}_k| + |\dTrue_k^TH_k\dTrue_k - \bar{d}_k^TH_k\bar{d}_k|\right)\\
            \leq \ &\tfrac{\bar\tau_k^2}{(1-\sigma)\tauTrue_{\min}} \left(\tfrac{(1+\kappa_H\zeta^{-1})\eta}{\zeta} + \tfrac{1+3\kappa_H\zeta^{-1}}{\sqrt{\kappa_l\tauTrue_k}} \right)\eta\Delta l(x_k,\tauTrue_k,\nabla f_k,\dTrue_k).
        \end{align*}
        Combining \textbf{Cases B.2.1--B.2.3}, and by \eqref{eq.case_b1}, \eqref{eq.case_b2}, and the definitions of $\{\omega_1,\omega_4,\omega_5\} \subset \mathbb{R}_{>0}$  (Assumption~\ref{ass.main_accuracy} and Table~\ref{tbl:constants}), 
        it follows that
      \begin{equation}\label{eq.case_b_bound}
        \begin{aligned}
            &\Delta l(x_k,\tauTrue_k,\nabla f_k,\dTrue_k) - \Delta l(x_k,\bar\tau_k,\bar{g}_k,\bar{d}_k) \\
            \leq \ &\max\left\{\epsilon_{\tau}\omega_1\Delta l(x_k,\bar\tau_k,\bar{g}_k,\bar{d}_k) + \eta\omega_4\Delta l(x_k,\tauTrue_k,\nabla f_k,\dTrue_k), \eta\omega_5\Delta l(x_k,\tauTrue_k,\nabla f_k,\dTrue_k)\right\},
        \end{aligned}
        \end{equation}
        where $\{\omega_1,\omega_4,\omega_5\} \subset \mathbb{R}_{>0}$ are defined in Assumption~\ref{ass.main_accuracy}.
        Thus, it follows,
\begin{equation}\label{eq.case_b_Deltal_deter_stoch_bound}
            \Delta l(x_k,\bar\tau_{k},\bar{g}_k,\bar{d}_k) \geq \min\left\{\tfrac{1-\eta\omega_4}{1+\epsilon_{\tau}\omega_1}, 1-\eta\omega_5\right\} \Delta l(x_k,\tauTrue_k,\nabla f_k,\dTrue_k).
        \end{equation}
        By selecting $\eta$ following Assumption~\ref{ass.main_accuracy}, using the fact that iteration $k$ is \textbf{\emph{successful}} and Definition~\ref{def.successful},
        \begin{align*}
            & \bar\phi(x_k^+,\bar\tau_k;\xi_k^+) - \bar\phi(x_k,\bar\tau_k;\xi_k) \\
            \leq \ &-\alpha_k\theta\Delta l(x_k,\bar\tau_{k},\bar{g}_k,\bar{d}_k) + 2\bar\tau_{k}\epsilon_f  \\
            \leq \ &-\alpha_k\theta\min\left\{\tfrac{1-\eta\omega_4}{1+\epsilon_{\tau}\omega_1}, 1-\eta\omega_5\right\} \Delta l(x_k,\tauTrue_k,\nabla f_k,\dTrue_k) + 2\bar\tau_{-1}\epsilon_f. 
        \end{align*}
        Hence, following similar logic as in~\eqref{eq.parta}, it follows that
        \begin{align*}
            &Z_{k+1} - Z_k \\
        \leq \ &\phi(x_{k+1},\bar\tau_{k+1}) - \bar\phi(x_{k+1},\bar\tau_k;\xi_k^+) + \bar\phi(x_{k+1},\bar\tau_k;\xi_k^+) - \bar\phi(x_k,\bar\tau_k;\xi_k) \\
        &- \bar\tau_{k+1}f_{\inf} + \bar\tau_kf_{\inf} + \bar\tau_k e_k \\
        \leq \ &-\alpha_k\theta\min\left\{\tfrac{1-\eta\omega_4}{1+\epsilon_{\tau}\omega_1}, 1-\eta\omega_5\right\} \Delta l(x_k,\tauTrue_k,\nabla f_k,\dTrue_k) + 2\bar\tau_{-1}\epsilon_f  \\
        &+ (\bar\tau_{k+1} - \bar\tau_k)(f(x_{k+1}) - f_{\inf}) + \bar\tau_k (e_k + e_k^+) \\
        \leq\ &-\alpha_k\theta\min\left\{\tfrac{1-\eta\omega_4}{1+\epsilon_{\tau}\omega_1}, 1-\eta\omega_5\right\} \Delta l(x_k,\tauTrue_k,\nabla f_k,\dTrue_k) + 2\bar\tau_{-1}\epsilon_f  + \bar\tau_k (e_k + e_k^+).
        \end{align*}
        Combining the  results for \textbf{Case A} and \textbf{Case B}, together with the assumption that the iteration is \emph{\textbf{true}}, it follows that
        \begin{align*}
        Z_{k+1} - Z_k \leq \ &-\alpha_k\theta\min\left\{\tfrac{1-\eta\omega_4}{1+\epsilon_{\tau}\omega_1}, 1-\eta\omega_5, \tfrac{1}{1+\omega_2+\omega_3}\right\} \Delta l(x_k,\tauTrue_k,\nabla f_k,\dTrue_k) \\
        &+ 2\bar\tau_{-1}\epsilon_f + \bar\tau_{-1}(e_k + e_k^+) \\
        \leq \ &-h(\alpha_k) + 4\bar\tau_{-1}\epsilon_f ,
        \end{align*}
        where the last inequality is from the conditions $\Delta l(x_k,\tauTrue_k,\nabla f_k,\dTrue_k) > \varepsilon_{\Delta l}^2$ and $e_k + e_k^+ \leq 2\epsilon_f$.
\end{proof}
\begin{proof}[Proof of (iv)]
We first show that for any $k\in\N{}$, if $\alpha_k\leq \tilde\alpha$ and iteration $k$ is \emph{\textbf{true}}, then
    \begin{equation*}
        \phi(x_k + \alpha \bar{d}_k,\bar\tau_k) \leq \phi(x_k,\bar\tau_k) - \alpha_k\theta\Delta l(x_k,\bar\tau_k,\bar{g}_k,\bar{d}_k).
    \end{equation*}
    Since iteration $k$ is \emph{\textbf{true}}, by Definition~\ref{def.true}, we again consider two cases separately: $\|\bar{g}_k - \nabla f_k\| \leq \kappa_{\mathrm{FO}}\alpha_k\sqrt{\Delta l(x_k,\bar\tau_k,\bar{g}_k,\bar{d}_k)}$ and $\|\bar{g}_k - \nabla f_k\| \leq \epsilon_g$. \\
    \textbf{Case A} \ When $\|\bar{g}_k - \nabla f_k\| \leq \kappa_{\mathrm{FO}}\alpha_k\sqrt{\Delta l(x_k,\bar\tau_k,\bar{g}_k,\bar{d}_k)}$, by $\alpha_k \leq \tilde\alpha$, the definition of $\tilde\alpha$,
Assumption~\ref{ass.good_merit_paramter} \change{(the occurrence of event $\mathcal{E}$)} and Lemmas~\ref{lem.Deltal_lb} and~\ref{lem.suff_decrease_exact},
    \begin{align*}
        &\ \phi(x_k + \alpha_k\bar{d}_k,\bar\tau_k) - \phi(x_k,\bar\tau_k) \\
        \leq \ &-\alpha_k\Delta l(x_k,\bar\tau_k,\bar{g}_k,\bar{d}_k) + \alpha_k\bar\tau_k(\nabla f_k - \bar{g}_k)^T\bar{d}_k + \tfrac{\bar\tau_kL+\Gamma}{2}\alpha_k^2\|\bar{d}_k\|^2 \\
        \leq \ &-\alpha_k\Delta l(x_k,\bar\tau_k,\bar{g}_k,\bar{d}_k) + \alpha_k\bar\tau_k\|\nabla f_k - \bar{g}_k\|\|\bar{d}_k\| + \tfrac{\bar\tau_kL+\Gamma}{2}\alpha_k^2\|\bar{d}_k\|^2 \\
        \leq \ &-\left(1 - \left(\sqrt{\tfrac{\bar\tau_{-1}}{\kappa_l}}\kappa_{\mathrm{FO}} + \tfrac{L}{2\kappa_l} + \tfrac{\Gamma}{2\bar\tau_{\min}\kappa_l} \right)\tilde\alpha\right)\alpha_k\Delta l(x_k,\bar\tau_k,\bar{g}_k,\bar{d}_k) \\
        \leq \ &-\alpha_k\theta\Delta l(x_k,\bar\tau_k,\bar{g}_k,\bar{d}_k).
    \end{align*}
    \textbf{Case B} \ When $\|\bar{g}_k - \nabla f_k\| \leq \epsilon_g$ and iteration $k$ is true, \eqref{eq.case_b_Deltal_deter_stoch_bound} holds. Moreover, by the condition that $k < T_{\varepsilon_{\Delta l}}$ and Definition~\ref{def.iter_term}, it follows that
    \begin{align*}
        \|\bar{g}_k - \nabla f_k\| \leq \epsilon_g < \eta\varepsilon_{\Delta l} < \eta\sqrt{\Delta l(x_k,\tauTrue_k,\nabla f_k,\dTrue_k)}.
    \end{align*}
    Therefore, by $\alpha_k \leq \tilde\alpha$, Assumption~\ref{ass.good_merit_paramter} \change{(the occurrence of event $\mathcal{E}$)}, \eqref{eq.case_b_Deltal_deter_stoch_bound} and Lemmas~\ref{lem.Deltal_lb} and~\ref{lem.suff_decrease_exact}, 
    \begin{align*}
        &\phi(x_k + \alpha_k\bar{d}_k,\bar\tau_k) - \phi(x_k,\bar\tau_k) \\
        \leq \ &-\alpha_k\Delta l(x_k,\bar\tau_k,\bar{g}_k,\bar{d}_k) + \alpha_k\bar\tau_k\|\nabla f_k - \bar{g}_k\|\|\bar{d}_k\| + \tfrac{\bar\tau_kL+\Gamma}{2}\alpha_k^2\|\bar{d}_k\|^2 \\
        \leq \ &-\alpha_k\left(\left(1 - \eta\sqrt{\tfrac{\bar\tau_{-1}}{\kappa_l}}\max\left\{\sqrt{\tfrac{1+\epsilon_{\tau}\omega_1}{1-\eta\omega_4}},\tfrac{1}{\sqrt{1-\eta\omega_5}}\right\}\right) - \tfrac{\bar\tau_{\min}L+\Gamma}{2\bar\tau_{\min}\kappa_l}\tilde\alpha\right)\Delta l(x_k,\bar\tau_k,\bar{g}_k,\bar{d}_k) \\
        \leq \ &-\alpha_k\theta\Delta l(x_k,\bar\tau_k,\bar{g}_k,\bar{d}_k).
    \end{align*}
    Combining \textbf{Cases A} and \textbf{B}, and that the iteration is \emph{\textbf{true}}, we conclude the proof of $(iv)$
    \begin{align*}
        \bar\phi(x_k + \alpha_k\bar{d}_k,\bar\tau_k;\xi_k^+) - \bar\phi(x_k,\bar\tau_k;\xi_k) &\leq -\alpha_k\theta\Delta l(x_k,\bar\tau_k,\bar{g}_k,\bar{d}_k) + \bar\tau_k e_k + \bar \tau_k e_k^+ \\
        &\leq -\alpha_k\theta\Delta l(x_k,\bar\tau_k,\bar{g}_k,\bar{d}_k) + 2\bar\tau_{k}\epsilon_f.
    \end{align*}
\end{proof}
\begin{proof}[Proof of (v)]
If iteration $k$ is \emph{\textbf{unsuccessful}}, then by definition $Z_{k+1} = Z_k$, so the inequality holds trivially. 
    Otherwise, starting with the second equation in~\eqref{eq.parta}
    \begin{align*}
        &Z_{k+1} - Z_k \\
        \leq \ &\phi(x_{k+1},\bar\tau_{k+1}) - \bar\phi(x_{k+1},\bar\tau_k;\xi_k^+) + \bar\phi(x_{k+1},\bar\tau_k;\xi_k^+) - \bar\phi(x_k,\bar\tau_k;\xi_k) - \bar\tau_{k+1}f_{\inf} + \bar\tau_kf_{\inf} + \bar\tau_k e_k \\
        \leq \ &-\alpha_k\theta\Delta l(x_k,\bar\tau_k,\bar{g}_k,\bar{d}_k)  + (\bar\tau_{k+1} - \bar\tau_k)(f(x_{k+1}) - f_{\inf}) + 2\bar\tau_{k}\epsilon_f + \bar\tau_k (e_k + e_k^+) \\
        \leq \ &2\bar\tau_{-1}\epsilon_f + \bar\tau_{-1} (e_k + e_k^+).
        \end{align*}
\end{proof}

%% file: StochSQP.bbl
\begin{thebibliography}{10}

\bibitem{BandScheVice14}
{\sc A.~S. Bandeira, K.~Scheinberg, and L.~N. Vicente}, {\em Convergence of
  trust-region methods based on probabilistic models}, SIAM J. Optim., 24
  (2014), pp.~1238--1264.

\bibitem{bellavia2022linesearch}
{\sc S.~Bellavia, E.~Fabrizi, and B.~Morini}, {\em {Linesearch Newton-CG
  methods for convex optimization with noise}}, Annali dell'Universita' di
  Ferrara, 68 (2022), pp.~483--504.

\bibitem{BeraBollZhou22}
{\sc A.~S. Berahas, R.~Bollapragada, and B.~Zhou}, {\em An adaptive sampling
  sequential quadratic programming method for equality constrained stochastic
  optimization}, arXiv preprint arXiv:2206.00712,  (2022).

\bibitem{BeraByrdNoce19}
{\sc A.~S. Berahas, R.~H. Byrd, and J.~Nocedal}, {\em {Derivative-free
  optimization of noisy functions via quasi-Newton methods}}, SIAM J. Optim.,
  29 (2019), pp.~965--993.

\bibitem{BeraCaoSche21}
{\sc A.~S. Berahas, L.~Cao, and K.~Scheinberg}, {\em Global convergence rate
  analysis of a generic line search algorithm with noise}, SIAM J. Optim., 31
  (2021), pp.~1489--1518.

\bibitem{BeraCurtONeiRobi21}
{\sc A.~S. Berahas, F.~E. Curtis, M.~J. O’Neill, and D.~P. Robinson}, {\em A
  stochastic sequential quadratic optimization algorithm for
  nonlinear-equality-constrained optimization with rank-deficient jacobians},
  Mathematics of Operations Research,  (2023).

\bibitem{BeraCurtRobiZhou21}
{\sc A.~S. Berahas, F.~E. Curtis, D.~Robinson, and B.~Zhou}, {\em Sequential
  quadratic optimization for nonlinear equality constrained stochastic
  optimization}, SIAM J. Optim., 31 (2021), pp.~1352--1379.

\bibitem{BeraShiYiZhou22}
{\sc A.~S. Berahas, J.~Shi, Z.~Yi, and B.~Zhou}, {\em Accelerating stochastic
  sequential quadratic programming for equality constrained optimization using
  predictive variance reduction}, Comput. Optim. Appl.,  (2023), pp.~1--38.

\bibitem{Bert98}
{\sc D.~Bertsekas}, {\em Network optimization: continuous and discrete models},
  vol.~8, Athena Scientific, 1998.

\bibitem{BlanCartMeniSche19}
{\sc J.~Blanchet, C.~Cartis, M.~Menickelly, and K.~Scheinberg}, {\em
  Convergence rate analysis of a stochastic trust-region method via
  supermartingales}, INFORMS J. Optim., 1 (2019), pp.~92--119.

\bibitem{byrd2008inexact}
{\sc R.~H. Byrd, F.~E. Curtis, and J.~Nocedal}, {\em {An inexact SQP method for
  equality constrained optimization}}, SIAM J. Optim., 19 (2008), pp.~351--369.

\bibitem{cao2022first}
{\sc L.~Cao, A.~S. Berahas, and K.~Scheinberg}, {\em First-and second-order
  high probability complexity bounds for trust-region methods with noisy
  oracles}, Math. Program.,  (2023), pp.~1--52.

\bibitem{CartSche18}
{\sc C.~Cartis and K.~Scheinberg}, {\em Global convergence rate analysis of
  unconstrained optimization methods based on probabilistic models}, Math.
  Program., 169 (2018), pp.~337--375.

\bibitem{ChenTungVeduMori18}
{\sc C.~Chen, F.~Tung, N.~Vedula, and G.~Mori}, {\em Constraint-aware deep
  neural network compression}, in Proceedings of the ECCV, 2018, pp.~400--415.

\bibitem{ChenMeniSche18}
{\sc R.~Chen, M.~Menickelly, and K.~Scheinberg}, {\em Stochastic optimization
  using a trust-region method and random models}, Math. Program., 169 (2018),
  pp.~447--487.

\bibitem{CurtONeiRobi21}
{\sc F.~E. Curtis, M.~J. O’Neill, and D.~P. Robinson}, {\em Worst-case
  complexity of an {SQP} method for nonlinear equality constrained stochastic
  optimization}, Math. Program.,  (2023), pp.~1--53.

\bibitem{CurtRobiZhou21}
{\sc F.~E. Curtis, D.~P. Robinson, and B.~Zhou}, {\em A stochastic inexact
  sequential quadratic optimization algorithm for nonlinear
  equality-constrained optimization}, INFORMS Journal on Optimization,  (2024).

\bibitem{gould2015cutest}
{\sc N.~I. Gould, D.~Orban, and P.~L. Toint}, {\em {CUTEst: a constrained and
  unconstrained testing environment with safe threads for mathematical
  optimization}}, Comput. Optim. Appl., 60 (2015), pp.~545--557.

\bibitem{GratRoyeViceZhan15}
{\sc S.~Gratton, C.~W. Royer, L.~N. Vicente, and Z.~Zhang}, {\em Direct search
  based on probabilistic descent}, SIAM J. Optim., 25 (2015), pp.~1515--1541.

\bibitem{GratRoyeViceZhan18}
{\sc S.~Gratton, C.~W. Royer, L.~N. Vicente, and Z.~Zhang}, {\em Complexity and
  global rates of trust-region methods based on probabilistic models}, IMA J.
  Numer. Anal., 38 (2018), pp.~1579--1597.

\bibitem{hamming2012introduction}
{\sc R.~W. Hamming}, {\em Introduction to applied numerical analysis}, Courier
  Corporation, 2012.

\bibitem{HazaLuo16}
{\sc E.~Hazan and H.~Luo}, {\em Variance-reduced and projection-free stochastic
  optimization}, in International Conference on Machine Learning, PMLR, 2016,
  pp.~1263--1271.

\bibitem{jin2022high}
{\sc B.~Jin, K.~Scheinberg, and M.~Xie}, {\em High probability complexity
  bounds for adaptive step search based on stochastic oracles}, SIAM Journal on
  Optimization, 34 (2024), pp.~2411--2439.

\bibitem{KushClar12}
{\sc H.~J. Kushner and D.~S. Clark}, {\em Stochastic approximation methods for
  constrained and unconstrained systems}, vol.~26, Springer Science \& Business
  Media, 2012.

\bibitem{Lan20}
{\sc G.~Lan}, {\em First-order and stochastic optimization methods for machine
  learning}, Springer, 2020.

\bibitem{LuFreu21}
{\sc H.~Lu and R.~M. Freund}, {\em Generalized stochastic {Frank--Wolfe}
  algorithm with stochastic “substitute” gradient for structured convex
  optimization}, Math. Program., 187 (2021), pp.~317--349.

\bibitem{menickelly2023stochastic}
{\sc M.~Menickelly, S.~M. Wild, and M.~Xie}, {\em A stochastic quasi-newton
  method in the absence of common random numbers}, arXiv preprint
  arXiv:2302.09128,  (2023).

\bibitem{more2009benchmarking}
{\sc J.~J. Mor{\'e} and S.~M. Wild}, {\em Benchmarking derivative-free
  optimization algorithms}, SIAM J. Optim., 20 (2009), pp.~172--191.

\bibitem{NaAnitKola22}
{\sc S.~Na, M.~Anitescu, and M.~Kolar}, {\em An adaptive stochastic sequential
  quadratic programming with differentiable exact augmented lagrangians}, Math.
  Program.,  (2022), pp.~1--71.

\bibitem{na2021inequality}
{\sc S.~Na, M.~Anitescu, and M.~Kolar}, {\em Inequality constrained stochastic
  nonlinear optimization via active-set sequential quadratic programming},
  Math. Program.,  (2023), pp.~1--75.

\bibitem{NaMaho22}
{\sc S.~Na and M.~W. Mahoney}, {\em Asymptotic convergence rate and statistical
  inference for stochastic sequential quadratic programming}, arXiv preprint
  arXiv:2205.13687,  (2022).

\bibitem{NandPathSing19}
{\sc Y.~Nandwani, A.~Pathak, and P.~Singla}, {\em A primal dual formulation for
  deep learning with constraints}, Advances in Neural Information Processing
  Systems, 32 (2019).

\bibitem{NoceWrig06}
{\sc J.~Nocedal and S.~Wright}, {\em Numerical optimization}, Springer Series
  in Operations Research and Financial Engineering, Springer-Verlag New York,
  2006.

\bibitem{OztoByrdNoce21}
{\sc F.~Oztoprak, R.~Byrd, and J.~Nocedal}, {\em Constrained optimization in
  the presence of noise}, SIAM J. Optim., 33 (2023), pp.~2118--2136.

\bibitem{PaquSche20}
{\sc C.~Paquette and K.~Scheinberg}, {\em A stochastic line search method with
  expected complexity analysis}, SIAM J. Optim., 30 (2020), pp.~349--376.

\bibitem{RaviDinhLokhSing19}
{\sc S.~N. Ravi, T.~Dinh, V.~S. Lokhande, and V.~Singh}, {\em Explicitly
  imposing constraints in deep networks via conditional gradients gives
  improved generalization and faster convergence}, in Proceedings of the AAAI
  Conference on Artificial Intelligence, vol.~33, 2019, pp.~4772--4779.

\bibitem{ReesDollWath10}
{\sc T.~Rees, H.~S. Dollar, and A.~J. Wathen}, {\em Optimal solvers for
  {PDE}-constrained optimization}, SIAM J. Sci. Comput., 32 (2010),
  pp.~271--298.

\bibitem{RobeRoye22}
{\sc L.~Roberts and C.~W. Royer}, {\em Direct search based on probabilistic
  descent in reduced spaces}, SIAM J. Optim., 33 (2023), pp.~3057--3082.

\bibitem{RoyMhamHara18}
{\sc S.~K. Roy, Z.~Mhammedi, and M.~Harandi}, {\em Geometry aware constrained
  optimization techniques for deep learning}, in Proceedings of CVPR, 2018,
  pp.~4460--4469.

\bibitem{scheinberg2022stochastic}
{\sc K.~Scheinberg and M.~Xie}, {\em Stochastic adaptive regularization method
  with cubics: A high probability complexity bound}, in 2023 Winter Simulation
  Conference (WSC), IEEE, 2023, pp.~3520--3531.

\bibitem{ShapDentRusz21}
{\sc A.~Shapiro, D.~Dentcheva, and A.~Ruszczynski}, {\em Lectures on stochastic
  programming: modeling and theory}, SIAM, 2021.

\bibitem{sun2023trust}
{\sc S.~Sun and J.~Nocedal}, {\em A trust region method for noisy unconstrained
  optimization}, Mathematical Programming,  (2023), pp.~1--28.

\bibitem{wachter2005line}
{\sc A.~W{\"a}chter and L.~T. Biegler}, {\em Line search filter methods for
  nonlinear programming: Motivation and global convergence}, SIAM J. Optim., 16
  (2005), pp.~1--31.

\end{thebibliography}
